%% file: circle3.tex
\documentclass[12pt, reqno]{amsart}

\usepackage{philstyle}

\input familydefs
\usepackage{amssymb,amsmath,amscd,xspace}
\usepackage{enumerate}
\usepackage{graphicx}

\graphicspath{ {./figures/} }

\title{On the abundance of $k$-fold semi-monotone minimal 
sets in bimodal circle maps}

\author{Philip Boyland}
\address{Department of
    Mathematics\\University of Florida\\372 Little Hall\\Gainesville\\
    FL 32611-8105, USA}
\email{boyland@ufl.edu}

\begin{document}

\begin{abstract}
Inspired by a twist maps theorem of Mather  
we study recurrent invariant sets that are ordered like rigid rotation
 under the action of the lift of a bimodal circle map $g$ to the
$k$-fold cover. For each irrational in the interior of the rotation
set the collection of the $k$-fold ordered semi-Denjoy minimal sets 
with that rotation number
contains a $(k-1)$-dimensional ball in the weak topology on their
unique invariant measures. We also describe completely their periodic
orbit analogs for rational rotation numbers. The main tool is 
a generalization of a construction of Hedlund and Morse which generates
the symbolic analogs of these $k$-fold well-ordered invariant sets.
\end{abstract}

\maketitle

\section{introduction}

In a dynamical system a rotation number or vector measures the asymptotic
speed and direction of an orbit. The rotation set collects all these
together into a single invariant of the system. The natural question is 
how much does this invariant tell you about the dynamics? Perhaps the
first issue is whether for each rotation number is there a nice
invariant set in which every point has that rotation number?

This question has been studied in a number of contexts with the
most complete answer known about maps of the circle, annulus and 
two-dimensional torus. In these cases the basic question is enhanced
by requiring that the invariant set of a given rotation vector
has the same combinatorics as a rigid rotation. So, for example, for 
a continuous
degree-one map $g$ of the circle and a number $\omega$ in its rotation set
is there an invariant set $Z_\omega$ on which the action of $g$
looks like the invariant set of rigid rotation
of the circle by $\omega$? This is made more precise and clear
by lifting the dynamics to the universal cover $\R$. 
The question then translates to whether the action of the lift
$\tg:\R\raw\R$ on the lift $\tZ_\omega$ is order-preserving?
For this class of maps the answer is yes;
 such invariant sets always exist (\cite{CGT}).

On the torus and annulus a general homeomorphism isotopic to the identity 
 lacks the structure to force the
desired invariant sets to be order preserving so topological
analogs are used (\cite{LC, Bdmon, Parwani}).
The required additional structure available in the annulus case is the
monotone twist hypothesis. In this case the celebrated Aubry-Mather
Theorem states that for each rational in the rotation set there
is a periodic orbit and for each irrational a Denjoy minimal set
and the action of the map on these invariant sets is ordered
in the circle factor like rigid rotation. These invariant sets are 
now called \textit{Aubry-Mather sets}.

For an area-preserving monotone twist map the minimal set
with a given irrational rotation number could be an invariant
circle. When a parameter is altered and this circle breaks it
is replaced by an invariant Denjoy minimal set. 
In \cite{mather} Mather investigated what additional dynamics this forces.
He showed that in the absence of an invariant circle  
with a given irrational rotation number there are many other 
invariant minimal Cantor sets with the same rotation number and the dynamics
on these sets is nicely ordered under the dynamics not in the base, 
but rather in finite covering spaces of the annulus.

More specifically, these invariant minimal sets are Denjoy minimal sets
which are uniquely ergodic. Their collection is  topologized 
using the weak topology on these measures. Mather showed that for
a given irrational rotation number in the rotation set the collection
of Denjoy minimal sets with that rotation number which are ordered
in the $k$-fold cover contains a topological disk of dimension
$k-1$. In this paper we prove the analog of this result for a class $\cG$ 
of  bimodal degree-one maps of the circle as well as describing their periodic
orbits which have nicely ordered lifts in the $k$-fold cover.

 Mather's proof use variational methods. The main methods here come 
from symbolic dynamics and utilize a construction that 
is a generalization of 
one due to Hedlund and Morse (\cite{HM1} and \cite{GH} page 111). 
Such generalizations are a common tool in topological dynamics
(\cite{MP}, \cite{A}  pages 234-241, and \cite{BdDenjoy}).
This HM construction used here for a rotation number $\omega$ and number $k$ 
generates the itineraries under rigid rotation by $\omega$
with respect to an address systems made from $2k$ intervals on the circle.  
The closure of this set of itineraries yields the symbolic analog
of an invariant set that is nicely ordered in the $k$-fold cover.
These sets are termed \textit{symbolic $k$-fold semi-monotone sets
(symbolic kfsm sets)}. Varying the address system parameterizes the symbolic kfsm sets
in both the Hausdorff and weak topologies.

A \textit{physical kfsm set} (or just a kfsm set if the context is clear)
is a $g$-invariant set $Z$  which has a lift $Z'$ to the $k$-fold cover
of the circle $S_k$ on which the lift $g_k$ of $g$ acts like rigid rotation.
Physical and symbolic kfsm sets are connected by a second main tool.

The second tool again uses addresses and itineraries, but this
time to code orbits under the bimodal map $g$. Restricting to all
orbits that land in the positive slope region we get an invariant
set $\Lambda(g)$ which is coded by an order interval in the one-sided
two shift $\Sigma_2^+$. Since we need to study invariant sets that
are ordered in the $k$-fold cover we lift this
coding to one on the orbits which stay in the positive slope region
under $g_k$ in the $k$-fold cover $S_k$. This yields a $g_k$-invariant set 
$\Lambda_k(g)$ which is then coded by a subshift 
$\hLambda_k(g)\subset\Sigma_{2k}^+$. 

This result
connects the physical kfsm sets in $\Lambda_k(g)$, the symbolic
kfsm sets in $\hLambda_k(g)$,  and the symbolic sets generated by the 
HM construction. Part (c) will be explained below.
\begin{theorem}\label{maintfae}  For $g\in\cG$ the following are
equivalent:
\begin{enumerate}[(a)]
\item $Z\subset \Lambda(g)$ is a recurrent kfsm set for $g$.
\item The symbolic coding of $Z$ via the itinerary
map is constructable via the HM process.
\item  $Z$ is a recurrent set of an interpolated semi-monotone
map $H_{\vc}$ in the $k$-fold cover.
\end{enumerate}
\end{theorem} 
Note that the result is restricted to recurrent kfsm sets. There are 
several reasons for this. First, recurrence is where the interesting
dynamics occurs, second, invariant measures are always supported on
recurrent sets, and finally, the HM construction produces recurrent
sets. As is well-known in Aubry-Mather theory there are also 
nonrecurrent kfsm sets which consist of a recurrent set  and
orbits homoclinic to  that set. We also restrict to orbits that
stay in the positive slope region of $g$. Considering kfsm sets
that also have points in the negative slope region at most adds
additional homoclinic orbits or shadow periodic orbits. See
Section~\ref{neghom}.

For each $k$, the HM construction depends on two parameters,
a rotation number $\omega$ and a parameter $\vnu$ describing
the address system on the circle. For a rational rotation number it produces
a finite cluster of periodic orbits while for irrationals it produces
a semi-Denjoy minimal set. Since $g$ is noninjective the analogs
of Denjoy minimal sets have pairs of points that collapse in
forward time, and hence the ```semi'' in their name. 

Another main result is that the HM construction
parameters $(\omega, \vnu)$  yield a  homeomorphic parameterization of 
the space of invariant measures on the recurrent symbolic kfsm sets with
the weak topology. Via the itinerary map, this is pulled back to
a parameterization of  the space of invariant measures on the 
physical recurrent kfsm sets with the weak topology. It yields the
following result in which $\rho(g)$ is the rotation interval of $g\in\cG$. 
\begin{theorem}\label{main}
Assume $g\in\cG$, $\alpha\not\in\Q$, $\alpha\in\Intt(\rho(g))$, and $k>0$.
\begin{enumerate}[(a)]
\item  In the weak topology there is $(k-1)$-dimensional disk of kfsm
 semi-Denjoy minimal sets with rotation number $\alpha$. 
\item If $p_n/q_n$ is a sequence of rationals in lowest terms with
$p_n/q_n \raw \alpha$, then the number of distinct kfsm
periodic orbits of $g$ with  rotation number $p_n/q_n$ grows like $q_n^{k-1}$.
\end{enumerate}
\end{theorem}
Informally, a kfsm semi-Denjoy minimal set  wraps
$k$-times around the circle with orbits moving at different average
speeds in each loop. Lifting to the $k$-fold cover these ``speeds''
are given by the amount of the unique invariant measure present
in a fundamental domain of $S_k$: more measure means slower speed.
The $k$-dimensional vector of these measures is called the \textit{skewness}
of the minimal set. The sum of the components of the skewness is one
and thus the collection of possible skewnesses contains
a $(k-1)$-dimensional ball. The skewness turns out to be an injective 
parameterization of the kfsm sets for a given irrational rotation
number in the interior of the rotation set of a $g\in \cG$ 
(see Remark~\ref{skewrk}).

The HM-parametrization of kfsm sets with the Hausdorff topology
is only lower semicontinuous. The points of discontinuity 
are given in Theorem~\ref{iotathm}.

%Acts like covering space, and inport of order structure on 
%symbolics. maybe covering sapce theor for subshifts
%crucial is order on circle, which doesn't make
%sense, so really order on universal cover, so that
%will be important. Will complement amtters, but
%clarifies the order stuff greatly.
%woek in three spaces, universal cover, cause there have linear
%order necessary to notion of kfsm. The $k$-fold cyclic
%cover where unwraps to semi-monotone and the base where the
%dyanics is. Indicated by subsrcipt , $k=\infty$ means
%universal cover, $k$ is k-fold and $k=1$ is the base.

We will on occasion use results derived from those of
Aubry-Mather theory. While the context here is a bit different
the proofs are virtually identical and so are omitted. There
are excellent expositions of Aubry-Mather theory; see, for example,
\cite{MS},  \cite{HK} Chapter 13, and \cite{gole} Chapter 2.
In the context of the generalization of Aubry-Mather theory to
monotone recursion maps a version of
Mather's theorem on Denjoy minimal sets is given in \cite{moredenjoy1}.

We restrict attention here to a  particular class of bimodal circle
maps defined in Section~\ref{classG}.  Using the Parry-Milnor-Thurston
Theorem for degree-one circle maps the results
can be transferred (with appropriate alterations) to general
bimodal circle maps (see Remark~\ref{generalg}). 

It is worth noting that the results here immediately apply to 
a class of annulus homeomorphisms. This application
can be done either via the Brown-Barge-Martin method
using the inverse limit of $g\in\cG$ (\cite{bbm, usbbm}) 
or via the symbolic dynamics
in annulus maps with good coding like a rotary horseshoe, for example, 
\cite{HH3, levi, BdDenjoy, GM}.

\begin{figure}[htbp]
\begin{center}
\includegraphics[width=0.4\textwidth]{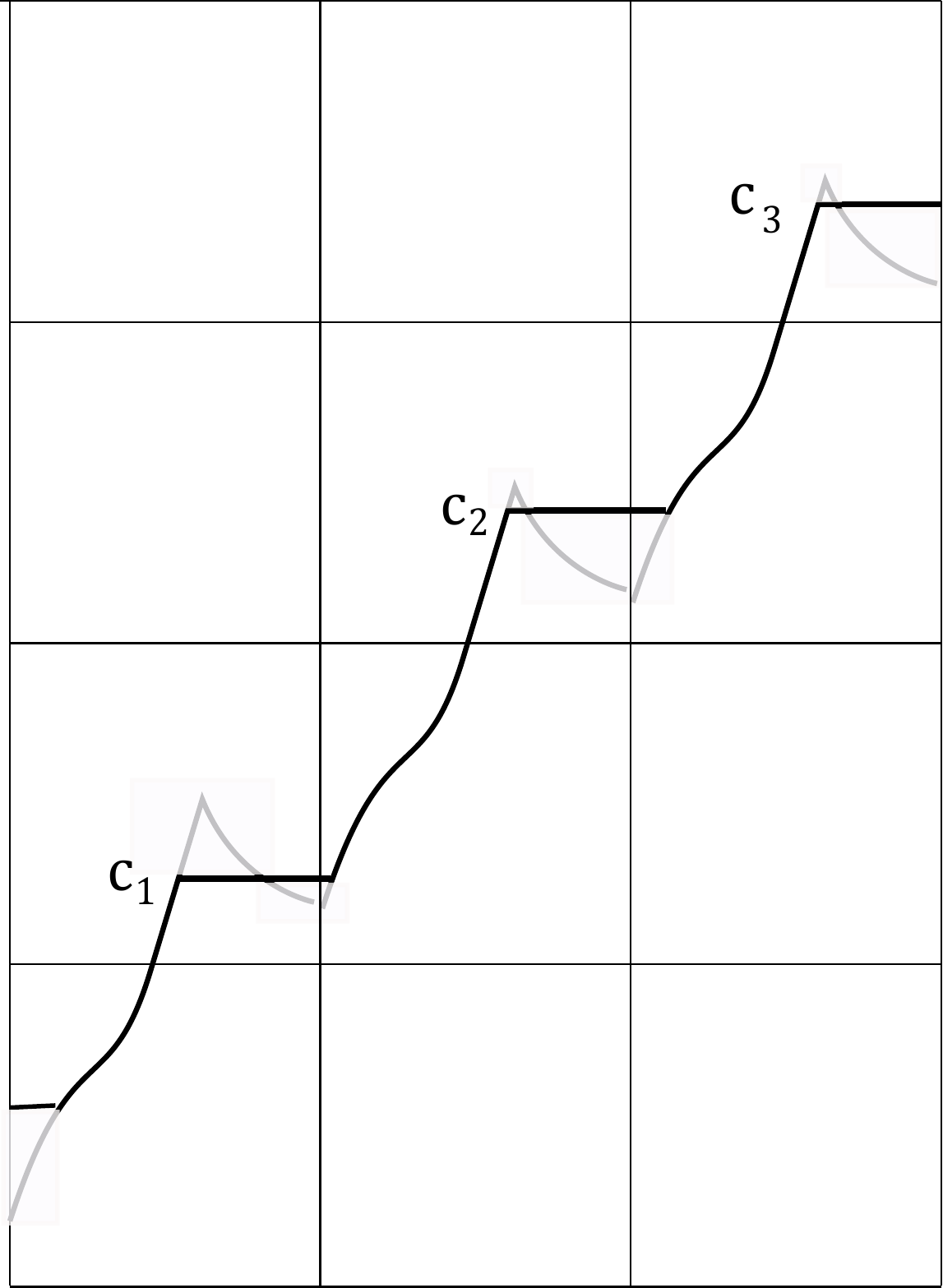}
\end{center}
\caption{The lift of a $g\in\cG$ to the 3-fold cover and an interpolated
semi-monotone map}
\label{figure1}
\end{figure}

Figure~\ref{figure1} illustrates the conceptual framework that inspired the
results here. It shows the graph of a map $g \in \cG$ lifted to
the $3$-fold cover. At three heights $(c_1, c_2, c_3) = \vc$ the 
graph is cut-off yielding a semi-monotone circle map $H_{\vc}$. Such
maps have a unique rotation number and well understood recurrent
sets which are of necessity semi-monotone sets. As $\vc$ is varied, the 
rotation number $\rho(H_{\vc}) = \rho(\vc)$ varies continuously.
Thus one would expect that the level sets $\rho^{-1}(\omega)$
 provide a parameterization of the kfsm sets with rotation
number $\omega$. In particular, for irrational $\omega$ this
level set should be a $(k-1)$-dimensional disk as in Theorem~\ref{main}(a).
This is true for $g\in\cG$.  Figure~\ref{rect5} below shows
some level sets. It is worth noting that this figure is 
not a bifurcation diagram, but rather a detailed analysis of
the dynamics present in a single map.

While providing a valuable heuristic this point of view
is not as technically tractable as the HM construction and
we content ourselves with  just a few comments on it in Section~\ref{hall}.
One of which  is the addition of  item (c)
to the list of equivalent conditions in Theorem~\ref{maintfae}.

The literature on bimodal circle map dynamics is vast and we briefly
mention  only two threads here.
Symbolic dynamics for degree one bimodal goes back at least to \cite{Bernhardt, HH1, HH2}.  The interpolated ``flat spot'' map 
trick for finding  one-fold semi-monotone sets was
discovered and used by many people in the early
80's; references include 
\cite{CGT, Mcirc, V1, V2, V3, Kadanoff, Bdcirc, physics}.
The author learned the trick from G.R. Hall in Spring, 1983 and the
idea of applying it in finite covers emerged in conversations with him. 

There are many  questions raised by this work, but we mention 
just three here. As is well-known, the one-fold symbolic semi-monotone
sets generated by the HM construction are the much-studied Sturmian sequences.
The general symbolic kfsm are thus a generalization of one property of the
Sturmians to more symbols (there are many other generalizations). The Sturmians
have many marvelous properties such as their connection to the Farey
tree and substitutions: which properties
are shared by symbolic kfsm sets? 

The HM construction is an explicit parameterized way of getting
well controlled orbits that do not preserve the cyclic order
in the base and thus in most cases force positive entropy and 
as well as other orbits. A second question is how does the
 parameterization given by
the HM construction interact with the forcing orders
on orbits in dimension one and two (see, for example, \cite{misbook}
and \cite{bdamster}).

A final question relates to the global parameterization of kfsm by the
HM construction. Each bimodal map $\in\cG$ corresponds to a specific
set of parameters, namely, those that generate symbolic kfsm
whose physical counterparts are present in $g$. What is the 
shape of this set of parameters?

After this work was completed the author became aware of the considerable 
literature studying invariant sets in the circle that are 
invariant and nicely ordered under the action of $z\mapsto z^d$
\cite{uk, gold, goldtress, bowman, uab, uabthesis}.
While the exact relationship of that theory and what is contained
in the paper is not clear, it is clear that the two areas share
many basic ideas and methods. These include using a family of
interpolated semi-monotone circle maps with flat spots, tight and loose
gaps in invariant Cantor sets, parametrizing
the sets using the position of the flat spots, and parametrizing
the sets with irrational rotation number by an analog of skewness.
Section~\ref{last} contains a few more comments on the relationship
of the problems.

\section{Preliminaries}
\subsection{Dynamics}
Throughout this section $X$ is a metric space and
$g:X\raw X$ is a continuous, onto map. Since the maps $g$
we will be considering will usually not be injective, we will
be just considering forward orbits, so $o(x,g) = \{x, g(x),
g^2(x), \dots\}$.

A point $x$  is \textit{recurrent} if there exists a sequence
$n_i\raw\infty$ with $g^{n_i}(x)\raw x$. A $g$-invariant set $Z$ is called
\textit{recurrent} if every point $z\in Z$ is recurrent. Note that
\textit{a recurrent} subset is usually different
 than \textit{the recurrent set},
with the latter being the closure of all recurrent points. A compact
 invariant set $Z$ is called \textit{minimal} if every point $z\in Z$ has 
a forward orbit that is dense in $Z$.

%Let $\cC(g)$ be the collection of compact, $g$-invariant
%sets with the Hausdorf topology. The Hausdorff distance between
%two sets is denoted $\HD(Y, Z)$. Let $\cM(g)$ be the collection of
%$g$-invariant, Borel, probability measures with the weak
%topology. When $X$ is compact, both $\cC(g)$ and $\cM(g)$ are also.
%A map $g:X\raw X$ is called \textit{uniquely ergodic}
%if it has a unique invariant measure, so $\cM(g)$ is a single point.

The one-sided shift space on $n$ symbols is
 $\Sigma^+_n = \{0, \dots, n-1\}^\N$
and that on $\Z$ symbols is $\Sigma^+_\Z = \Z^\N$. Occasionally
we will write $\Sigma_\infty^+$ for $\Sigma^+_\Z$.  For clarity we
note that in this paper $0\in\N$. In every case we give one-sided
shift spaces the lexicographic order and the left shift map is denoted
$\sigma$, perhaps with a subscript to indicate the shift space upon
which it acts. Maps between shifts and subshifts here will
always be defined by their action on individual symbols
so, for example, $\varphi:\Sigma_n^+\raw \Sigma_n^+$ defined on
symbols
by $s\mapsto \varphi(s)$ means that $\varphi(s_0 s_1 s_2\dots) 
= \varphi(s_0) \varphi(s_1) \varphi(s_2) \dots$. For a block 
$B = b_0 \dots b_{N-1}$ in $\Sigma^+_n$ its cylinder set is 
$[B] = \{ \us\in\Sigma_n^+\colon s_i = b_i, i = 0, \dots, N-1\}$.
Note that all our cylinder sets start with index $0$

The space-map pairs $(X,f)$ and $(Y,g)$ are said to be
\textit{topologically conjugate} by $h$ if $h$ is a homeomorphism
from $X$ onto $Y$ and $hf = gh$.

We will frequently use the standard dynamical tool of addresses and itineraries.
Assume  $X = X_0\sqcup X_1\sqcup\dots \sqcup X_{n-1}$ with $\sqcup$
denoting disjoint union.
Define the \textit{address map} $A$
as $A(x) = j$ when $x\in X_j$ and the \textit{itinerary map}
$\iota:X\raw \Sigma_n^+$  by
$\iota(x)_i = A(g^i(x))$. It is immediate that $\sigma\circ \iota
= \iota\circ g$.  In many cases here,
$\iota$ will be continuous and injective yielding a topological conjugacy from
$(X,g)$ to a subset of $(\Sigma_n^+, \sigma)$.

We will also encounter the situation where the $X_j$ are not disjoint, but
intersect only in their frontiers  $\Fr(X_j)$. In this case we
define a ``good set'' $G = \{x \colon o(x, g)\cap (\cup \Fr(X_j)) =
\emptyset\}$. In this case the itinerary map is defined as
$\iota:G\raw \Sigma_n^+$.

For $Z\subset X$, its interior, closure and frontier are denoted
by $\Intt(Z), \Cl(Z),$ and $\Fr(X)$, respectively. The $\epsilon$-ball
about $x$ is $N_\epsilon(x)$. The Hausdorff distance between two
sets is denoted $\HD(X, Y)$. For an interval $I$ 
in $\R$, $|I|$ denotes it length, and for a finite set $Z$, $\# Z$ 
is its cardinality. On an ordered $k$-tuple the map $\tau$ is the
left cyclic shift, so $\tau(a_1, a_2, \dots, a_k) = (a_2, a_3, 
\dots, a_k, a_1)$. On the circle $S^1$,  $\theta_1 < \theta_2$ is defined
as long as $|\theta_1-\theta_2| < 1/2$. 

\subsection{The circle, finite covers and degree one circle maps}\label{coverdef}
While the only compact manifold here will be a circle it will clarify
matters to use the language of  covering spaces.

In general, if $\pi:\tY\raw Y$ is a covering space, and $Z\subset Y$,
\textit{a lift} of $Z$ is any set $Z'\subset \tY$ with $\pi(Z') = Z$. 
The \textit{full lift} of $Z$ is $\tZ = \pi^{-1}(Z)$. If
$g:Y\raw Y$ lifts to $\tg:\tY\raw \tY$ and $Z\subset Y$ is
$g$-invariant then the full lift $\tZ$ is $\tg$-invariant, a
property that is usually not shared by a lift $Z'$.

The universal cover of the circle is $\R$ with deck transformation
$T(x) = x+1$ and the covering space is thus
$\pi:\R\raw S^1 = \R/T = \R/\Z = [0,1]/\mysim$. We will
only study maps $g:S^1\raw S^1$ whose lifts $\tg$ commute with the
deck transformation, $\tg T = T \tg$, or $\tg(x+1) = \tg(x) + 1$.
These circle maps are commonly termed \textit{degree one}. Our
given $g$ will usually have a preferred lift $\tg$ and so 
all other lifts are obtained as $T^n \tg$ or $\tg + n$.
 
Central to our study are the finite $k$-fold covers of the circle
for each $k>0$, $S_k = \R/T^k = \R/k\Z = [0,k]/\mysim$. The deck
transformation $T_k:S_k \raw S_k$ is induced by $T$ on $\R$
and the covering space is $\pi_k:S_k\raw S^1$. A preferred
lift $\tg$ of $g$ to $\R$ induces a preferred lift
 $\tg_k:S_k\raw S_k$ that commutes with $T_k$. 
We also need the map from the universal cover
to the $k$-fold cover treating $S_k$ as the base space $p_k:\R\raw S_k$.

A $g$-periodic point $x$ is said to have \textit{rotation type} 
$(p,q)$ with respect to the preferred lift $\tg:\R\raw\R$
if $x$ has period $q$ and for some  lift
$x'\in\R$, $\tg^q(x') = T^p x'$. Note that there is 
no requirement here that $p$ and $q$ are relatively prime.

A central concern in this paper is how $g$-minimal sets in $S^1$
lift to $k$-fold covers.

\begin{theorem}\label{lift}
 Let $g:S^1\raw S^1$ be degree one and  fix  $1 < k< \infty$.
\begin{enumerate}[(a)]
\item If $Z\subset S^1$ is a minimal set, then there exists an
 $m$ which divides $k$ so that the full lift of $Z$ to $S_k$ satisfies
\begin{equation}\label{already}
\tZ = \sqcup_{j=1}^m Z_j'
\end{equation}
with each $Z_j'$ minimal under $\tg_k$, $\pi_k(\tZ_j') = Z$ and
$T_k(Z_j') = Z_{j+1}'$ with indices $\mymod k$.
\item If $Z', Z''\subset S_k$ are $\tg_k$ minimal
sets, we have $\pi_k(Z') = \pi_k(Z'')$ if and only
if $T^p_k(Z') = Z''$ for some $p$.
\item If $x\in S^1$ is a periodic point  with rotation type $(p,q)$
let $m = \gcd(k,p)$. There exist $x_j'\in \pi_k^{-1}(x)\subset S_k$ with
$1 \leq j \leq m$ and
 \begin{equation}\label{already2}
\pi_k^{-1}(o(x,g)) = \sqcup_{j=1}^m o(x_j', \tg_k)
\end{equation}
 the period of each $x_j'$ under
$\tg_k$ equal to $kq/m$, and $T_k(x_j') = x_{j+1}' $ 
with indices $\mymod k$.
\end{enumerate}
\end{theorem}

\begin{proof}
To prove (a) we begin with two preliminary facts with similar proofs. 
First, we show that for
any $z'\in\tZ$, $\pi_k(\otz) = Z$. Let $z = \pi_k(z')$ and
pick $y\in Z$. By minimality
there exists $g^{n_i}(z)\raw y$. Lifting and using
the compactness of $ S_k$ there is a subsequence $n_{i'}$ 
and a $y'\in S_k$ with $\tg_k^{n_{i'}}(z')\raw y'$. Thus
$g_k^{n_{i'}}(z) = \pi_k(\tg_k^{n_{i'}}(z'))\raw \pi_k(y')$ and
so $y = \pi_k(y')$.

Second, we show that for any $z',y'\in\tZ$, there exists an
$p$ with $T^p(y')\in \otz$. Let $z=\pi_k(z')$ and $y = \pi_k(y')$.
By minimality again, we have $g^{n_i}(z)\raw y$. Lifting and passing
to a subsequence, there is a subsequence $n_{i'}$ 
and a $y''\in S_k$ with $\tg_k^{n_{i'}}(z')\raw y''$. Thus
$\pi_k(y'') = y$ also, so there is a $p$ with $y'' = T_k^p(y')$
and so $T_k^p(y')\in \otz$.

Now for the main proof, pick $z'\in\tZ$ and let $Z_1' = \otz$,
so by the first fact, $\pi_k(Z_1') = Z$. We now show $Z_1'$ is minimal
under $\tg_k$. If not, there is a $y'\in Z_1'$ with 
$\oty \kludge \otz$. By the second preliminary fact, there is some
$p$ with
 $$\Cl(o(T_k^p(z'), \tg_k)) \subset \oty \kludge \otz.$$
Acting by the homeomorphism $T_k^p$ and iterating the strict inclusions
we have 
$$
\otz = \Cl(o(T_k^{pk}(z'), \tg_k)) \kludge 
\Cl(o(T_k^{p(k-1)}(z'), \tg_k)) \kludge \dots
\kludge \Cl(o(T_k^p(z'), \tg_k)) \kludge \otz
$$
a contradiction, so $\tg_k$ acting on $Z_1'$ is minimal. Thus
since $\tg_k T_k = T_k \tg_k$, $\tg_k$ acting on each $Z_j' := \tg_k^j(Z_1')$
is minimal. Now minimal sets either coincide or are disjoint, so there is
a least $m$ with $T_k^{m+1}Z_1' = Z_1'$.

For the proof (b), assume $\pi_k(Z') = \pi_k(Z'')$.
Now $Z:= \pi_k(Z')$
is minimal under $g$ and thus since $Z' \subset \pi_k^{-1}(Z)$ and
minimal sets are always disjoint or equal, using \eqref{already}
we have that $Z' = Z_j'$ for some $j$. Similarly,
$Z'' = Z_{j'}'$ for some $j'$, and thus $Z' = T^p{Z_1'}$
and $Z'' = T^{p'}{Z_1'}$ and so $Z'' = T^{p'-p}(Z')$ as required.

Now for (c), Since the deck group of $S_k$ is $\Z_k$ there is
a natural identification of $\pi_k^{-1}(x)$ with  $\Z_k$
with $x_1'$ identified with zero. Since $\tg_k^q(\tx) = T^p\tx$ in 
$\R$ the
induced action of $\tg_k^q$ on $\Z_k$ is $n \mapsto n+p \mod k$.
An easy elementary number theory argument yields that this action has 
exactly $\gcd(p,k)$ distinct orbits. Thus $\tg_k^q$ has exactly
$\gcd(p,k)$ distinct orbits when acting on $\pi^{-1}_k(x)$. But 
$x_i', x_j'\in\pi^{-1}_k(x)$ are on the same $\tg_k^q$
orbit if and only if they are on the same $\tg_k$ orbit and each 
orbit in $\pi^{-1}_k(o(x,g))$ contains at least one point from
$\pi^{-1}_k(x)$. Thus $\tg_k$ acting on  $\pi^{-1}_k(o(x,g))$ has exactly
$\gcd(p,k)$ orbits. The rest of the form of \eqref{already2}
follows from part (a).
\end{proof}
While it is not used here a similar result holds
for $\Z$ covers and their cyclic quotients when a map
in the base has a lift that commutes with the deck transformations.

\begin{remark}\label{rk1}
Some  special cases of (c) are worth pointing out. If $\gcd(p,k) = 1$,
then $x$ lifts to a single period $qk$ orbit in $S_k$. 
If $p=k$, then $x$ lifts to a $k$ different period-$q$ orbits in $S_k$.
When $k=2$, there is a simple dichotomy.
When $p$ is odd, $x$ lifts to one period $2q$ orbit and
when $p$ is even, $x$ lifts to a pair of period $q$ orbits.
\end{remark}

\subsection{Rotation number and interval}\label{rotsection}
 For $\tg:\R\raw\R$ a fixed lift of a degree one $g:S^1\raw S^1$ 
define the \textit{rotation number} of $x'\in \R$ as 
\begin{equation}\label{rotdef}
\rho(x', \tg) = \lim_{n\raw\infty} \frac{\tg^n(x')-x'}{n}
\end{equation}
when the limit exists. Note that 
this value  depend in a simple way on the choice
of lift $\tg$ of $g$, namely, $\rho(x', \tg + m) = \rho(x', \tg) + m$.
In most cases below there will be a preferred lift of
a given $g$ that will be used in \eqref{rotdef} and
we define $\rho(x, g)=\rho(x', \tg)$ where $x'$ is a lift of $x$. 
 When $g$ is understood we will just write $\rho(x)$. 
If $x$ is a periodic point of rotation type $(p,q)$ then $\rho(x) = p/q$.

If $Z$ is a $g$-invariant set, let 
\begin{equation*}\label{rotsetdef}
\rho(Z) = \{ \rho(x, g) : x\in Z\}
\end{equation*}
and $\rho(g) = \rho(S^1, g)$. 
The latter set has been proved to be
 a closed interval \cite{ito, NPT} and
thus it is called the \textit{rotation interval} of the map.
We shall also have occasion to use $\rho(\tg)$ with the obvious meaning.

There is a alternative way of computing the rotation interval
using upper and lower maps that
is now standard (\cite{CGT, Mcirc, Kadanoff, Bdcirc} and elsewhere). 
Given a lift of a degree-one 
circle map $\tg:\R\raw \R$ let $\tg_u(x) = \sup\{\tg(y)\colon
y \leq x\}$ and $\tg_\ell(x) = \inf\{\tg(y)\colon
y \geq x\}$. If $g_u$ and $g_\ell$ are their descents to $S^1$
they are both semi-monotone maps and so each of their rotation
sets is a single point (see Lemma~\ref{basic} below). 
The rotation interval of $g$ is 
\begin{equation}\label{int}
\rho(g) = [\rho(g_\ell), \rho(g_u)]. 
\end{equation}

To  define the rotation number of a $g$-invariant Borel probability measure
$\mu$ start by letting $\Delta_g:S^1\raw \R$ be $\Delta_g(x) = 
\tg(x') - x'$ which is independent of the choice of lift $x'$.
Then
\begin{equation}\label{measrot} 
\rho(\mu) = \int \Delta_g\; d\mu
\end{equation}
Note that when $\mu$ is ergodic 
by the Pointwise Ergodic Theorem  for $\mu$ a.e. $x$,
$\rho(x,g) = \rho(\mu)$.

For points, invariant sets and measures in the cyclic cover
$S_k$ under the preferred
lift $\tg_k$, there
are two ways to consider the rotation number. The
most common will be to project to the base and define 
\begin{equation}\label{rotk}
\rho_k(x, g) = \rho(\pi_k(x), g)
\end{equation}
For $\mu$ a $\tg_k$ invariant measure, let 
\begin{equation}\label{rotkmeas}
\rho_k(\mu) = \rho((\pi_k)_* \mu ).
\end{equation}

\begin{remark}\label{rescale}
The other way to work with rotation numbers in 
$S_k$ is to consider $\tg_k$ as a map of the circle itself.
To work on the standard circle we first rescale $S_k$ via
 $D_k:S_k\raw S^1$ via $D_k(\theta) = \theta/k$.
Note that $D_k$ is not a covering map but rather a coordinate
rescaling homeomorphism.  For $Z\subset S_k$, then
 $\rho(D_k Z, D_k\circ \tg_k\circ D_k^{-1})$ is the desired rotation number.
These two methods are related simply by 
$\rho_k(x, g) = k\rho(D_k Z, D_k\circ \tg_k\circ D_k^{-1})$ 
\end{remark}

\section{Semi-monotone degree-one maps}
\subsection{Definition and basic properties}
In this section we give the basics of a small, but crucial
expansion of the class of circle homeomorphisms, namely,
continuous maps whose lifts are \textit{semi-monotone}. They
share many of the properties of circle homeomorphisms and
are a standard and important tool in circle dynamics 

Thus we consider continuous, degree one $h:S^1 \raw S^1$ whose
lifts $\th$ to $\R$ satisfy
 $x_1' <x_2'$ implies $\th(x_1') \leq \th(x_2')$.\footnote{In topology
a monotone map is one with connected point inverses. In this sense
a semi-monotone map is monotone. On the other hand, considering
the point of view of order relations, semi-monotone is contrast with
monotone. We adapt the latter viewpoint.} Note that
this is independent of the choice of lift $\th$ of $h$. We shall
also call such maps \textit{weakly order preserving}.
Let $\cH$ be the collection of all such maps with the $C^0$-topology, 
and $\tcH$ denotes all their lifts.

A \textit{flat spot} for a $h\in\cH$ is a nontrivial
closed interval $J$ where $h(J)$ is a constant and for which there is
no larger interval containing $J$ on which $h$ is constant. A given
$h$ can have at most a countable number of flat spots
$J_i$ and we define the ``positive slope region'' of $h$ 
as $P(h) = S^1\setminus (\cup \Intt{J_i})$. The proof of the next
result is standard.

\begin{lemma}\label{basic} Assume 
$h\in \cH$ with preferred lift $\th$.
\begin{enumerate}[(a)] 
\item The rotation number $\rho(x, h)$ exists 
and is the same for all $x\in S^1$
 and so $\rho(h)$ is a single number.
\item The map $\rho:\tcH\raw \R$ is continuous.
\item If $\th,\th_1\in\tcH$ and  $\th_1 \leq \th$ then 
$\rho(\th_1)\leq \rho(\th)$.
\item If $\rho(h) = p/q$ in lowest terms then
 all periodic orbits have rotation type $(p,q)$ and the
recurrent set of $h$ consists of a union of such periodic orbits.
\item If $\rho(h)= \alpha\not\in\Q$ then $h$ has exactly one
recurrent set which is a minimal
set $Z$ and it is wholly contained in $P(h)$. Further, $h$ is
uniquely ergodic with the unique invariant measure supported
on $Z$.
\end{enumerate}
\end{lemma}

\begin{definition}\label{semidef}
The minimal set in (e) above is called a \textit{semi-Denjoy
minimal set}. 
More generally, an abstract minimal set  is called
semi-Denjoy if it is topologically conjugate to the semi-Denjoy
minimal set in a semi-monotone degree one circle map.
\end{definition}

\begin{remark}\label{semirk}
A semi-Denjoy minimal set looks like a usual Denjoy minimal
set with the added feature that endpoints of gaps can collapse
to a point under forward iteration. 
It is clear that any $h\in \cH$ is a near-homeomorphism
(the uniform limit of homeomorphisms). Thus as a consequence
of a theorem of Morton Brown (\cite{brown}), the inverse limit
$\ilim (h, S^1)$ is a circle and the natural extension is
a circle homeomorphism. In particular, the inverse limit
of a semi-Denjoy minimal set is a Denjoy minimal set.
For example, in case of single flat spot, the two endpoints of
the flat spot form a gap in the minimal set and they
have  same forward orbit. Taking the inverse
limit splits  open this orbit into a forward invariant gap.
\end{remark}

\subsection{Finitely many flat spots}
We next introduce a subclass of $\cH$ which includes
the semi-monotone maps considered in this paper. Let $\cH_\ell$
consist of those  $h\in\cH$ which have exactly $\ell$ flat spots and in $P(h)$
we require that $h$ is $C^1$ and $h' > 1$  where we have used a one-sided
derivative at the end points of the flat spots.

\begin{definition}
If $h\in\cH_\ell$ and $\rho(h)\not\in\Q$ has semi-Denjoy minimal set $Z$,
since $Z\subset P(h)$ for any flat spot $J$, $Z\cap \Intt(J) = 
\emptyset$.  The flat spot $J$ is called tight for $h$ if $\Fr(J)\subset Z$,
and  otherwise the flat spot is loose.
\end{definition}

\begin{lemma}\label{flat}
Assume $h\in \cH_\ell$. 
\begin{enumerate}[(a)]
\item If $\rho(h) = p/q$ in lowest terms then
the number of $(p,q)$-periodic orbits wholly contained
in $P(h)$ is at least one and at most $\ell$.
\item If $\rho(h) \not\in\Q$, a flat spot $J_i$ is loose if and only
if there is an $n>0$ and a $i'$ with $h^n(J_i) \in J_{i'}$.
In particular, there is always at least one tight flat spot.
\item If $Z$ is the maximal recurrent set of $h$ in $P(h)$,
then 
\begin{equation}\label{comp}
Z = S^1\setminus \bigcup_{n=0}^\infty\bigcup_{i=1}^\ell h^{-n}(\Intt(J_i)),
\end{equation}
and so if $o(x,h)\subset P(h)$ then $h^n(x)\in Z $ for some $n\geq 0$.
\end{enumerate}
\end{lemma}

\begin{proof} For part (a), since $\rho(h) = p/q$ in lowest terms,
every periodic point has period $q$. By the conditions on 
the derivatives of $h\in\cH_\ell$, there are four classes 
of periodic points.
\begin{enumerate}[(1)]
\item $x$ is unstable with $Dh^q(x) >1$ and $o(x,h)\subset \Intt(P(h))$.
\item $x$ is superstable with $Dh^q(x) =0$ and 
$o(x, h) \cap (\cup_{i=1}^\ell \Intt(J_i)) \not= \emptyset$ while
$o(x, h) \cap (\cup_{i=1}^\ell \Fr(J_i)) = \emptyset$.
\item $x$ is superstable with $Dh^q(x) =0$ and 
$o(x, h) \cap (\cup_{i=1}^\ell \Intt(J_i)) \not= \emptyset$ while
$o(x, h) \cap (\cup_{i=1}^\ell \Fr(J_i)) = \emptyset$ and $o(x,h)$ contains
both left and right endpoints of flat spots.
\item $x$ is semistable with $Dh^q(x) =0$ from one side
and $Dh^q(x) >1$ from the other and 
$o(x, h) \cap (\cup_{i=1}^\ell \Intt(J_i)) = \emptyset$ while
$o(x, h) \cap (\cup_{i=1}^\ell \Fr(J_i)) \not= \emptyset$ and $o(x,h)$ contains
only left or only right endpoints of flat spots.
\end{enumerate}
This implies that all periodic points are isolated so there 
are finitely many of them.

Let $n_i$ be the number of periodic orbits of  type (1). Using the
fixed point index on $h^q$, $n_1 = n_2 + n_3$. Each orbit of
type (3) hits two flat spots and each of type (2) and (4) at least
one and a flat spot can't contain multiple periodic orbits and
so $n_1 + 2 n_3 + n_4 \leq\ell$. Thus  the total number of periodic orbits 
wholly contained on $P(h)$ is $n_1 + n_3 + n_4 
= n_2 + 2 n_3 + n_4 \leq \ell$

For part (b)
assume first that $h^n(J_i) \cap J_{i'} = \emptyset$ for all $n>0$ and
$i'$.
If $J_i$ was loose, there would exist $z_1, z_2\in Z$ with $z_1\leq J_i \leq z_2$
with at least one inequality strict and $Z\cap (z_1, z_2) = \emptyset$.
Thus $h((z_1, z_2))$ is a nontrivial interval with $h^n((z_1, z_2))
\subset P(h)$ for all $n>0$. This is impossible since $h$ is expanding
in $P(h)$ and so $J_i$ must be loose.

For the converse, say $h^n(J_i) \in J_{i'}$ for some $n>0$ and first
note $i=i'$ is impossible since that would imply $h$ has a periodic
point. Since $h^n(J_i)$ is a point there exists a nontrivial 
interval $[z_1, z_2]$ properly containing $J_i$ with
 $h^n([z_1, z_2]) = J_{i'}$ and so $(z_1, z_2)\cap
Z = \emptyset$ and so $J_i$ is a loose flat spot.

Finally, since $h^n(J_i) \cap J_{i} = \emptyset$ and there are finitely
many flat spots there is at least one $J_{i}$ with 
$h^n(J_i) \cap J_{i'} = \emptyset$ for all $n>0$ and $i'$.

For (c), assume $y$ is such that $o(y,h)\subset P(h)$ and
$o(y,h)\cap Z = \emptyset$. Let $x,x'\in Z$ with $y\in (x, x')$ and
$(x, x')\cap Z = \emptyset$. Because of the uniform expansion in $P(h)$
there is a flat spot $J$ and an $n\geq 0$ so that $J\subset h^n([x, x'])$.
If $\rho(h)\not\in\Q$, then by (c) for some $n'$, $h^{n+n'}([x, x'])$
is a tight flat spot and so $h^{n + n' + 1}(y) \in Z$.

Now assume $\rho(h) = p/q$. In this case $x$ and $x'$ are periodic
orbits and so $J \subset h^{n+w q}([x, x'])$ for all $w\geq 0$ and
so $ h^{n + wq} (y) \in h^{n+w q}([x, x'])\setminus \Intt(J)$ using the
assumption that $o(y,h)\subset P(h)$. But from (a), $ h^{n + wq} (J)
\subset J$. Thus by monotonicity, $ h^{n + wq} (y)$ is either 
always in the left component of $[x, x'] \setminus \Intt(J)$ 
or in the right component. This violates the expansion in 
$P(h)$ and so  for some $j> 0$, $h^j(y)\in Z$ which yields \eqref{comp}.
\end{proof}

%\begin{remark} For example, if there is one flat spot
%in the rational case there is a single periodic orbit in $P(h)$ and
%in the irrational case, the frontier of the flatspot is in
%the minimal set. When there are two flat spots in the rational
%case there can be one or two periodic orbits in $P(h)$ depending
%onthere are two cases in the
%generic case of non-corner periodic orbits. If the two flat
%spots are on the same orbit, there is one periodic $q$ orbit in
%$P(h)$. This is beacuse $h^q$ has $q$ flat spots
%and the map is expanding between them so contains exactly one fixed
%point  each. When the two flat spots are one different orbits there
%are a pair of period $q$ orbits. In this case $h^q$ has $2q$ flat
%spots so there are $2q$ gaps between. Put another way, in the
%two (generic cases) we have either $2$ or $4$ super-attracting
%sink - saddle pairs. And in irrational case, could be two tight gaps or
%one. In latter case te one (which)is on fwd orbit of the other.
%corner orbits.
%\end{remark} 

\section{A class of bimodal circle maps and their positive slope orbits}
\subsection{The class $\cG$}\label{classG}
We introduce the class of bimodal, degree one maps of the
circle that will be the focus here. The class is defined
using properties of their lifts. We say that a lift $\tg:\R\raw\R$ is
 piecewise smooth if it is continuous and
there are $0\leq x_0 \leq \dots \leq x_n \leq 1$ so that 
$\tg$ is $C^2$ in each interval $(x_i, x_{i+1})$ and the right and
left hand derivatives exist at each $x_i$. 
\begin{definition}\label{Gdef}
Let $\tcG$ be the class of all  $\tg:\R\raw\R$ with
the following properties.
\begin{enumerate}[(a)]
\item $\tg$ is piecewise smooth and $\tg(x'+1) = \tg(x')+1$.
\item There are a pair of points $0 = \xmin < \xmax < 1$ so
that $g'>1$ in $[\xmin, \xmax]$ and $g$ is monotone decreasing
in $[\xmax, \xmin + 1]$.
\item $\xmin \leq \tg(\xmin) < \tg(\xmax) \leq \xmax + 1$
\end{enumerate}
The class $\cG$ consists of all $g:S^1\raw S^1$ which have
a lift in $\tcG$.
\end{definition}

Note that without loss of generality we have assumed that $\xmin = 0$. 
Also by assumption, $\xmin$ and $\xmax$ are a nonsmooth local
minimum and maximum respectively. It follows from 
\eqref{int} that $g\in\cG$ implies $\rho(g)\subset [0,1]$. 
\medskip

\noindent\textbf{Standing assumption:} From this point on
 $g$ denotes a given element of $\cG$ and its preferred lift
is the one with $\tg\in\tcG$.

\begin{remark}\label{generalg} Using the Parry-Milnor-Thurston Theorem for
degree-one circle maps  a general bimodal
$h$ with $\rho(h)\subset (0,1)$ and not a point is semiconjugate to 
a PL map $g\in\cG$\footnote{The result for circles doesn't
seem to be stated and proved anywhere in the literature, but
as noted in \cite{BB} the proof in \cite{misbook} works 
for the circle with minimal alteration}. Point inverses of the semiconjugacy are either
points or a closed interval.  Thus using standard results
from one-dimensional dynamics and various hypotheses most of the
results of the paper can be transferred with appropriate
alterations to a general bimodal map. 
\end{remark}

\subsection{the model map}
We will use a  model map $f_m$ as a specific example throughout
the paper. We shall see that, in a sense, it is 
the largest map in the class $\cG$ and
all other maps $g\in\cG$ may be considered subsystems.

\begin{figure}[htbp]
\begin{center}
\includegraphics[width=0.4\textwidth]{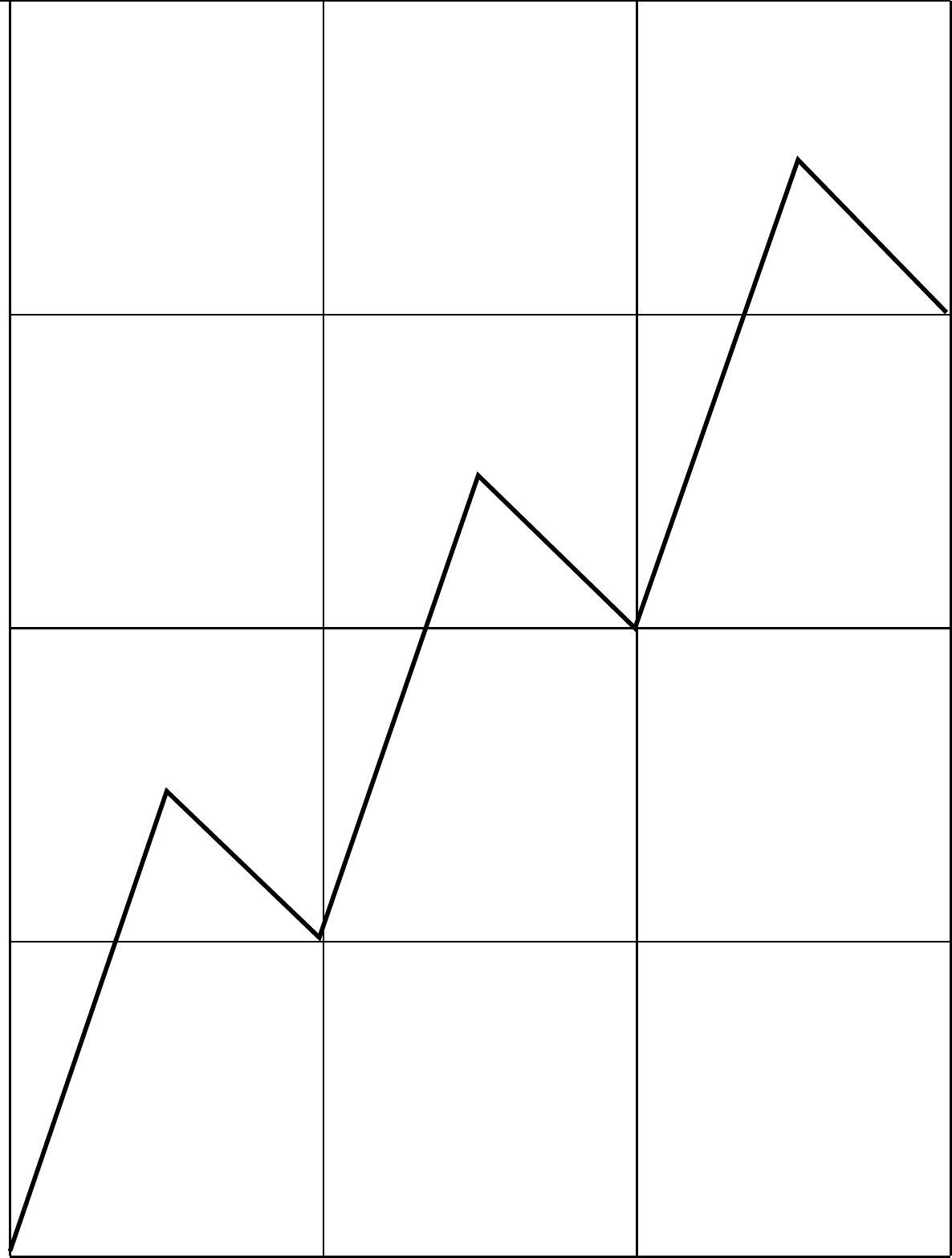}
\end{center}
\caption{The model map $f_m$ in the $3$-fold cover}
\label{model}
\end{figure}

Define $\tf_m:\R\raw\R$ on $[0,1]$ as 
\begin{equation*}\label{maindef}
	\tf_m(x) = \begin{cases}
		3 x  & \text{for}\ 0  \leq x \leq 1/2\\
		-x + 2  & \text{for}\  1/2 \leq x \leq 1
	\end{cases}
\end{equation*}
and extend it to $\R$ to satisfy $\tf_m(x+1) = \tf_m(x) + 1$.
Let $f_m$ be the projection of $\tf_m$ to $S^1$. See Figure~\ref{model}.
Thus, $\xmin = 0$, $\xmax = 1/2$, and $\rho(f_m) = [0,1]$.

\subsection{Positive slope orbits}\label{pso}
Given $g\in\cG$ with preferred lift $\tg$ let $\Lambda_\infty(g)$ 
be the points $x'\in\R$
 whose orbits under $\tg$ stay in
the closed region where $\tg$ has positive slope, so  
\begin{equation*}\label{lamdainfdef}
\Lambda_\infty(g) = 
\{x'\in \R : o(x',\tg)\subset \bigcup_{j=-\infty}^\infty [j, j+\xmax]\}.
\end{equation*}
We give $\Lambda_\infty(g)$ the total order coming from its
embedding in $\R$. Note that it is  both $\tg$ and $T$ invariant.

Now we treat the $k$-fold cover as $S_k = [0,k]/\mysim$ and
let $\Lambda_k(g)$ be the orbits that stay in the positive slope
region of $\tg_k:S_k\raw S_k$, so 
\begin{equation*}
\Lambda_k(g)= \{x'\in S_k : o(x',\tg_k)\subset \bigcup_{j=0}^{k-1} [j, j+\xmax]\}.
\end{equation*}
Alternatively, $\Lambda_k(g) = p_k(\Lambda_\infty(g))$ or
$\Lambda_k(g) = \pi_k^{-1}(\Lambda_1(g))$.

We discuss the restriction to positive slope orbits in Section~\ref{neghom}.

\noindent\textbf{Standing assumption:} Unless otherwise specified the
terminology ``physical kfsm set'' or just ``kfsm set'' carries
the additional restriction that it is contained in the positive slope
region of some $g\in\cG$.

\section{Symbolic description of positive slope orbits}
For a map $g\in \cG$  we develop in this section a symbolic
coding for the  orbits in $\Lambda_k$ for $k = 1, \dots, \infty$.
\subsection{the itinerary maps}\label{Idef}
We work first in the universal cover or $k=\infty$.
Since  $g\in\cG$, we may find points
$\zmax$ and $\zmin$ with $0 = \xmin < \zmax < \zmin < \xmax$
and $\tg(\zmax) = \xmax$ and $\tg(\zmin) = \xmin + 1$.
For $j\in\Z$ define a collection of intervals 
$\{I_j\}$ on $\R$
 by
\begin{equation}\label{address}
\begin{split}
I_{2j} &= [j, \zmax + j] \\
I_{2j + 1} &= [\zmin + j, \xmax + j]
\end{split}
\end{equation}
See Figure~\ref{figure3}.
Note that since $\tg([\zmax, \zmin]) = [\xmax, \xmax+1]$ we have that 
\begin{equation*}
\Lambda_\infty(g) = 
\{x'\in \R: o(x',\tg)\subset \cup_{j=-\infty}^\infty I_j\}.
\end{equation*}
Using $\{I_j\}$ as an address system with the dynamics $\tg$ 
let the itinerary map  be
	$\iota_\infty:\Lambda_\infty(g) \raw \Sigma_\Z^+$. Note that
$\Lambda_\infty$ is the good set and using expansion and the disjointness
of the address intervals, $\iota_\infty$ is a homemomorphism onto its
image.

\begin{figure}[htbp]
\begin{center}
\includegraphics[width=0.3\textwidth]{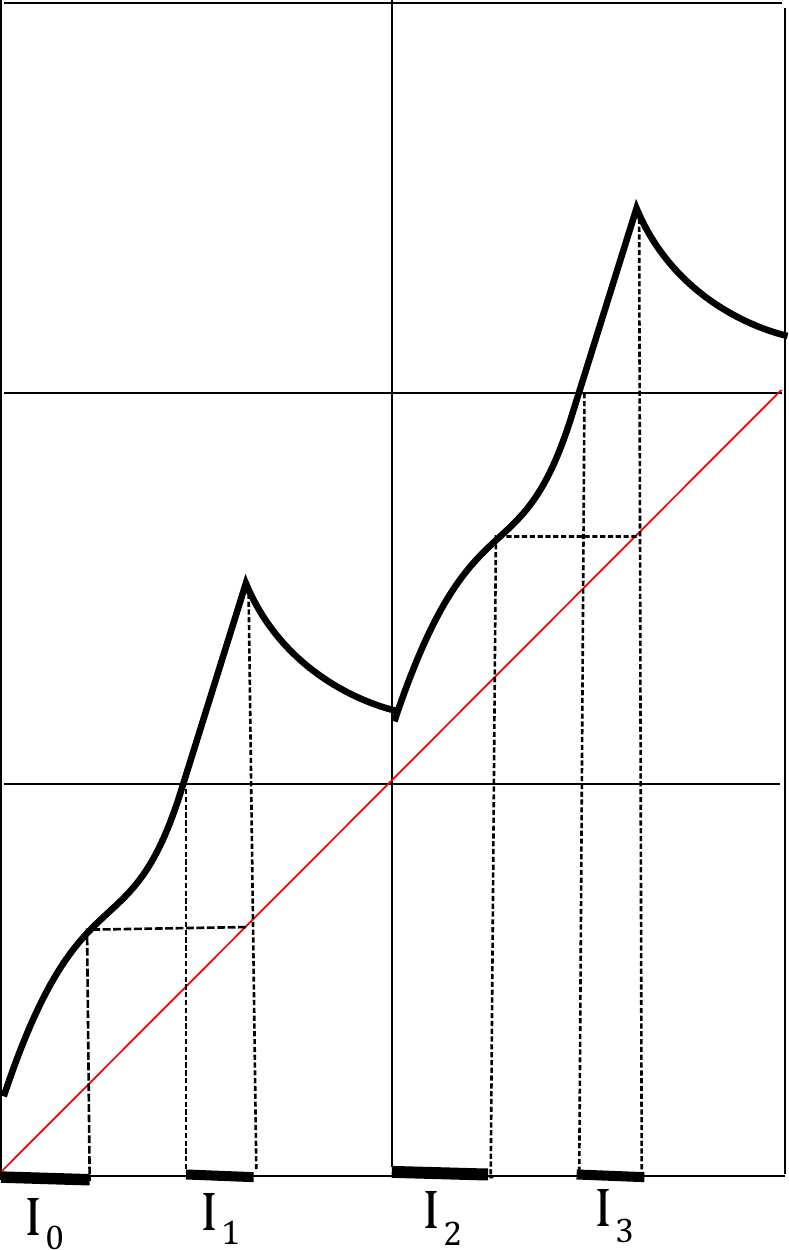}
\end{center}
\caption{The address intervals in the $3$-fold cover}
\label{figure3}
\end{figure}

Now passing to the $k$-fold cover, 
to code the positive slope orbits $\Lambda_k(g)$, treat
$S_k = [0,k]/\mysim$ and use the dynamics $\tg_k$ with the address
system $\{I_0, I_1, \dots, I_{2k-2}, I_{2k-1}\}$. This yields
an itinerary map $\iota_k:\Lambda_k\raw \Sigma_{2k}^+$ which is also
a homeomorphism onto its image.
 
\medskip

\noindent\textbf{Example: The model map} 
For the model map $f_m$ we have $\zmax = 1/6$ and $\zmin = 1/3$ and so 
$I_{2j} = [j, 1/6 + j]$ and  $I_{2j + 1} = [1/3 + j, 1/2 + j]$

\subsection{Symbolic analogs of covering spaces}
This section develops the necessary machinery for the complete
description of the image of the various itinerary maps.
We will need the symbolic analogs of the covering spaces and maps
described in Section~\ref{coverdef}. 
\begin{definition}\label{omegadef}
Define a subshift $\Omega_\infty \subset\Sigma_\Z^+$ 
by its  allowable transitions
\begin{equation}\label{transitions}
2j \raw 2j, \ 2j \raw 2j+1, \
2j+1\raw 2j + 2, \ 2j + 1 \raw 2j+3.
\end{equation}
%\begin{equation}\label{transitions}
%\begin{split}
%2j \raw 2j, &\ 2j \raw 2j+1\\
%2j+1\raw 2j + 2, &\ 2j + 1 \raw 2j+3.
%\end{split}
%\end{equation}
For $k <\infty$ let 
$\Omega_k$ be the subshift of $\Sigma_{2k}^+$ with allowable
transitions as in \eqref{transitions} for $j = 0, \dots, 2k-1$
and indices reduced $\mymod 2k$.
\end{definition}
Since for $g\in\cG$ we have
$\tg(I_{2j})\subset I_{2j}\cup I_{2j+ 1}$ and 
$\tg(I_{2j+1})\subset I_{2j+2}\cup I_{2j+ 3}$ we have: 
\begin{lemma}\label{iotaacts} For $g\in\cG$ and $k = 1, \dots, \infty$,
$\iota_k(\Lambda_k(g))\subset \Omega_k$.
\end{lemma} 

Under the itinerary maps the spaces $\R$, $S_k$ and $S^1$ will
correspond to the shift spaces $\Omega_\infty$, $\Omega_k$ and
$\Omega_1 = \Sigma_{2}^+$. The dynamics on the ``physical spaces''
induced by $g$ will correspond to left shifts on the symbol
spaces. The shift spaces will also have the analogs of the covering
projections and deck transformations. These maps will be indicated
by a hat and defined using the action on individual symbols as follows.

The analogs of the covering translations are $\hT_\infty:\Omega_\infty
\raw \Omega_\infty$ given by $s\mapsto s+2$ for all $s\in\Z$ and 
$\hT_k:\Omega_k
\raw \Omega_k$ given by $s\mapsto s+2\mod 2k $ for all $s\in\Z$,
while the analogs of the covering maps are $\hp_k:\Omega_\infty\raw\Omega_k$
by $s\mapsto s\mod 2k$ and $\hpi_k:\Omega_k\raw\Sigma_{2}^+$
by $s\mapsto s\mod 2$. In the later we allow $k=\infty$ under
the convention that $2\infty = \Z$, yielding
$\hpi_\infty:\Omega_\infty\raw\Sigma_{2}^+$. Note then that
$\hpi_\infty = \hp_1$. A lift and the full lift are defined
as usual with, for example,  a lift of $Y\subset \Omega_1 = \Sigma_2^+$
to $\Omega_k$ is a set $Y'\subset\Omega_k$ with $\hpi_k(Y') = Y$.
Note that $\hT_k, \hpi_k$, and $p_k$ are all continuous.

The roles of the maps $g, \tg_k$ and $\tg$ in Section~\ref{coverdef} are
played by the various shift maps on the sequence spaces. For
clarity we use a subscript to indicate which space the shift
is acting on: $\sigma_k:\Omega_k\raw \Omega_k$. We again
allow  $k=\infty$. 
All the various maps satisfy the same commutativity relations as their 
un-hatted analogs. So, for example, $\hpi_k\hT_k = \pi_k$,
$\sigma_k \hT_k = \hT_k \sigma_k$, and $\hpi_k\sigma_k = \sigma_1 \hpi_k$.
The itinerary maps $\iota_k:\Lambda_k(g)\raw \Omega_k $ 
act naturally by transforming the spaces and maps of Section \ref{coverdef}
to their symbolic analogs as in part (b) of the next Lemma.

\begin{lemma}\label{fact1} For $k = 1, \dots, \infty$, 
\begin{enumerate}[(a)] 
\item $\Omega_k = \hp_k(\Omega_\infty)$
\item $\hpi_k \iota_k = \iota_1 \pi_k$
\item If $\us,\ut\in \Omega_k$ and $\pi_k(\us) = \pi_k(\ut)$, then
there exists a $n$ with $\us = \hT_k^n\ut$.
\end{enumerate}
\end{lemma}
\begin{proof} Parts (a) and (b) are easy to verify. For (c) 
we prove the case $k=\infty$ which implies the $k < \infty$ cases. Assume
$\hpi_\infty(\us) = \uw$. The transitions in \eqref{transitions} coupled
with the structure of $\uw$ imply that once $s_0$ is set the
parity structure of $\us$ determines all of $\uw$. Similarly, 
once $t_0$ is set all of $\ut$ is determined. Once again \eqref{transitions}
implies that if $s_0 - t_0 = 2n$ then for
all $i$, $s_i = t_i + 2n$. 
\end{proof} 

\begin{remark} It would perhaps seem more natural that $\Sigma^+_\Z$
should act as the symbolic universal cover of $\Sigma_2^+$, but the
crucial covering space property expressed by (c) wouldn't hold in
this case. For example, if $\us =.1 3 1^\infty$ and 
$\ut = .1 5 1^\infty$ then   $\hpi_\infty(\us) = \hpi_\infty(\ut)$
but $\hT^n(\us) \not= \ut $ for all $n$.
\end{remark}

\subsection{Rotation numbers and sets}
We give the analogs of the definitions in Section~\ref{rotsection}
for the symbolic case.
For $\us\in\Sigma_2^+$ let
\begin{equation}\label{simrho}
\hrho(\us) = \lim_{n\raw\infty} \frac{1}{n+1} \sum_{i=0}^n s_i.
\end{equation}
when the limit exists. For $\hmu$ a shift invariant measure on 
$\Sigma_2^+$, let $\rho(\hmu) = \hmu([1])$. When $\hmu$ is ergodic,
by the Pointwise Ergodic Theorem, for $\mu$ a.e. $\us$,
$\hrho(\us) = \hrho(\mu)$
 
For $\hZ\in \Omega_k$ let $\hrho_k(\hZ) = \hrho(\hpi_k(\hZ))$
and for $\hmu$ a $\sigma_k$ invariant measure on $\Omega_k$, let
$\hrho_k(\hmu) = \hrho((\hpi_k)_*(\hmu))$.

\section{Topological conjugacies and the image of the itinerary maps}
In this section we develop the analog of kneading invariants for
the symbolic coding of the positive slope orbits for $g\in\cG$.

Recall that $\Sigma_2^+$ is given the lexicographic order. Assume
 $\ukappa_0, \ukappa_1\in\Sigma_2^+$ satisfy 
\begin{equation}\label{kappa}
\ukappa_0 \leq o(\ukappa_i, \sigma) \leq \ukappa_1
\end{equation}
 for $i = 0,1$.
The corresponding \textit{dynamical order interval} is
$$\lrk = \{ \us\colon \ukappa_0 \leq o(\us, \sigma) \leq \ukappa_1\}
$$ 
 
Returning to $g\in\cG$, note that $\tg(I_0) \subset [\xmin, \xmax]$ and 
$\tg(I_1) \subset [\xmin, \xmax]$ while 
$\tg([\zmax, \zmin]) = [\xmax, \xmin + 1]$. This implies that
$\Lambda_1(g) \subset [\xmin, \xmax]$. Since
$\Lambda_1(g)$ is compact  we may define
$\ukappa_0 = \ukappa_0(g) =\iota_1(\min(\Lambda_1))$  and 
$\ukappa_1 = \ukappa_1(g) = \iota_1(\max(\Lambda_1))$. By construction
these $\ukappa$'s satisfy \eqref{kappa}.

We showed above that $\iota_k(\Lambda(g))\subset \Omega_k$. The next
theorem says that the image is as constrained by the dynamical
order interval $\lrk$. Accordingly for $k=1, \dots, \infty$ we define
$\hLambda_k(g) = \Omega_k \cap \hpi_k^{-1}(\lrk)$ and note that
this is a $\sigma_k$ invariant set.
\begin{theorem}\label{conjugacy} Assume $g\in\cG$ and construct $\kappa_0$ and
$\kappa_1$ from $g$ as above. Then for $k = 1, \dots, \infty$
the itinerary map $\iota_k$ is a
topological conjugacy from $(\Lambda_k(g), (\tg_k)_{\vert \Lambda_k(g)})$
to $(\hLambda_k(g), \sigma_k)$. Further, $\iota_\infty$ 
is order preserving.
\end{theorem}

\begin{proof} We first prove the first assertion for $k=1$ or that
$\iota_1(\Lambda_1(g)) = \lrk$. Let $\ast$ be an arbitrary symbol and define
a map $\chi:[0,\xmax]\sqcup\{\ast\}\raw  [0,\xmax]\sqcup\{\ast\}$ by
\begin{equation*}
\chi(x) = \begin{cases} 
\tg(x) \ \text{for} \ x\in I_0\\
\ast\ \text{for} \ x\in (\zmin, \zmax)\sqcup\{\ast\}\\
\tg(x) -1 \ \text{for} \ x\in I_1
\end{cases}
\end{equation*}
It easily follow that 
\begin{equation*}
\Lambda_1(g) = \{x\in [0, \xmax]\colon \chi^n(x)\not= \ast \ \text{for all}
\ n>0\}.
\end{equation*}
and if use the dynamics $\chi$ with the address system $I_0, I_1$
the resulting itinerary map $\Lambda_1\raw \Sigma_2^+$ is exactly
$\iota_1$. Now since $g$ is expanding on $I_0 \cup I_1$ and 
$I_0 \cap I_1 = \emptyset$, $\iota_1$ is an order preserving
conjugacy from $(\Lambda_1,g)$ to $(\iota_1(\Lambda_1),\sigma_1)$.
Finally, since $\min \Lambda_1(g) \leq o(x,g) \leq \max \Lambda_1(g)$
for all $x\in\Lambda_1(g)$ 
 we have that $\kappa_0 \leq o(\us,\sigma) \leq \kappa_1$
for all $\us\in\iota_1(\Lambda_1)$ and further that for any such $\us$
there is an $x\in\Lambda_1(g)$ with $\iota_1(x) = \us$. Thus
$\iota_1(\Lambda_1(g)) = \lrk$.

We now show that 
\begin{equation}\label{image}
\iota_k(\Lambda_k(g)) = 
\Omega_k \cap \hpi_k^{-1}(\lrk)
\end{equation}
We already know from Lemma~\ref{iotaacts} that the left hand side is in 
$\Omega_k$. Next, since $\pi_k(\Lambda_k(g)) = \Lambda_1(g)$
using Fact~\ref{fact1}(b) and the first paragraph of the proof we have 
\begin{equation}\label{image2}
\lrk = \iota_1(\Lambda_1(g)) = \iota_1(\pi_k(\Lambda_k(g)))
= \hpi_k\iota_k(\Lambda_k(g))
\end{equation}
so the left hand side of \eqref{image} is also in $\hpi_k^{-1}(\lrk)$.

Now assume that $\us$ is in the right hand side of \eqref{image}. Certainly
then $\hpi_k(\us)\in\lrk$ and so there is an $x\in\Lambda_1(g)$ with
$\iota_1(x) = \hpi_k(\us)$. Pick a lift $x'\in\Lambda_k(g)$ with 
$\pi_k(x') = x$. Again using Lemma~\ref{fact1}(b)
\begin{equation}\label{eq3}
\hpi_k(\us) = \iota_1(x) = \iota_1\pi_k(x')  = 
\hpi_k\iota_k(x').
\end{equation}
Thus using Lemma~\ref{fact1}(c),
 there is an $n$ with $\iota_k(x') = \hT_k^n(\us)$
and so 
\begin{equation*}
\iota_k\hT_k^{-n} x' = \hT_k^{-n}\iota_k x' = \us
\end{equation*}
and $ \hT_k^{-n} x'\in \Lambda_k$. Thus $\us\in \iota_k(\Lambda_k)$
as required. 

For $\iota_\infty$ as with $\iota_1$, since the $I_j$ are disjoint
 and the $\tg_\vert I_j$ are expanding, we have that $\iota_\infty$
is an order preserving homeomorphisms onto its image. The fact that
it is a semiconjugacy follows because  it is an itinerary map.
\end{proof}

\noindent\textbf{Example: The model map} For the model map $f_m$ we have
 $\kappa_0 = .0^\infty$ and $\kappa_1 = .1^\infty$
and so in this case $\hLambda_k(\lrk)$ is the entire subshift
$\Omega_k$.

\begin{remark}\label{fact3}
\begin{enumerate}[(a)]
\item $\hrho\circ \iota_k = \rho$ (when defined) and 
$\hrho_k\circ \iota_k = \rho_k$
\item For $\mu$ a $g$ invariant measure supported
in $\Lambda_1(g)$, $\rho(\mu) = \mu(I_1)$
\end{enumerate}
\end{remark}
  
\section{$k$-fold semi-monotone sets}
While our eventual interest is in invariant
sets in the circle, it is convenient to first give definitions in the
universal cover $\R$ and the cyclic covers $S_k$.
\subsection{Definitions}
The next definition makes sense for any degree one map but for concreteness
we restrict to $g\in\cG$.
 \begin{definition}\label{kfolddef}
Let $g\in\cG$ have preferred lift $\tg:\R\raw \R$.
\begin{enumerate}[(a)]
\item A $\tg$-invariant set $Z'\subset\R$ is \textit{k-fold semi-monotone}
(kfsm) if $T^k(Z') = Z'$ and  $\tg$ restricted to $Z'$ is
weakly order preserving, or for $z_1', z_2'\in Z'$ 
\begin{equation*}
z_1' <  z_2' \ \ \text{implies} \ \ \tg(z_1') \leq \tg(z_2')
\end{equation*} 
\item A $\tg_k$-invariant set $Z\subset S_k$ is
\textit{k-fold semi-monotone} (kfsm) 
if it has a $\tg$-invariant lift $Z'\subset\R$ which is.
\end{enumerate}
\end{definition}
These definitions are independent
of the choice of lift $\tg$.
Note that the same terminology is used for sets in the universal
and cyclic covers and that implicit in being a kfsm set is
the fact that the set is invariant.
 
When $k=1$ the lift $Z'$ in
the definition must satisfy $T(Z') = Z'$ and $\pi(Z') = Z$
and so  $Z' = \pi^{-1}(Z)$, the full lift to $\R$.

\subsection{interpolation}
To say that $Z\subset S_k$ is $k$-fold semi-monotone means
roughly that it is semi-monotone treating $S_k$ as the usual
circle. To formalize this as in Remark~\ref{rescale} it will be useful
 to rescale $S_k$ to
$S^1$ using $D_k:S_k\raw S^1$ and consider
the map  $ D_k\circ \tg_k\circ D_k^{-1}$.

\begin{lemma}\label{tfae} The following are equivalent
\begin{enumerate}[(a)]
\item The  $\tg_k$-invariant set $Z\subset S_k$ is kfsm
\item  $DZ$  is $1$-fold semi-monotone under $ D_k\circ \tg_k\circ D_k^{-1}$
and there exists a semi-monotone circle map $h$ defined on $S_k$ which
interpolates $\tg_k$ acting on $Z$.
\item The lift $Z'\subset \R$ of $Z$ in Definition~\ref{kfolddef}(b) 
has the property that there is a continuous $H:\R\raw\R$ that interpolates
$\tg$ acting on $\tZ^*$,  is  weakly order preserving, and
satisfies $H(x + k) = H(x) + k$.
\end{enumerate}
\end{lemma}

We now restrict to positive slope orbits as in Section~\ref{pso} and 
collect together kfsm invariant sets in $S_k$ and their invariant measures.
We will comment on kfsm sets which intersect the negative slope region in 
Section~\ref{neghom}. We also restrict attention to invariant sets that
are recurrent.
\begin{definition}\label{spacedef}
Given $g\in \cG$ let $\cS_k(g)$ be all compact, recurrent
kfsm sets in $\Lambda_k(g)\subset S_k$ with the Hausdorff topology
 and $\cN_k(g)$ be all $\tg_k$-invariant, Borel
probability measures with the weak topology whose support is a $Z\in \cS_k(g)$. 
\end{definition}
 
\begin{remark}\label{compact}
 A standard argument from Aubry-Mather theory yields that the collection of
all kfsm sets is compact in the Hausdorff topology. Since
$\Lambda_k(g)$ is compact, the collection of positive slope
kfsm sets is also compact. However, since  $\cS_k(g)$ contains just the
 recurrent kfsm sets, it is not compact (see Section~\ref{holesec} and~\ref{neghom}.
 We show shortly that $\cN_k(g)$ is compact.
\end{remark}

\subsection{symbolic k-fold semi-monotone sets and the map $g$}
As with kfsm sets in the ``physical'' spaces $S_k$
and $\R$ we define their symbolic analogs in the symbol spaces $\Omega_k$ 
and $\Omega_\infty$ where we give the symbol spaces the lexicographic order.
\begin{definition}\label{symkfsm} $\ $
\begin{enumerate}
\item A $\sigma_\infty$-invariant set $\hZ'\subset\Omega_\infty$ 
is \textit{symbolic k-fold semi-monotone} (kfsm)
if $\hT_\infty^k(\hZ') = \hZ'$ and  $\sigma_\infty$ restricted to $\hZ'$ is
weakly order preserving, or for $\us, \ut\in \hZ'$ 
\begin{equation*}
\us <  \ut \ \ \text{implies} \ \ \sigma_\infty(\us) \leq \sigma_\infty(\ut).
\end{equation*}  
\item A $\sigma_k$-invariant set 
$\hZ\subset \Omega_k$ is
 \textit{symbolic k-fold semi-monotone} (kfsm)
if there is a $\sigma_\infty$-invariant lift $\hZ'$ to $\Omega_\infty$. 
(i.e., $\hp_k(\hZ') = \hZ$) which is kfsm.

\end{enumerate}

\end{definition}

Everything has been organized thus far to ensure that 
k-fold semi-monotone sets are preserved under the itinerary maps.
\begin{theorem}\label{symandreal} Given $g\in\cG$ 
 for $k = 1, 2, \dots, \infty$,
 a $\tg_k$ invariant set $Z\subset \Lambda_k(g)$ is kfsm
if an only if $\iota_k(Z)\subset \hLambda_k(g)$ is.
\end{theorem}
\begin{proof}
We prove the $k=\infty$ case; the $k<\infty$ case follows.
Theorem \ref{conjugacy} shows that $\iota_\infty$ is an order preserving bijection.
Since $\iota_\infty T_\infty^k = \hT_\infty^k \iota_\infty$ we have that
$T_\infty^k(Z) = Z$ if and only if 
$\hT^k_\infty \iota_\infty(Z) = \iota_\infty(Z)$. Using the additional
fact that $\iota_\infty\tg = \sigma_\infty \iota_\infty$ we have 
that $\tg$ is weakly order preserving on $Z$ if and only if 
$\sigma_\infty$ is  weakly order preserving on $\iota_\infty(Z)$
\end{proof}

In analogy with Definition~\ref{spacedef} we collect
together the various symbolic kfsm sets and their invariant 
measures.
\begin{definition}\label{symdef} For $k<\infty$
given $g\in \cG$, let $\hcS_k(g)$ be all compact, invariant, recurrent
symbolic kfsm sets in $\hLambda_k(g)$ with the Hausdorff topology 
and $\hcN_k(g)$ be all $g$-invariant, Borel
probability measures with the weak topology whose support is a
 $\hZ\in \hcS_k(g)$. 
\end{definition}
\begin{lemma}\label{iotainduced} 
\begin{enumerate}[(a)] For $k < \infty$
\item  The map $\iota_k:\Lambda_k(g) \raw \hLambda_k(g)$
induces homeomorphisms
$\cS_k(g)\raw \hcS_k(g)$ and $\cN_k(g)\raw \hcN_k(g)$. 
\item The spaces $\cN_k(g)$ and $\hcN_k(g)$ are compact.
\end{enumerate}
\end{lemma}
\begin{proof}
For part (a) we know that 
 $\iota_k$ is a conjugacy that that takes kfsm sets to kfsm  sets which yields
that $\cN_k(g)\raw \hcN_k(g)$ is a homeomorphism. By hypothesis
any $g\in\cG$ is $C^2$ in $P(g)$ and so there is some $M>1$ 
with $g'< M$ on $P(g)$ and thus on all address intervals
$I_j$. It is standard that this implies  that $\iota_k$ is
H\"older with exponent $\nu = \log 2k/\log M$. This then  implies
that $\iota_k$ preserves Hausdorff convergence and so
 $\cS_k(g)\raw \hcS_k(g)$  is a homeomorphism.

For part (b), since the space
of all $\tg_k$ invariant Borel probability measures is compact
metric, it suffices to show that $\cN_k(g)$ is closed,  and so
assume $\mu_n\in \cN_k(g)$ and $\mu_n\raw\mu$ weakly with
$X_n := \spt(\mu_n)$ a recurrent kfsm set.

A noted in Remark~\ref{compact}
the collection of all kfsm sets in $\Lambda_k$ is compact in the
Hausdorff topology and so there exists a kfsm set $X$ and 
$n_i\raw\infty$ with $X_{n_i}\raw X$. A standard argument which we give
here shows that $\spt(\mu) \subset X$.  
If this inclusion does not hold, there exists an 
$x\in \spt(\mu) \cap X^c$, then
let $\epsilon = d(x, X)$. Since the atoms of $\mu$
are countable, we may find an $\epsilon_1 < \epsilon/4$ so
that letting $U = N_{\epsilon_1}(x)$ we have
so that $\mu(\Fr(U)) = 0$ and so $U$ is a continuity
set for $\mu$  thus  via a standard result (page 16-17 of \cite{bill}) 
 $\mu_{n_i}(U) \raw \mu(U) >0$
using the fact that  $x\in\spt(\mu)$. Thus for large enough $i$,
with $m = n_i$ we have $X_m\subset N_{\epsilon/4}(X)$
and so $\emptyset = U \cap X_m = U\cap\spt(\mu_m)$ with
$\mu_m(U)>0$ a contradiction. Thus $\spt(\mu) \subset X$.
Now any invariant measure supported on $X$ must be supported
on its recurrent set and so $\mu\in\cN_k(g)$, as required.
The compactness of $\hcN_k(g)$ follows from part(a).
\end{proof}

\noindent\textbf{Example: The model map} For the model map 
$f_m$, $\hLambda_k(f_m) = \Omega_k$, and so
the set $\hcS_k(f_m)$ is the collection of all symbolic 
recurrent kfsm sets in $\Omega_k$. Thus while the definition
of symbolic kfsm set is abstract and general by Theorem~\ref{symandreal}
 and Theorems~\ref{conjugacy} they
will share all the properties of ``physical'' kfsm sets.

\subsection{Rotation numbers and sets}
For $Z\in \cS_k(g)$ recall from section~\ref{rotsection} that 
$\rho_k(Z) = \rho(\pi_k(Z), g)$. 

\begin{lemma}\label{rotlemma} Assume $Z\in\cS_k(g)$,
\begin{enumerate}[(a)]
\item $\rho_k(Z)$ exists and is a single number.
\item If $\rho_k(Z) = \omega\not\in\Q$ then $Z$ is a semi-Denjoy minimal
set. 
%$\#\phi^{-1}(x) = 2$.
\item If $\rho_k(Z) = p/q$ with $\gcd(p,q) = 1$, then $Z$ consists
of at least one and at most $k$ periodic orbits all with the same
rotation number and period equal to $q k/\gcd(p,k)$. 
\item $\rho_k:\cS_k(g)\raw \R$ and $\hrho_k:\hcS_k(g)\raw \R$ are continuous
\end{enumerate}
\end{lemma}

\begin{proof}
By Theorem~\ref{tfae} there exists a continuous, semi-monotone $H:S_k\raw S_k$
which interpolates the action of $\tg_k$ on $Z$. Rescaling
to the standard circle let $H_k:S^1\raw S^1$ be defined
as $H_k:= D_k\circ H\circ D_k^{-1}$. By Lemma~\ref{basic}(a),
 $\rho(H_k) = \omega$ is
a single number and since $\rho_k(Z) = k \rho(D Z, H_k)$, (a) follows.
If $\rho_k(Z)\not\in\Q$ then $\rho(D Z, H_k)\not\in\Q$ and so by
 Lemma~\ref{basic}(e) 
$DZ$ and thus $Z$ is a semi-Denjoy minimal set yielding (b).

Now assume $\rho_k(Z) =  p/q$ in lowest terms and so $\rho(D Z, H_k) = p/(qk)$.
Written in lowest terms  
$$\frac{p}{qk} = \frac{p/\gcd(p,k)}{kq/\gcd(p,k)}.$$
But since $H_k$ is semi-monotone, its recurrent set is a collection
of periodic orbits and its rotation number in lowest terms
has their period as its denominator which is thus $q k/\gcd(p,k)$.
Since by assumption, $Z\subset\Lambda_k(g)$ we may choose $H$
to have $k$ flat spots  then using Lemma~\ref{flat}, $Z$ consists
of at least one and at most $k$ periodic orbits, finishing (c).
 
It is standard from Aubry-Mather theory
 that $\rho$ is continuous on the collection of all
kfsm sets and thus it is continuous restricted to
the recurrent kfsm sets. As for measures, since 
$\rho(\mu) = \int \Delta_g\; d\mu$ \eqref{measrot}
with $\Delta_g$ continuous, continuity follows from the definition
of weak convergence.
\end{proof}

\begin{definition}\label{clusterdef} If $Z\subset \cB_k$ and consists of
a finite collection of periodic orbits it is called a cluster.
\end{definition}

\begin{remark}\label{rk3}
\begin{enumerate}[(a)]
\item For the case of general recurrent symbolic kfsm $\hZ$
as we commented at the end of the last subsection we may consider
$\hZ\in\hLambda_k(f_m) = \Omega_k$ with $f_m$ the model map. 
Using the itinerary map 
$\iota_k:\Lambda_k(f_m) \raw \hLambda_k(f_m)$ we have from 
Theorem~\ref{symandreal} that
$(\iota_k)^{-1}(\hZ)$ is kfsm for $f_m$ and then all the
conclusions of the previous theorem hold for it. Then using
Theorem~\ref{conjugacy}, the conclusions of the previous theorem hold with the obvious
addition of hats in the appropriate places.
\item We shall need this implication of the symbolic case below.
If $\hZ\subset\Omega_k$ with $\rho_k(\hZ) = \alpha\not\in\Q$,
tnen there exists a continuous, onto 
$\phi:\hZ\raw S_k$ which is weakly order preserving,
$\phi\sigma_k = R_\alpha \phi$, and $\#\phi^{-1}(x) = 1$
for all but a countable number of $R_\alpha$-orbits on which 
$\#\phi^{-1}(x) = 2$.

\item Using Lemma~\ref{basic} a measure in $\cN_k(g)$ is either the unique
measure on a semi-Denjoy minimal set or a convex combination or
measures supported on the periodic orbits in a cluster.
\item A $Z\in\cS_k(g)$ is minimal if and only if it is uniquely
ergodic and similarly for $Z\in\hcS_k(g)$ 
\end{enumerate}
\end{remark}

\section{The HM construction}
For a given $g\in\cG$ at this point we have reduced the identification
of its positive slope kfsm sets to a question in symbolic dynamics. In this
section we answer this symbolic question via a generalization of the 
procedure of Hedlund and Morse. The generalization
  constructs all symbolic kfsm recurrent
sets for each $k$.

Since a linear order is essential to the notion of semi-monotone
 we will again begin
working on the line and  then project to cyclic covers.

\subsection{definition and basic properties}\label{Xdef}
Fix an integer $k>0$, a real number  $\omega\in (0,1)$,
and a vector $\vnu = (\nu_1, \dots, \nu_k)$ with $\nu_i\geq 0$
and  $\sum \nu_i =k -k\omega$. Such a pair $(\omega, \vnu)$ is
called \textit{allowable}.
Start with the intervals  defined for $0 \leq j \le k-1$ by
\begin{equation}\label{addressHM}
\begin{split}
X_{2j} &= \left(\sum_{i=1}^j \nu_i + j\omega, 
\sum_{i=1}^{j+1} \nu_i + j\omega\right)\\
X_{2j+1} &= \left(\sum_{i=1}^{j+1} \nu_i + j\omega,
 \sum_{i=1}^{j+1} \nu_i + (j+1)\omega\right)
\end{split}
\end{equation}
and then extend for  $\ell\in \Z$ and $0 \leq m \le 2k-1$
as $X_{\ell k + m} = X_m + \ell k$. Thus  each $X_{2j}$ has
width $\nu_{j+1}$ and each $X_{2j+1}$ has
width $\omega$ and the entire structure yields a $T^k$ invariant address system
under the dynamics $R_\omega(x) = x + \omega$ on $\R$

The good set is  $G_\infty$ depends on $k, \omega$ and $\nu$ and is given by
\begin{equation*}
G = \{ x'\in \R: o(x', R_\omega)\cap \partial X_i = \emptyset\ 
\text{for all}\ i\}
\end{equation*}
Note that $G$ is dense, $G_\delta$ and has full Lebesgue measure.
The itinerary map with respect to the given address system
is denoted $\zeta_\infty:G\raw \Sigma^+_{\Z}$.
\begin{definition}\label{Akdef}
Let $A_{k}(\omega,\vnu) = \Cl(\zeta_\infty(G))$.
\end{definition}
\begin{remark}\label{rk4}
By construction,
$A_{k}(\omega,\vnu)$ is $\sigma_\infty$  and $\hT_\infty^k$ invariant.
In addition, since for all $j$, 
$R_\omega(X_{2j}) \subset X_{2j} \cup X_{2j+1} $
and $R_\omega(X_{2j+1}) \subset X_{2j+2} \cup X_{2j+3} $ 
we have $\zeta_\infty(G_\infty)\subset \Omega_\infty$. Also, since
$\Omega_\infty$ is compact, $A_{k}(\omega,\vnu)\subset \Omega_\infty$
is also.
\end{remark}

\subsection{Cyclic covers}
We now return to the compact quotients where the recurrent dynamics
takes place and introduce measures into the HM-construction.

For fixed $k>0$ and allowable $(\omega, \vnu)$ 
 treat $\{X_0, \dots, X_{2k-1}\}\subset S_k =  [0,k]/\mysim$
as an address system under the dynamics given
by $R_\omega(x) = x + \omega \mod k$.
Define the good set $G_{k\omega\vnu}$ and on it define the itinerary map
 $\zeta_{k\omega\vnu}$. We will often suppress the dependence
of these quantities on various of the subscripted variables when
they are clear from  the context.
\begin{definition}\label{HMitin} Given $k$ and an allowable
 $(\omega, \vnu)$ define
the itinerary map $\zeta_k:G_k\raw \Sigma^+_{2k}$ as above.
 Let $B_{k}(\omega, \vnu) = \Cl(\zeta_k(G_{k}))\subset\Sigma_{2k}^+$ and
$\lambda_{k}(\omega, \vnu) = (\zeta_k)_*(m)/k$ where $m$ is the measure
on $S_k$ induced by Lebesgue measure on $\R$.
\end{definition}

\begin{remark}\label{rk2} $\ $
\begin{enumerate}
\item By construction
 $\hp_k(A_{ k}(\omega,\vnu)) = B_{k}(\omega, \vnu)\subset \Omega_k$.
and so  $\rho_k(B_k(\omega, \vnu)) = \omega$.
\item Let $W_k = \{(x, \omega, \vnu) \colon x\in G_{k\omega\vnu}\}.$
 It is easy to check that the map
$(x,\omega, \vnu)\mapsto \zeta_{k\omega\vnu}(x)$ is continuous on $W_k$.
\end{enumerate}
\end{remark}

 The next theorem  
describes the structure of   the $B_k(\omega, \vnu)$  and
shows that all symbolic kfsm sets are constructed by the HM 
procedure with $\omega$ equal to their rotation number.
\begin{theorem}\label{Bkstruc} $\ $
\begin{enumerate}[(a)]
\item For $\alpha\not\in\Q$, $B_k(\alpha, \nu)$ is 
a semi-Denjoy minimal set with unique invariant probability measure
$\lambda_k(\omega,\vnu)$.
  
\item For $p/q\in\Q$, 
$B_k(p/q, \nu)$ is a finite collection of periodic
orbits each with rotation number $p/q$ and period $qk/\gcd(p,k)$,
 and $\lambda_k(p/q, \nu)$
is a convex combination of the measures supported on the
periodic orbits.
\item A $\hZ\subset\Omega_k$ is a recurrent symbolic kfsm
set with $\rho_k(Z) = \omega$ if and only if 
$\hZ = B_k(\omega, \vnu)$ for some allowable $\vnu$. Thus the
collection of invariant probability measures supported on
symbolic recurrent kfsm set is exactly the collection of 
$\lambda_k(\omega, \vnu)$ for all allowable $(\omega, \vnu)$.
\end{enumerate}
\end{theorem}
\begin{proof} We begin by proving portions of (a) and (b).
For part (a) we first show that $B_k(\alpha, \nu)$ is minimal
using a characterization
usually attributed to Birkhoff: If $f:X\raw X$ is a continuous
function of a compact metric space and $x\in X$, then
$\Cl(o(x,f))$ is a minimal set if and only if for all $\epsilon>0$
there exists an $N$ so that for all $n\in\N$ there is a 
$0<i\leq N$ with $d(f^{n+i}, x) < \epsilon$. Pick $x$ in
the good set $G$, since $(S^1, R_\alpha)$ is minimal,
$o(x, R_\alpha)$ has the given property. Since $\zeta_k$ restricted
to $G$ is a homeomorphism  and $\zeta_k R_\alpha = 
\sigma_k \zeta_k$, $o(\zeta_k(x), R_\alpha)$ has the desired
property and further, $o(\zeta_k(x), R_\alpha)$ is dense in
$\zeta_k(G)$ and thus in $B_k(\alpha, \vnu) = \Cl(\zeta_k(G))$.
Thus $B_k(\alpha, \vnu)$ is minimal under $\sigma_k$.

For part (b) note first that since $R_{p/q}$ is finite order
 and there are finitely
many address intervals, $B_k(p/q, \vnu)$ must consist of
finitely many periodic orbits. The other properties in (a) and
(b) will follow from (c) (proved using just these two partial
results on (a) and (b)) and Theorem \ref{rotlemma} using Remark~\ref{rk3}(a).

For part (c), we first show that $B_k(\omega, \vnu)$ is a recurrent
symbolic kfsm set.  By parts (a) and (b) we know that $B_k(\omega, \vnu)$ 
is recurrent and by Remark~\ref{rk2} that $B_k(\omega, \vnu)\subset\Omega_k$ 
and $\rho_k(B_k(\omega, \vnu) = \omega$.
 We  show that $B_k(\omega, \vnu)$ is symbolic kfsm set
by showing that its full lift $A_k(\omega, \vnu)$ to $\Omega_\infty$
is as required by Definition~\ref{Akdef}. As noted in Remark~\ref{rk1}, 
$\hT^k_\infty(A_{k}(\omega,\vnu)) = \A_{k}(\omega,\vnu)$ so we need
to show that $\hsigma_\infty$ is semi-monotone on  $A_k(\omega, \vnu)$.  

The first step is to show that $\zeta_\infty$ is weakly order preserving.
Assume $x_1', x_2'\in G$ with $x_1' < x_2'$. It could happen
(when $\omega$ is rational) that $\zeta_\infty(x_1') = \zeta_\infty(x_2')$,
but if there exists a least $n$ with 
$(\zeta_\infty(x_1'))_n \not= (\zeta_\infty(x_2'))_n$, then since $I_m < I_{m+1}$
for all $m$ and $R_\omega$ is order preserving, 
certainly $(\zeta_\infty(x_1'))_n < (\zeta_\infty(x_2'))_n$, and
so $\zeta_\infty$ is weakly order preserving.

We now show that $\hsigma_\infty$ is semi-monotone on  $A_k(\omega, \vnu)$.
 Let $G$ be the good set for $\zeta_\infty$ and assume 
$\us, \ut\in\zeta_\infty(G)$ with $\us < \ut$.
Then there exist $x_1', x_2'\in G$ with $\zeta_\infty(x_1') = \us$
and $\zeta_\infty(x_2') = \ut$ and of necessity, $x_1' < x_2'$ and so
$R_\omega(x_1') < R_\omega(x_2')$. Since $\zeta_\infty R_\omega
= \sigma_\infty\zeta_\infty$, we have
\begin{equation*}
\sigma_\infty(\us) = \sigma_\infty \zeta_\infty(x_1') = 
\zeta_\infty R_\omega(x_1) \leq \zeta_\infty R_\omega(x_2) 
= \sigma_\infty \zeta_\infty(x_2') = \sigma_\infty(\ut).
\end{equation*}
Thus $\sigma_\infty$ is weakly order preserving on $\zeta_\infty(G)$
and so on $A_k(\omega, \vnu)$. We have that $A_k(\omega, \vnu)$ satisfies
all the conditions of the lift in  Definition~\ref{symkfsm} and thus $B_k(\omega, \vnu)$
is symbolic kfsm. 

Now for the converse assume that $\hZ\subset\Omega_k$ is symbolic recurrent
kfsm with $\hrho(\hZ) = \omega$. Let $\hZ'\subset \Omega_\infty$
be the lift that satisfies Definition~\ref{symkfsm}. The proof splits into the two
cases when $\omega$ is rational and irrational. 

First assume $\omega = p/q$ with $\gcd(p,q) = 1$. We know from  
Lemma~\ref{rotlemma} and Remark~\ref{rk3}
that $\hZ$ consists of at most $k$ distinct periodic orbits
each with period $kq/d$ with $d=\gcd(p,k)$. We assume for simplicity 
that $\hZ$ is a single periodic orbit. The case of multiple
periodic orbits is similar but with more elaborate indexing.

For $i = 0, \dots, kq/d-1$ let 
$P_i = (2 i + 1)d/2q\subset S_k$ and $\cP= \{P_i\}$. Since $\hZ$ is 
a kfsm periodic orbit with $\rho_k$-rotation number
$p/q$  we may find an order preserving bijection 
$\phi:\hZ\raw \cP$ with $\phi\sigma_k = R_{p/q} \phi$ on 
$\hZ$. Thus $\phi\sigma_k\phi^{-1}$ acts on $\cP$ as
$P_i \mapsto P_{i+p/d}$ reducing indices mod $kq/d$.

For $j = 0, \dots, k-1$, let $X'_j = \phi(\hZ \cap [j])$ where recall that 
$[j]$ is the length one cylinder set in $\Sigma^+_{2k}$. Since $\phi$
is order preserving, each $X_j'$ consists of 
a collection of adjacent points from $\cP$. If $X_j' =
\{ P_{n(j)}, \dots, P_{m(j)}\} \not = \emptyset$, let 
$X_j =   [P_{n(j)} - d/(2q), P_{m(j)} + d/(2q)]$ and when $X_j'=\emptyset$
let   $X_j=\emptyset$. We now claim that $\{X_j\}$ is an address
system as used in the HM-construction where $\vnu$ is defined
by  $\nu_{j+1} = |X_{2j}|$
and that $|X_{2j+1}| = p/q$ for $j = 0, \dots, k-1$ yielding
$\hZ' = B_k(p/q, \vnu)$.

Letting $\zeta$ be the itinerary map for the address system
$\{X_j\}$ we have by construction that
 for $\us\in \hZ$ that $\zeta\phi(\us) = \us$.
In addition, for all $x\in [\phi(\us)-d/(2q), \phi(\us) + d/(2q)]$
we also have $\zeta(x) = \zeta\phi(\us) = \us$. Thus for any
point $x$ in the good set $G$, $\zeta(x) = \us$ for some $\us\in \hZ$.
This shows that $\hZ = \Cl(\zeta(G))$. The last step needed to
show that $\hZ = B_k(p/q, \vnu)$ is to check that the address system
is of the type used in the HM construction.

We need only check that 
$|X_{2j+1}| = p/q$ and for this it suffices to show that
$\# X_{2j+1} = p/d$. Assume first that $\# X_{2j+1} < p/d$. 
Recalling that $\phi \sigma_k \phi^{-1}$ acts on the
$X_i'$ like $i \mapsto i + p/d$, we see that there will be some
$P_m\in X_{2j}'$ and $P_{m+p/d}\in X_{2j+2}'$. Thus using $\phi^{-1}$
there is a $\us\in \hZ$ with $s_0 = 2j$ and $s_1 = 2j+2$, a
contradiction to the fact that $\hZ\subset \Omega_k$ and 
thus its allowable transitions are given by \eqref{transitions}. On the other
hand, if $\# X_{2j+1} > p/d$ we have some 
$P_m\in X_{2j+1}'$ and $P_{m+p}\in X_{2j+1}'$ again yielding
a contradiction to $\hZ\subset \Omega_k$.

The irrational case is basically a continuous version of
the rational one. By Remark~\ref{rk3}(b) we have a continuous, onto 
$\phi:\hZ\raw S_k$ which is weakly order preserving,
$\phi\sigma_\infty = R_\alpha \phi$, and $\#\phi^{-1}(x) = 1$
for all but a countable number of $R_\alpha$-orbits on which 
$\#\phi^{-1}(x) = 2$.

For $j= 0, \dots, k-1$, let $X_j = \phi([j])$. Thus $X_j$ is a closed interval
(perhaps empty) with $\cup X_j = \R$, $X_j \leq X_{j+1}$ and
adjacent intervals intersect only in their single common boundary point.
We use $\{ X_j\}$ as an address system with dynamics $R_\alpha$,
good set $G$, and itinerary map $\zeta$. By construction if
$\us\in \hZ$ with $\phi(\us)\in G$, then $\us = \zeta\phi(\us)$ and
so $\phi^{-1}(G) = \zeta\phi(\phi^{-1}(G)) = \zeta(G)$. Since $\hZ$
is a Cantor set and $\phi^{-1}(G)$ is $\hZ$ minus a countable set of
 $\sigma_k$-orbits we have that $\phi^{-1}(G)$ is dense $\hZ$.
Thus taking closures, $\hZ = \Cl(\zeta(G))$.

To finish we must show that $\{ X_j\}$ is the type of address system
allowable in the HM-construction. 
We just need $|X_{2j+1}| = \alpha$ for all $j$. The proof is similar
to the rational case. If $|X_{2j+1}| < \alpha$ then $\hZ$ has a transition
$2j \raw 2j+2$ and if $|X_{2j+1}| > \alpha$ then $\hZ$ has a transition
$2j+1 \raw 2j+1$. Either is a contradiction to $\hZ\subset\Omega_k$.
Thus letting $\nu_{j+1} = |X_{2j}|$ for $j = 0, \dots, k-1$, we have
$\hZ = \Cl(\zeta(G)) = B_k(\alpha, \vnu)$.

The last sentence in (c) follows from the construction of 
$\lambda_k(\omega, \vnu)$.
\end{proof}

\begin{remark}\label{unique1}
In Section~\ref{ss} below we shall see
that for the irrational case $\rho(\hZ) = \omega\not\in\Q$
that there is a unique $\vnu$ with $\hZ =  B_k(\omega, \vnu)$ and
for rational $p/q$ that there are, in general, many $\vnu$
with $Z = B_k(p/q, \vnu)$. But note
that if $\hZ$ is a single periodic orbit then the proof above produces
what we show is the unique $\vnu$ with $\hZ = B_k(p/q, \vnu)$. 
\end{remark}

\section{Parameterization of $\cS_k(g)$ and $\cN_k(g)$ by the HM
construction}

We know from Theorem~\ref{Bkstruc}(c) that the HM construction
 yields a correspondence
between sets $B_k(\omega, \vnu)$ and  symbolic kfsm set  in $\Omega_k$.
In addition, for a map $g\in \cG$  using
Theorem~\ref{symandreal} we get a bijection from
 kfsm sets in $\Lambda_k(g)$ to those in $\hLambda_k(g)\subset\Omega_k$.
Thus the HM construction provides a parameterization of $\cS_k(g)$.
 In this section we examine this parameterization as well as that
of $\cN_k(g)$ in detail.

\subsection{Resonance and holes}\label{holesec}
As commented on above, the collection of all kfsm sets is closed
in the compact metric space consisting of all compact $g$-invariant
sets with the Hausdorff topology. Thus the collection of all kfsm sets is 
complete. We have restricted attention here to
recurrent kfsm sets or $\cS_k(g)$. This is because the recurrent ones
are the most dynamically interesting  and carry the invariant
measures, but also as shown in 
Theorem~\ref{Bkstruc}, they are what is parameterized by the HM
construction. As a consequence our primary space of interest
$\cS_k(g)$ is not complete, but rather has holes at points
to be specified. What happens roughly is that as one takes 
the Hasudorff limit of recurrent kfsm sets the resulting kfsm set has homoclinic
points that are not recurrent and so the limit is not recurrent and
thus not any $B_k(\omega, \vnu)$. This is a phenomenon well known
in Aubry-Mather theory. Another  point of view on these
``holes'' is given in Section~\ref{neghom} using the family of interpolated
semi-monotone maps.

In the HM construction fix $0<k < \infty$. For a given allowable
  $(\omega, \vnu)$
recall the address intervals are $X_j = X_{j\vnu}$ for
 $j = 0, \dots, 2k-1$. Define
$\ell_j = \ell_{j\vnu}$ and $r_j = r_{j\vnu}$ by  
$[\ell_j, r_j] := X_j$. Note that $r_{j+1} = \ell_j$ with indices
reduced mod $2k$. 
\begin{definition}\label{resonantdef} 
%\begin{enumerate}[(a)]
%\item 
The pair $(\omega, \vnu)$ is 
called \textit{resonant} if for some $n>1$ and $j,j'$, 
$R_\omega^n(\ell_j) = \ell_{j'}$. 
A pair that is not resonant is called \textit{nonresonant}.

\end{definition}
\begin{remark}\label{isres} Note that for a rational $\omega=p/q$ all
$(p/q, \vnu)$ are resonant as are all $(\omega, \vnu)$ when
some $\nu_i = 0$. Also, for all $(\omega, \vnu)$ and $j$, 
\begin{equation}\label{bad}
R_\omega(\ell_{2j-1}) = \ell_{2j}
\end{equation}
which is the reason $n$ is restricted to $n>1$ in the definition.
\end{remark}
The next lemma locates the ``holes'' in the space of all symbolic
kfsm sets and thus in any $\hcS(g)$.
\begin{lemma}\label{holes} $\ $
\begin{enumerate}[(a)]
\item Assume $(\alpha, \vnu)$ with $\alpha\not\in\Q$ is
resonant. There exists a sequence $\vnu^{(i)}\raw \vnu$
and a nonrecurrent kfsm $Z$ with
$B_k(\alpha, \vnu^{(i)})\raw Z$ in the Hausdorff topology
on all compact subsets of $\Sigma^+_{2k}$.
\item Assume $(p/q, \vnu)$ with $p/q\in\Q$.
 There exists a sequence $\omega^{(i)}\raw p/q$
and a nonrecurrent kfsm $Z$  with
$B_k(\omega^{(i)}, \vnu)\raw Z$ in the Hausdorff topology
on all compact subsets of $\Sigma^+_{2k}$.
\end{enumerate}
\end{lemma}
\begin{proof} We suppress the dependence on $k$ to simplify notation.
For (a),  the resonance hypothesis implies that there are odd
$a$ and $b$ with $R_{\alpha}^n(X_{a\vnu}) = X_{b\vnu}$ for some $n>0$ where
we may assume $a<b$. Since $R_\alpha^n(r_{a\vnu}) = r_{b\vnu}$ 
by shrinking some $\nu_j$ for $a < j < b$ we obtain $\vnup$ and 
$x < r_{a\vnu}$ and arbitrarily close to it with $x\in G_{\alpha \vnup}$
and $R_{\alpha}^n(x) \in X_{b+1, \vnup}$. In this way we can obtain
sequences $\vnu^{(i)}\raw\vnu$ and
$x_i\nearrow r_{a\vnu}$ with $x_i\in G_{\alpha, \vnu^{(i)}}$
and $R_{\alpha}^n(x_i) \in X_{b+1, \vnu^{(i)}}$. Thus 
$$
\zeta_{\alpha,\vnu^{(i)}}(x_i) = .a\dots (b+1) 
\zeta_{\alpha,\vnu^{(i)}}(R_{\alpha}^{n+1}(x_i) 
$$ 

To simplify matters, assume that
 $R_{\alpha}(r_{b\vnu}) \in G_{\alpha, \vnu}$ ; more complicated resonances
are similar. Since $R_\alpha^{n+1}(x_i)\raw R_\alpha^{n+1}(r_{a\vnu}) =
R_{\alpha}(r_{b\vnu})$
using 
Remark~\ref{rk2}(b), 
$$
\zeta_{\alpha,\vnu^{(i)}}(x_i) \raw .a\dots (b+1) 
\zeta_{\alpha,\vnu}(R_{\alpha}(r_{b\vnu}) := \us.
$$
Passing to a subsequence if necessary, by the compactness of
 the collection of symbolic kfsm sets there is a kfsm $Z$ with 
$B_k(\alpha, \vnu^{(i)})\raw Z$ in the Hausdorff topology and
by its construction, $\us\in Z$. But $\us$ can't be recurrent since
by the resonance any length $(n+1)$ block in 
$\zeta_{\alpha,\vnu}(R_{\alpha} (r_{b\vnu})$ must start with
$a$ and end in $b$.

The argument for (b) is similar, but now the perturbation must be in
$\omega$ since if $\omega = p/q$ is fixed, $R_{\omega}^n(X_a) = X_a$
when $n = qk/\gcd(p,k)$ for all $\vnu$. Fix an $a$ and so 
$R_{p/q}^n(r_a) = r_a$. By increasing $\omega$ incrementally
we may find sequences $\omega^{(i)}\searrow p/q$ and $x_i\nearrow r_a$
with $x_i\in G_{\omega^{(i)}, \vnu}$ so that the initial length
$(n+1)$ block of $\zeta_{\omega^{(i)} \vnu}(x_i)$ is 
$a \dots a+1$. Thus if $\zeta_{p/q\vnu}(r_a + \epsilon) =
P^\infty$ for small $\epsilon$ then $$
\zeta_{\omega^{(i)},\vnu}(x_i) \raw .a\dots (a+1) P_2 P_3 \dots P_{n-1} P^\infty
:= \ut
$$ where $P = (a+1) P_2 P_3 \dots P_{n-1}$. As in the proof of (a)
passing to a subsequence if necessary, there is a kfsm $Z$ with 
$B_k(\omega^{(i)}, \vnu)\raw Z$ in the Hausdorff topology and
by its construction, $\ut\in Z$. But $\ut$ can't be recurrent since
any length $(n+1)$ block in 
$P^\infty$ must start and end with $a$.
\end{proof}

\subsection{continuity and injectivity}

In doing the HM construction 
the explicit dependence of  $A_k$ and $ B_k$  on
the pair $(\omega, \vnu)$ was included. However,  the elements of
the pair have the interdependence $\sum\nu_i = k(1-\omega)$ and so when we
treat $A_k$ and $B_k$ as   functions it is sometimes better to eliminate the
interdependence and treat them as functions of $\vnu$ alone, but the two variable
version will also continue to be useful.
Thus we sometimes overload the function $A_k$ and write 
$$
A_k(\vnu) = A_k(1-\sum\nu_i/k, \vnu)
$$
and similarly for $B_k$  and the measure valued map
$\lambda_k$. 
The collection of allowable parameters for each $k$ is then
$$
\cD_k = \{ \vnu\in\R^k\colon \nu_i\geq 0, \sum_{i=1}^k\nu_i \leq k \}.
$$
The set of HM parameters corresponding to symbolic kfsm sets for $g\in\cG$ 
is defined as
$$\HM_k(g) = \{\vnu\in\cD_k\colon B_k( \vnu)\subset \hcS_k(g)\}.$$
%= B_k^{-1}(\hcS_k(g)).$$

\begin{remark}\label{isparam} By Theorem~\ref{Bkstruc},
  $B_k:\HM(g)\raw \hcS_k(g)$ is surjective and so 
$\iota_k^{-1}B_k:\HM(g)\raw \cS_k(g)$ provides
a parameterization  of the positive slope kfsm recurrent 
sets of $g\in \cG$ and $(\iota_k^{-1})_* \lambda_k:\HM(g)\raw \cN_k(g)$ 
their invariant measures.
 \end{remark}

\noindent\textbf{Example: The model map} For the model map $f_m$, 
$\HM_k(f_m) = \cD_k$ since  $\hLambda_k(f) = \Omega_k$. 
\medskip

The first issue in understanding what the HM construction tells
us about $\cS_k(g)$ and $\cN_k(g)$ is to understand the nature
of the maps $B_k$ and $\lambda_k$. As indicated by Lemma~\ref{holes} 
 in the behaviour of the set-valued maps there is 
an essential distinction between  the 
resonance and nonresonance cases. 

\begin{theorem}\label{iotathm} Assume $g\in\cG$, for each $k>0$,
\begin{enumerate}[(a)] 
\item The map $(\iota_k)^{-1}\circ B_k:\HM_k(g)\raw\cB_k(g)$ is onto and
 further it is 
continuous at nonresonant values and discontinuous at resonant values.
\item The
map $(\iota_k)^{-1}_*\circ\lambda_k:\HM_k(g)\raw \cN_k(g)$  is 
a homeomorphism and
thus $\HM_k(g)$ is compact.
\end{enumerate}
\end{theorem} 

\begin{proof}
Since we know from Lemma~\ref{iotainduced} that $\iota_k$ and $(\iota_k)_*$ are homemorphisms
we only consider $B_k$ and $\lambda_k$. While these are functions
of $\vnu$ alone, for the proof it is clearer to resort to the 
two variable versions with the proviso that $\omega = 1 - \sum\nu_i/k$.
Note that we have already shown in Theorem~\ref{Bkstruc}
  that $\lambda_k$ and $B_k$ are onto 
$\hcN_k(g)$ and $\hcS_k(g)$.  
We will often need to include the explicit dependence 
of various objects on the variables, for example, $\ell_j(\omega, \vnu)$,
and we often suppress the dependence on $k$. 

We prove (b) first.
We first show $\lambda_k$ is continuous.
 For each $j = 1, \dots, 2k-1$ and $i\in\N$, 
let $\ell_j^{(i)}(\omega, \vnu) =
R_{\omega}^{-i}(\ell_j(\omega, \vnu))$.
 The first observation from the HM construction is that
\begin{equation}\label{obs}
|\ell_j^{(i)}(\omega, \vnu) - \ell_j^{(i)}(\omega_0, \vnu_0)|
\leq \|(\omega, \vnu) - (\omega_0, \vnu_0)\|_1
\end{equation}
For a length $N$ block $B = b_0 \dots b_{N-1}$ in $\Omega_k$, let 
$$
Y_B(\omega, \vnu) = \bigcap_{i=0}^{N-1} R^{-i}_{\omega} (\Intt(X_{b_i}(\vnu)))
$$
and so $x\in G_{\omega, \vnu} \cap Y_B(\omega, \vnu)$ implies 
that $\zeta_{\omega\vnu} (x)$ begins with the block $B$.
 Also by the HM construction, $\lambda_k(\omega, \vnu) ([B]) =
m( Y_B(\omega, \vnu))$ with $m$ Lebesgue measure on the circle.

Recall that the weak topology on $\Sigma_{2k}^+$ is generated
by the metric
\begin{equation*}\label{weakmetric}
d(\mu, \mu') = \sum_{i=1}^\infty \frac{|\mu([B_i]) - \mu'([B_i])|}{2^i}
\end{equation*}
where $\{B_i\}$ is some enumeration of the blocks in $\Sigma_{2k}^+$.
Since each $Y_B(\omega, \vnu)$ is a (perhaps empty) interval
 with endpoints some $\ell_i^{(j)}(\omega, \vnu)$,
\eqref{obs} implies that 
\begin{equation*}
|m(Y_B(\omega, \vnu)) - m(Y_B(\omega_0, \vnu_0))|
\leq 2 \|(\omega, \vnu) - (\omega_0, \vnu_0)\|_1.
\end{equation*}
Thus summing over blocks 
\begin{equation*}
d(\lambda_k(\omega, \vnu), \lambda_k(\omega_0, \vnu_0)|
\leq 2 \|(\omega, \vnu) - (\omega_0, \vnu_0)\|_1.
\end{equation*}
so $\lambda_k$ is continuous.

Since by definition in the HM construction, 
$\lambda_k(\omega, \vnu)([2j]) = \nu_{j+1}$, $\lambda_k$
is injective. Recall now that for the model map, $\HM(f_m) =
\cD_k$ which is compact. So $\lambda_k:\HM(f_m)\raw \hcN_k(f_m)$ is
a homeomorphism with image the set of all measures on recurrent
symbolic kfsm sets in $\Omega_k$. Thus, since $\HM_k(g)\subset \cD_k$
we have that $\lambda_k:\HM_k(g)\raw \hcN_k(g)$ is also a homeomorphism.
The compactness of $\hcN_k(g)$ was proved in Lemma~\ref{iotainduced}.

The proof of (a) is based on the following claim: $B_k$ is continuous 
at $(\omega_0, \vnu_0)$
if and only if for all $N$  there exits $\delta >0$ so that 
 $\|(\omega, \vnu) - (\omega_0, \vnu_0)\| < \delta$ implies that
for all blocks $B$ of length $\leq N$ we have $Y_B(\omega_0, \vnu_0)$
nonempty exactly when $Y_B(\omega, \vnu)$ is nonempty.

To prove the claim, first note that continuity is equivalent to the
following:
given $\epsilon > 0$ there exists  $\delta>0$ so
that $\| (\omega, \nu) - (\omega_0, \nu_0)\| < \delta$ implies that
for each $\us\in \zeta_{\omega\vnu}(G_{\omega \nu})$ there is a 
 $\ut\in \zeta_{\omega_0\vnu_0}(G_{\omega_0 \vnu_0})$ with $d(\us, \ut) < \epsilon/2$ and for each $\ut\in \zeta_{\omega_0 \nu_0}(G_{\omega_0 \nu_0})$
 there is a 
 $\us\in\zeta_{\omega\vnu}(G_{\omega \vnu})$ with $d(\us, \ut) < \epsilon/2$.
This implies that $\HD(\zeta_{\omega\vnu}(G_{\omega \vnu}), 
\zeta_{\omega_0 \nu_0}(G_{\omega_0 \nu_0})(G_{\omega_0 \vnu_0}))
< \epsilon/2$ and thus $\HD(B_k(\omega,\vnu), B_k(\omega_0,\vnu_0)) < \epsilon$.
 Since  $d(\us, \ut)$ is
small exactly when $\us$ and $\ut$ agree in a long prefix block $B$ and
$Y_B(\omega, \vnu) = \zeta_{\omega\vnu}^{-1}([B])$, the claim follows.

We show that $B_k$ satisfies the condition in the claim when
$(\omega, \vnu)$ is nonresonant. Given $N$  
 for $j = 1, \dots, 2k-1$ and $i = 0, \dots, N$  consider again
$\ell_j^{(i)}(\omega, \vnu) =
R_{\omega}^{-i}(\ell_j(\omega, \vnu)$. By the HM construction
we have $\ell_{2m}^{(n+1)}(\omega, \vnu) = \ell_{2m-1}^{(n)}(\omega, \vnu)$
for all $n$, $m$ and $(\omega, \vnu)$. By nonresonance at $(\omega_0, \vnu_0)$,
all the other $\ell_j^{(i)}(\omega_0, \vnu_0)$ are disjoint. Since by
\eqref{obs}  each
$\ell_j^{(i)}(\omega, \vnu)$ depends continuously on $(\omega, \vnu)$
and the endpoints of each $Y_B(\omega, \vnu)$ is some 
$\ell_j^{(i)}(\omega, \vnu)$, we may find a $\delta$ so that
$\| (\omega, \nu) - (\omega_0, \nu_0)\| < \delta$ implies that 
the $\ell_j^{(i)}(\omega, \vnu)$ are ordered around $S_k$ in
the same way and with the same gaps between them as the
$\ell_j^{(i)}(\omega_0, \vnu_0)$. This implies that for each block
$B$ of length $B\leq N$,  $Y_B(\omega_0, \vnu_0)$ is
nonempty exactly when $Y_B(\omega, \vnu)$ is nonempty and so 
$B_k$ is continuous.

For the discontinuity, since the sets $Z$ in Lemma~\ref{holes}(ab)
are not recurrent, they are not equal to $B_k(\omega_0, \vnu_0)$.

\end{proof}

\begin{remark} $\ $ 
\begin{enumerate}[(a)]
\item The parameter space $\cD_k$ is $(k-1)$-dimensional. Assuming
$\omega\not\in\Q$, then for a fixed $n>1$ and $j,j'$, the collection of
all $\vnu\in\cD_k$ which yield $R_\omega^n(\ell_j) = \ell_{j'}$
is a $(k-2)$-dimensional affine subspace. Thus the set of
resonance parameters is a countable  dense collection of codimension 
one affine subspaces and so the resonance case is full measure and
dense $G_\delta$

\item One can show that $B_k$ is lower semicontinuous 
\cite{BdDenjoy}, in particular  if $(\omega^{(i)}, \vnu^{(i)})
\raw (\omega, \vnu)$ and some subsequence of
$B_k(\omega^{(i)}, \vnu^{(i)})$ converges to $Z$ in
the Hausdorff topology, then $  B_k(\omega, \vnu) \subset Z$.
The Semi-Continuity Lemma (see page 114 of \cite{Choquet})
yields that a lower semi-continuous
set-valued function is  continuous on a dense
$G_\delta$ set. In the case of $(\iota_k)^{-1}\circ B_k$ the 
last theorem exactly identifies
this continuity set as the nonresonant $(\omega, \vnu)$.
\end{enumerate}
\end{remark}

\subsection{Slices and skewness}\label{ss}
Since $\rho$ and $\hrho$ are defined and continuous on
the various spaces $\cS_k(g)$ etc., we may define the
closed slices with a given rotation number. 
\begin{definition}\label{slicedef} For $g\in\cG$ let 
 $\cS_{k\omega}(g) = \{Z\in\cS_k(g)\colon \rho(Z) = \omega\}$ and 
$\hcS_{k\omega}(g) = \{Z\in\hcS_k(g)\colon \hrho(Z) = \omega\}$
and the restriction of $\iota_k$ to $\cS_{k\omega}(g)$ is denoted
 $\iota_{k\omega}$ 
The slices of invariant measures 
$\cN_{k\omega}(g)$ and $\hcN_{k\omega}(g)$ are defined similarly.
The $\omega$-slice of HM parameters is 
$\HM_{k\omega}(g) = B_k^{-1}(\cS_{k\omega}(g)) = 
\lambda_k^{-1}(\cN_{k\omega}(g))$
\end{definition}

\begin{definition}\label{puredef} For $p/q\in\Q$ an 
allowable parameter $\vnu$ is
called pure if $B_k(p/q, \vnu)$ consists of a single periodic
orbit. The collection of $p/q$ pure parameters is denoted 
$\Pure_{k, p/q}\subset\cD_{k,p/q}$ 
and it will  be shown in Lemma~\ref{latticelem} to be an affine lattice.
For a $g\in\cG$ its pure parameters are $ \Pure_{k, p/q}(g) 
= \HM_{k, p/q}(g)\cap \Pure_{k, p/q}$.
\end{definition}

\begin{remark}\label{unique2}
For a given symbolic kfsm $p/q$-periodic orbit $P$ by 
Theorem~\ref{Bkstruc}(c)
 there is some $\vnu$ with $B_k(p/q, \vnu) = P$. 
Since a periodic orbit is uniquely ergodic and $\lambda_k$ is
injective this $\vnu$ is unique. Thus there is a bijection
between symbolic kfsm $p/q$-periodic orbits and $\Pure_{k, p/q}$.
\end{remark}

\begin{lemma}\label{iota2lemma} Assume $g\in\cG$
\begin{enumerate}[(a)]
\item For all $\omega$, 
$(\iota_{k\omega}^{-1})_*\circ\lambda_{k\omega}:\HM_{k\omega}(g)\raw \cN_{k\omega}(g)$ is a homeomorphism.
\item When $\alpha\not\in\Q$,
 $\iota_{k\alpha}^{-1}\circ B_{k\alpha}:\HM_{k\alpha}(g)\raw \cS_{k\alpha}(g)$ 
is injective as well as  continuous
at nonresonant $(\alpha, \vnu)$ and discontinuous at resonant $(\alpha, \vnu)$.
\item When $p/q\in\Q$, 
 $\iota_{k p/q}^{-1}B_{k p/q}:\HM_{k p/q}(g)\raw \cS_{k p/q}(g)$ is injective on $\Pure_{k, p/q}$. 
\end{enumerate}
\end{lemma}
\begin{proof}
Since $\iota_k$ restricts to a homeomorphism on slices we only consider
$B_{k\omega}$ and $\lambda_{k\omega}$. Part (a) follows immediately from 
Theorem~\ref{iotathm}.

For (b), when $\alpha\not\in\Q$ the assignment of a semi-Denjoy kfsm set with
rotation number $\alpha$ to its unique invariant measure
yields a bijection $\cS_{k\alpha}(g)\raw \cN_{k\alpha}(g)$ and
$\hcS_{k\alpha}(g)\raw \hcN_{k\alpha}(g)$. Since by (a), $\lambda_{k\alpha}$ 
is injective, we have that $B_{k\alpha}$ is also. Continuity of
$B_{k\alpha}$ at nonresonant values on
 irrational slices follows directly from (a). 
Discontinuity at resonant values on irrational slices follows
from Lemma~\ref{holes}(a).

\end{proof}

\begin{remark} Since $\cS_{k p/q}(g)$ is a finite, set the continuity
of $\iota_{k p/q}^{-1}B_{k p/q}:\HM_{k p/q}(g)\raw \cS_{k p/q}(g)$ is
not particularly interesting, but we will remark on it in Section~\ref{subres}.
\end{remark}

The \textit{skewness} $\gamma(\mu)$ of a $\tg_k$-invariant measure in $S_k$
is the amount of measure in each fundamental domain. When its $j^{th}$ component
is large, roughly its $\tg_k$-orbits are moving slowly through $[j-1,j)$.
When we project to the base $S^1$ in the next section the skewness
will thus indicate how quickly orbits are moving of the  $j^{th}$ loop
of the kfsm set.
\begin{definition}\label{skewdef} Assume $g\in\cG$
\begin{enumerate}[(a)]
\item For $\eta\in\cN_k(g)$, $\gamma(\eta) = (\eta([0,1)), \eta([1,2)),
\dots, \eta([k-1,k))$
\item For $\heta\in\hcN_k(g)$, $\hgamma(\eta) = (\heta([0]\cup [1]), 
\heta([2]\cup [3])
\dots, \heta([2k-2]\cup [2k-1])$
\end{enumerate}
\end{definition}
Note that the skewness takes values in the unit simplex 
$\sum a_i = 1$, $a_i \geq 0$ and contains no information 
about the rotation number.
\begin{lemma}\label{skewlem} Assume $g\in\cG$
\begin{enumerate}[(a)]
\item $\hgamma \circ (\iota_k)_* = \gamma$
\item $\gamma(\lambda_k(\omega, \vnu)) = 
(\omega + \nu_1, \omega + \nu_2, \dots, \omega + \nu_k)/k$
\item For $\eta\in\cN_{k\omega}$, 
$\gamma_1(\eta) = k \gamma(\eta) - \omega \One$ is
inverse to $(\iota_k)_*^{-1} \circ \lambda_k$ and so it is a
homeomorphism. 
\item $\gamma$ is a homeomorphism from $\cN_{k\omega}(g)$
onto its image as is $\hgamma$ from $\hcN_{k\omega}(g)$ 
onto its image
\end{enumerate}
\end{lemma}

\begin{remark}\label{skewrk} The last lemma formalizes the description
in the Introduction on the \newline 
parametrization of the weak disks of
semi-Denjoy minimal sets by their speed in each ``loop'' around
the circle. For rational pure parameters the skewness counts the
number of elements in each fundamental domain and this thus 
yields a discrete parametrization of the kfsm $p/q$-periodic
orbits.
\end{remark}
\section{kfsm sets in $S^1$ and $\Omega_1$}
\subsection{In $S^1$}
We now return to our central concern, $g$-invariant sets in $S^1$ that
have a lift to $S_k$ that is semi-monotone. Once again the definition
makes sense for any degree one circle map but we restrict to the class
$\cG$.
\begin{definition}\label{kfsminbase} Given $g\in \cG$, 
a compact $g$-invariant set $Z\subset S^1$ is kfsm 
if it has a $\tg$-invariant lift $Z'\subset\R$ which is kfsm, or equivalently,
$Z$ has a $\tg_k$-invariant lift $Z^*\subset S_k$ which is kfsm. 
Let $\cC_k(g)$ be all compact, invariant, recurrent
kfsm sets in $\Lambda_1(g)$ with the Hausdorff topology
 and $\cO_k(g)$ be all $g$-invariant, Borel
probability measures supported on $Z\in \cC_k(g)$ with the weak topology
\end{definition}
Thus when $Z$ is kfsm, it has a lift to $S_k$ which is 
semi-monotone under the action of $\tg_k$ on its lift.
 
To make contact with the usual definitions in Aubry-Mather
theory, assume that  $x\in S^1$ is such that $o(x,f)$ is 
k-fold semi-monotone. This happens exactly when there
is a point $x'\in\R$ with $\pi_\infty(x') = x$ and for all positive
integers $\ell, m,n$, 
\begin{equation*}
\tg^\ell(x') < T^{k m} \tg^n(x')\ \ \text{implies} \ \ 
\tg^{\ell+1}(x') \leq T^m \tg^{n+1}(x')
\end{equation*}
In Aubry-Mather theory one would write $x_j = \tg^j(x')$.

\begin{remark}\label{lots} $\ $
\begin{enumerate}[(a)]
\item $\pi_k:S_k\raw S^1$ induces continuous onto maps
$\cB_k(g)\raw \cC_k(g)$ and $\cN_k(g)\raw \cO_k(g)$.
\item $Z^*\subset S_k$ is kfsm if and only if $\pi_k(Z^*)\subset S^1$ is.
\item If $Z\subset S^1$ is is a $k$-fold semi-monotone then it is also
$\ell k$-fold semi-monotone for any $\ell > 0$.
\item If $P$ is a periodic orbit of $g$ of type $(p,q)$ (which
are perhaps not relatively prime) then $P$ has
a lift $P'$ to $\R$ with $T^p(P') = P'$ and is monotone
since $g\in\cG$ implies $\tg(x') \geq x'$ and so 
$P$ is automatically $p$-fsm.  
\item Using Lemma~\ref{symandreal}  a  recurrent kfsm set in  $S^1$ is either a 
collection of periodic orbits all with the same rotation
number (a cluster) or else a semi-Denjoy minimal set.    
A  minimal kfsm set  $S^1$ is either  a single periodic orbit 
or else a semi-Denjoy minimal set.
\item A collection of periodic orbits all with the same rotation
number that individually are kfsm when considered as a set is not
of necessity a kfsm (i.e., a cluster) 
\end{enumerate}
\end{remark}

\subsection{Symbolic kfsm sets in $\Omega_1$}
We now consider symbolic kfsm sets in the symbolic base $\Omega_1 =
\Sigma_2^+$.
\begin{definition}\label{symkfsminbase}
A $\sigma_1$-invariant set 
$\hZ\subset \Omega_1 = \Sigma_{2}$ is kfsm
if there is a $\sigma_\infty$-invariant lift $\hZ'$ 
(i.e., $\hp_\infty(\hZ') = \hZ$) which is kfsm
or equivalently,
$\hZ$ has a $\sigma_k$-invariant lift $\hZ^*\subset \Omega_k$ which is
kfsm.  Given $g\in\cG$ let $\hcC_k(g)$ be all  recurrent
kfsm sets in $\hLambda_1(g)$ with the Hausdorff topology
 and $\hcO_k(g)$ be all $g$-invariant, Borel
probability measures supported on $\hZ\in \hcC_k(g)$ with the weak topology
\end{definition}

Using Theorem~\ref{symandreal} 
we connect kfsm sets in $\Lambda_1(g)$ to their symbolic
analogs in $\hLambda_1(g)$ and get
\begin{corollary}\label{basecor}
 A $g$ invariant set $Z\subset \Lambda_1(g)$ is kfsm
if an only if $\iota_1(Z)\subset \hLambda_1(g)$ is. Further,
$\iota_1$ induces homeomorphisms $\cC_k(g)\raw \hcC_k(g)$ and 
$\cO_k(g)\raw \hcO_k(g)$.
\end{corollary}

\begin{remark} All the comments in Remark~\ref{lots} hold for
symbolic kfsm sets \textit{mutatis mutandis}.
\end{remark}

\subsection{The HM construction and its symmetries}

We bring the HM construction back into play and take the projections
from $\Omega_k$ to $\Omega_1$.
\begin{definition}\label{Ckdef}
 Let $C_{k}(\omega, \vnu) = \hpi_k(B_{k}(\omega, \vnu))$ and
$\mu_{k}(\omega, \vnu) = (\hpi_k)_*(\lambda_{k}(\omega, \vnu))$
\end{definition}

We know from Theorem~\ref{iotathm}
 that the HM construction provides a parameterization
of $\cS_k(g)$ and $\cN_k(g)$, the goal now is to get a parameterization
of the kfsm sets and their invariant measures in $S^1$, i.e., of $\hcC_k(g)$
and $\hcO_k(g)$. For this we need to understand the symmetries
inherent in the HM construction.

Recall the left shift on the parameter $\nu$ is
$\tau(\nu_1, \dots, \nu_k) = (\nu_2, \dots, \nu_k, \nu_1)$.
There are two types of symmetries to be considered. The first is
when different $\vnu$ give rise to the same $C_{k}(\omega, \vnu)$.
For minimal  $C_{k}(\omega, \vnu)$ this happens if and only if 
the $\vnu$'s are shifts of each other as is stated
in parts (a) and (d) in the next theorem. The second sort of symmetry
happens when some $C_k(\omega, \vnu)$ is also a $C_j(\omega, \vnup)$
for some $j<k$, which is to say the map $\hpi_k:B_k(\omega, \vnu)
\raw C_k(\omega, \vnu)$ is not one-to-one. In the minimal case
this happens if and only if $\tau^j(\vnu) = \vnu$ as is stated
in parts (b) and (c) below.
\begin{lemma}\label{sym}
\begin{enumerate}[(a)] Fix $k>0$ and assume $\vnu$ is allowable for 
$\omega$.
\item  For all $j$, $B_k(\omega,\tau^j(\vnu)) = \hT_k^j( B_k( \omega,\nu))$
and so $C_k(\omega,\tau^j(\vnu)) = C_k( \omega,\nu)$ 
\item If $\tau^j(\vnu) = \vnu$ for some $0 < j < k$ then 
\begin{equation}\label{dumb}
B_k(\omega,\vnu) = \hT_k^j( B_k( \omega,\vnu))
\end{equation}
and $C_k(\omega,\vnu) = C_j(\omega,\vnup)$ where
$\vnup = (\nu_1, \dots, \nu_j)$.
\item If $B_k(\omega,\vnu)$ is minimal and
\eqref{dumb} holds then $\vnu = \tau^j(\vnu)$. 
If $\vnu \not= \tau^j(\vnu)$ for all $0 < j < k$,
then $\hpi_k:B_k(\omega,\vnu) \raw C_k(\omega,\vnu)$
is a homeomorphism.
\item If $B_k(\omega, \vnu)$ and $B_k(\omega, \vnup)$ 
are minimal and $C_k(\omega, \vnu)=C_k(\omega, \vnup)$,
then for some $j$, $\vnup = \tau^j(\vnu)$.
\end{enumerate}
\end{lemma}
\begin{proof}
The fact that $B_k(\omega,\tau^j(\vnu)) = \hT_k^j( B_k( \omega,\nu))$
is an easy consequence of the HM construction and since
$\hpi_k \hT^k = \hpi_k$ we have 
$C_k(\omega,\tau^j(\vnu)) = C_k( \omega,\nu)$, proving (a)
The first part of (b) follows directly from (a) using the 
given fact that $\tau^j(\vnu) = \vnu$.

For the second part of (b), first note that if $\{X_i\}$ is the
address system for $k$ and $(\omega,\vnu)$ then since $\tau^j(\vnu) = \vnu$,
we have $T_k^j(X_i) = X_{i + 2j}$. This implies that
under the quotient $S_k\raw S_j$,   $\{X_i\}$ descends 
to an allowable HM-address system on $S_j$ using $(\omega,\vnup)$.
Thus using the dynamics $R_\omega$ on both address systems,
the corresponding entries of $B_k(\omega, \vnu)$ and
$B_j(\omega, \vnup)$ are equal $\mod 2$  and so 
$C_k(\omega, \vnu)  = C_j(\omega, \vnup)$.

To prove the first part of (c), as remarked in 
Remark~\ref{rk3}, if $B_k( \omega,\nu)$ is
minimal it is uniquely ergodic. Thus if \eqref{dumb} holds, then 
$\lambda_k(\omega,\vnu) =  \lambda_k( \omega,\tau^j(\vnu))$ and since
$\lambda_k$ is injective by Theorem~\ref{iotathm}, 
$\vnu = \tau^{j}(\vnu)$. Now for
the second part of (c), certainly 
$\hpi_k:B_k(\omega,\vnu) \raw C_k(\omega,\vnu)$ is continuous and onto,
so assume it is not injective. Then there exists $\us, \ut\in 
B_k(\omega,\vnu)$ with $\us\not=\ut$ and $\hpi_k(\us) = \hpi_k(\ut)$.
Thus  for some $0 < j' < k$, $\ut = \hT_k^{j'}(\us)$ and
so if $j = k-j'$, $B_k(\omega,\vnu) \cap T_k^j B_k(\omega,\vnu)
\not=\emptyset$. But by assumption $B_k(\omega,\vnu)$ is minimal
and so $B_k(\omega,\vnu) = T_k^j B_k(\omega,\vnu)$ and
so $\vnu = \tau^j(\vnu)$, a contradiction. Thus 
$\hpi_k:B_k(\omega,\vnu) \raw C_k(\omega,\vnu)$ is injective, as required

For part (d), $C_k(\omega, \vnu)=C_k(\omega, \vnup)$ implies
that
\begin{equation*}
\cup_{i=1}^k T^i( B_k(\omega, \vnu)) =\hpi^{-1}( C_k(\omega, \vnu))
=\hpi^{-1}(C_k(\omega, \vnup)) = \cup_{i=1}^k T^i( B_k(\omega, \vnup)).
\end{equation*}
Since each $T^i( B_k(\omega, \vnu))$ and $T^i( B_k(\omega, \vnup))$
is minimal, for some $j$, 
$T^j( B_k(\omega, \vnu))= B_k(\omega, \vnup)$, and
so by part (c), $\tau^j(\vnu)  = \vnup$.
\end{proof}

\begin{remark}\label{clustermaybe}
 It is possible that if $B_k(p/q,\vnu)$ is a cluster
of periodic orbits, $\pi_k$ could be injective on
some of them and not on others.
\end{remark}

\subsection{continuity and injectivity}

Let $\bHMk = \HM_k(g)/\tau$ with equivalence classes denoted $[\vnu]$.
 Note that $\tau^j(\vnu)\in \cD_k$
for some $j$ is resonant if and only if $\vnu$ is, so we may call
$[\vnu]$ resonant or nonresonant. 

 Since the $\tau$ action
preserves slices we define $\bHM_{k\omega} = \HM_{k\omega}(g)/\tau$.
The $\omega$-slices of  $\cC_k(g)$ and $\cO_k(g)$ are defined in the obvious
way. If $(p/q, \vnu)$ is a pure parameter so is 
$\tau^j(\vnu)$ for any $j$ and so we define $\bPure(k, p/q) = \Pure(k,p/q)/\tau$.
Note that $\bPure(k,p/q)$ is all $[\vnu]$ such that $B_k(p/q, \vnu)$
is a single periodic orbit it is \textit{not} all $[\vnu]$ such
that  $C_k(p/q, \vnu) = \hpi_k B_k(p/q, \vnu)$ is a single periodic
orbit.

\begin{definition}\label{alphadef}
Lemma~\ref{sym} implies
that $(\iota_1)^{-1} \circ C_k$ induces a map
$\theta_k:\bHM_k(g)\raw \cC_k(g)$ and
that $(\iota_1)^{-1}_* \circ \mu_k$ induces a map
$\beta_k:\bHM_k(g)\raw \cO_k(g)$.  The induced maps on slices
are  $\theta_{k\omega}:\bHM_{k\omega}(g)\raw \cC_{k\omega}(g)$
and   
$\beta_{k\omega}:\bHM_{k\omega}(g)\raw \cO_{k\omega}(g)$
\end{definition}

\begin{theorem}\label{alphathm} Assume $g\in\cG$, for each $k>0$,
\begin{enumerate}[(a)] 
\item The map $\theta_k$ is onto, 
continuous at nonresonant values and discontinuous at resonant values.
Restricted to an irrational slice it is injective, continuous at 
nonresonant values and discontinuous at resonant values. 
Restricted to a rational slices it is injective on the pure lattice. 
\item The map $\beta_k$ is a homeomorphism when restricted to irrational
slices and pure rational lattices. 
\end{enumerate}
\end{theorem} 
\begin{proof} By construction we have the following commuting diagram
\begin{equation}\label{cdiag}
 \begin{CD}
 \HM_k(g) @>B_k>>\hcS_k(g) @>{\iota_k^{-1}}>>\cS_k(g)\\
 @V{\sim}VV    @VV{\hpi_k}V @VV{\pi_k}V\\
 \bHM_k(g) @>C_k>>\hcC_k(g) @>{\iota_1^{-1}}>>\cC_k(g)%\\
 \end{CD}
 \end{equation}
with the vertical maps all onto and continuous
and $\theta_k$ the composition of the bottom
horizontal maps and the map $C_k$ also denotes the map
induced on equivalence classes in $\bHM_k(g)$.
Since
the given versions of ${\iota_k}$ and $\iota_1$ are homeomorphisms
we need only consider $C_k$ and $\mu_k$. The
fact that these are continuous follows from Lemma~\ref{iota2lemma}
and  the just stated properties
of the diagram as do the various continuity assertions in
the theorem. We prove the discontinuity result for $C_k$ on
irrational slices. The other discontinuity assertions follow
similarly.

Assume $(\alpha, \vnu)$ is resonant with $\alpha\not\in\Q$.
 From Lemma~\ref{holes} and its proof  we
have a sequence $(\alpha, \vnu^{(i)}) \raw  (\alpha, \vnu)$ so
that $B_k(\alpha, \vnu^{(i)}) \raw  Z$ and a 
$\us\in Z\setminus B_k(\alpha, \vnu)$ with $\us$ nonrecurrent. 
In the quotients $[\vnu^{(i)}] \raw [\vnu]$ and $C_k(\alpha, \vnu^{(i)})
\raw \hpi_k(Z)$ by continuity. We need to show that
$\hpi_k(Z)\not= C_k(\alpha, \vnu)$  Now if $\pi_k(\us)\in 
\hpi_k(Z)\setminus C_k(\alpha, \vnu)$ we are done so assume
$\pi_k(\us)\in C_k(\alpha, \vnu)$. Thus for some $\ut\in B_k(\alpha, \vnu)$,
$\pi_k(\us) = \pi_k(\ut)$ and so by Lemma~\ref{fact1}(e), 
for some $j$, $\hT^j_k(\us) =
\ut$. This implies that the action of $\sigma_k$ on $\Cl(o(\us, \sigma_k))$
is conjugated to that on  $\Cl(o(\ut, \sigma_k))$ by $\hT_k^j$. 
But by the classification Theorem~\ref{Bkstruc},
 $\Cl(o(\ut, \sigma_k))$ is a minimal set
and thus so is $\Cl(o(\us, \sigma_k))$ and so $\us$ is recurrent, a
contradiction, yielding the discontinuity.

To show $C_k$ is injective on the sets indicated, assume $C_k(\omega, \vnu) =
C_k(\omega, \vnup)$ with either $\omega = p/q$ and $\vnu, \vnup$ in
the pure lattice or $\omega\not\in\Q$. In either case $B_k(\omega, \vnu)$
and $B_k(\omega, \vnup)$ are minimal and since $\hpi_k$ is a 
semiconjugacy, $C_k(\omega, \vnu)$ and $C_k(\omega, \vnup)$ are
also. Thus by Lemma~\ref{sym}(d), for some $j$, $\vnu = \tau^{j}\vnup$
and so $[\vnu] = [\vnup]$.

Now for part (b), there is a diagram similar to \eqref{cdiag} 
for $\beta_k$. Since Denjoy minimal sets and individual periodic
orbits are uniquely ergodic, the injectivity asserted for $\mu_k$
follows from that of $C_k$ just proved. Continuity and surjectivity follow
from the diagram and Lemma~\ref{iota2lemma}. 
\end{proof}

\begin{remark}\label{clustereg} We remark on the relationship 
of pure parameters to $C_k$ and $B_k$. As a short hand we indicate symbolic
periodic orbits by their repeating block.
 A simple computation shows that $B_2(2/5, (3/5, 3/5)) = 
01223\ \cup\ 00123$ and so $C_2(2/5, (3/5, 3/5)) = 01001$. Note
that as required $\hT_2( 01223) = 00123$ and  $01001$ is the $2/5$-Sturmian.
 Now $B_2(2/5, (4/5, 2/5)) =  00123$ and so
$C_2(2/5, (4/5, 2/5)) = 01001 = C_2(2/5, (3/5, 3/5))$. A further
computation shows that both  $\mu_2(2/5, (4/5, 2/5))$ and
$\mu_2(2/5, (3/5, 3/5))$ are the unique invariant measure on 
$01001$ and thus $\mu_2$ is not injective on rational
slices of $\bHM_k$ despite the fact that it is injective
on rational slices of $\HM_k$. The underlying explanation is that
being a pure parameter requires $B_k$ to be a single periodic
orbit not that $C_k$ be one.
\end{remark}

\begin{definition}\label{skewbas}
Let $\mathcal{Q}_k = P_k/\tau$ where $P_k\subset \R^{k+1}$ is the
standard $k$-dimensional simplex and $\tau$ is the shift.
Equivalence classes in $\mathcal{Q}_k$ are denoted $[\cdot ]$. For
$\heta\in \hcO_k(g)$ with $\rho(\heta) = \omega$ 
from Theorem~\ref{alphathm} we may find an $\vnu$ with 
$\heta = \mu_k(\omega, \vnu)$.
 The skewness of $\heta$ is defined
as $\bgamma(\heta) := [\gamma(\lambda_k(\omega, \vnu))]$.
Note that by Lemma~\ref{sym} this is independent of the choice of 
$\mu_k(\omega, \vnu)$. And also for $\eta\in\cO_k(g)$ via
$\bgamma(\eta) = \bgamma((\iota_k)_*(\eta))$.
\end{definition}

\begin{remark}\label{paraminv} On an irrational quotient 
slice $\cO_{k, \alpha}$,  let $\bgamma_1 = k \bgamma - \alpha\One$,
then $\bgamma_1$ is the inverse of $\beta_k$ and may be viewed
as a parameterization of $\hcO_{k, \omega}$ by skewness as in
Remark~\ref{skewrk}. Also as in that remark, skewness also
provides a parameterization of the quotient of the pure parameters.
\end{remark}

\subsection{Sturmian minimal sets, the case $k=1$}
We will need the special and much studied case of symbolic kfsm sets
for $k=1$.  When $k=1$ there is only one allowable choice for $\nu$, namely 
$\nu = 1-\omega$ and
so we write $C_1(\omega)$ for $C_1(\omega, 1-\omega) = B_1(\omega, 1-\omega)$.
 When $\omega$ is
rational $C_1(\omega)$ is a single periodic orbit and when
$\omega$ is irrational it is a semi-Denjoy minimal set.
These minimal sets (and associated sequences)
have significant historical importance and an abundance of 
literature (see \cite{sturm} for a survey). Their main importance here is as an indicator of when
a given number is in the rotation set.
\begin{definition}\label{sturmdef}
The minimal set $C_1(\omega)\subset\Sigma_2^+$ 
is called the Sturmian minimal set
with rotation number $\omega$. 
\end{definition}
To avoid confusion with the many definitions in the literature
we  note that here ``Sturmian'' refers to a minimal
set and not a sequence and it is subset of the \textit{one-sided} shift
$\Sigma_2^+$.  The next result is standard and we remark on one proof in
Remark~\ref{interprk}.
\begin{lemma}\label{sturmorder}
$\omega\in\rho(\lrk)$ if and only
if $C_1(\omega)\subset\lrk$.
If $0 \leq \omega_1 < \omega_2 \leq 1$, then in $\Sigma_2^+$,
\begin{equation*}
\min C_1(\omega_1) < \min C_1(\omega_2) <
\max C_1(\omega_1) < \max C_2(\omega_2).
\end{equation*}
\end{lemma}

\begin{definition}\label{vnudef}
For a fixed $k$, let $\vnu_s(\omega)$ be defined by
$(\vnu_s(\omega))_i = 1-\omega$ for  $i = 1, \dots, k$.
\end{definition}
\begin{remark}\label{sturmrk}
Since $\tau(\vnu_s) = \vnu_s$ if follows directly from Lemma~\ref{sym} that
for any $k$,  $C_k(\omega, \vnu_s) = C_1(\omega)$, the Sturmian minimal
set with rotation number $\omega$.
\end{remark}

\section{Structure of $\HM_k(g)$}
One obvious property of $\HM_k(g)$ is the symmetry $\tau(\HM_k(g)) = 
\HM_k(g)$ for all $k$. The full structure of $\HM_k(g)$ for a general
 $g\in\cG$  is quite complicated and will be saved for future papers. 
Here we focus on the structure near the diagonal in $\cD_k$.

\subsection{Irrationals on the diagonal}
We parameterize the diagonal $\Delta_k\subset\cD_k$
by $\omega$ using $\vnu_s(\omega)\in\Delta_k$ as defined
in the previous section and so 
$$
\Delta_k = \{ \vnu_s(\omega) \colon 0 \leq \omega \leq 1\}.
$$

For $g\in\cG$ the next result asserts that for 
each irrational $\alpha\in\Intt(\rho(g))$
there is some $\delta = \delta(\alpha)$ so that
the neighborhood $N_{\delta}(\vnu_s(\alpha))\subset
\HM_k(g)$. It gives the proof of Theorem~\ref{main}(a).

\begin{theorem}\label{mainone}
Assume $g\in\cG$ and  $k>0$ 
\begin{enumerate}[(a)]
\item $\HM_k(g) \cap \Delta_k = \{\vnu_s(\omega)\colon \omega\in\rho(g)\}$
\item If  $\alpha\not\in\Q$ with 
$\alpha\in\Intt(\rho(g))$, there exists a
$\delta>0$ so that $N_\delta(\vnu_s(\alpha))\subset \HM_k(g)$.
\item   If 
$\alpha\in\Intt(\rho(g))\setminus \Q$, then $\cO_k(g)$ 
contains a $(k-1)$-dimensional 
topological disc consisting of unique invariant measures
each supported on a member of
a family of $k$-fold semi-monotone semi-Denjoy minimal sets with
rotation number $\alpha$.

\end{enumerate}
\end{theorem} 
\begin{proof} Assume $\hLambda_1(g) = \lrk$.
 For (a) 
$C_k(\omega, \vnu_s(\omega)) = C_1(\omega)$, the Sturmian minimal set
with rotation number $\omega$ and from Lemma~\ref{sturmorder}, 
$C_1(\omega) \subset
\lrk$ if and only if $\omega\in\rho(\lrk) = \rho(g)$.

For (b) note that the pair $(\alpha, \vec{\nu}_s)$ is nonresonant.
We will first show
that  if $\alpha\in\Intt(\rho(g))$ then  there exists an $\epsilon >0$,
so that $\HD(C_k(\alpha, \vnu_s(\alpha)), C_k(\omega, \vnu)) < \epsilon$
implies $ C_k(\omega, \vnu) \subset 
 \lrk$. Pick $\alpha_1,\alpha_2\in\Intt(\rho( g))$  with 
$\alpha_1< \alpha <\alpha_2$. Thus by Lemma~\ref{sturmorder}, in $\Sigma_2^+$
\begin{equation*}
\ukappa_0 < \min C_1(\alpha_1) < \min  C_1(\alpha) < \max  C_1(\alpha) < 
\max  C_1(\alpha_2) < \ukappa_1
\end{equation*} and let 
\begin{equation*}
\epsilon = \min\{d(\min C_1(\alpha_1),  \min  C_1(\alpha)),
d(\max  C_1(\alpha), \max  C_1(\alpha_2))\}.
\end{equation*}
Thus $\HD(C_k(\alpha, \vnu_s(\alpha)), C_k(\omega, \vnu)) 
=\HD(C_1(\alpha), C_k(\omega, \vnu)) < \epsilon$ implies that
the compact, invariant set  $C_k(\omega, \vec{\nu})$ satisfies
$\min C_1(\alpha_1) < C_k(\omega, \vec{\nu}) < \max  C_1(\alpha_2)$
and so $ C_k(\omega, \vec{\nu}) \subset \lrk$.

Using the continuity of $C_k$ at nonresonant irrationals from
 Theorem~\ref{alphathm}(a) 
 there is a $\delta>0$ so that 
$\|(\omega,\vec{\nu}) - (\alpha, \vec{\nu}_s\| < \delta$ implies
$C_k(\omega,\vec{\nu}) \subset N_\epsilon(C_k(\alpha, \vec{\nu}_s)$,
 and so  $(\omega, \vnu)\in \HM(g)$.

Since $\tau(N_\delta(\vnu_s(\alpha))) = (N_\delta(\vnu_s(\alpha)))$,
the neighborhood descends to one in $\bHM_k(g)$ and 
$\beta_k$ is a homeomorphism
on irrational slices of $\bHM_k(g)$ (Theorem~\ref{alphathm}(b))
yielding (c).
\end{proof}

\section{Rational Slices}
In this section we study rational slices in the HM parameter
and in $\cS_k(g)$ and $\cC_k(g)$. As proved in Theorem~\ref{Bkstruc},
 each $B_k(p/q, \vnu)$
consists of a collection of periodic orbits all of period
$qk/\gcd(p,k)$ with the same rotation
number and as a set they are kfsm. Note that this is stronger than each
periodic orbit being individually kfsm. The invariant measure 
$\lambda_k(p/q, \vnu)$ is a convex combination of the unique measures
supported on each periodic orbit.

\subsection{periods in $\Omega_1$}
The next lemma examines how the periods of $B_k$ can change after
projection to $C_k$ via $\hpi_k$.

\begin{lemma}\label{symlem1} Fix $k>0$ and $p/q\in\Q$ and assume $\vnu$ is allowable for $p/q$.
If $\tau^j(\vnu) = \vnu$ with $0 < j \leq k$ and it
is the least such $j$, then the period of $C_k(p/q, \vnu)$
is $jq/\gcd(j,p)$.
\end{lemma} 

\begin{proof} Recall from Theorem~\ref{Bkstruc}
 that the period of $B_k(p/q, \vnu)$ is
$kq/\gcd(k,p)$. If $j=k$ by Lemma~\ref{sym}(c) 
$\hpi_k: B_k(p/q, \vnu) \raw C_k(p/q, \vnu)$ is injective and 
since $\sigma_1\hp_k = \hpi_k \sigma_k$, $B_k(p/q, \vnu)$ and 
$C_k(p/q, \vnu)$ have the same period. Now if $j<k$  by Lemma~\ref{sym}(b),
$C_k(p/q, \vnu) = C_j(p/q, \vnup)$ where $\vnup = (\nu_1, \dots, \nu_j)$.
Since $j$ is the least such, $\hpi_j:B_j(p/q, \vnup)\raw C_j(p/q, \vnup)$
is injective and $C_j(p/q, \vnup)$ has period $jq/\gcd(j,p)$.
\end{proof}

\subsection{the rational structure theorem}
The theorem in this section describes in more detail how the 
measures on $p/q$-kfsm vary with the parameter.

In the HM construction fix $k$,  $0 < p/q < 1$ with $\gcd(p,q) = 1$
and an allowable $\vnu$. We often suppress dependence on these choices and
so $R = R_{p/q}$, etc. Let $N =qk/\gcd(p,k)$ so $N$ is the period of 
$R$ acting on $S_k$. Recall the address intervals are 
$X_j = [\ell_j, r_j]$ for $j = 0, \dots, 2k-1$ and so 
$r_j = \ell_{j+1}$. The good set
is $G$ and the itinerary map is $\zeta$. When we write $\zeta(x)$ it is
implicitly assumed that $x\in G$.

The orbit of $0$, $o(0, R)$, partitions $S_k$ into $N$ pieces, each
of width $k/N = \gcd(p,k)/q$. Thus $J = [k-\gcd(p,k)/q, 1)$ is a fundamental
domain for the action of $R$ on $S_k$ in the sense that $S_k =
\cup_{i=0}^{N-1} R^i(J)$ as a disjoint union.
 Thus for each $0 \leq p \leq 2k-1$
there is a unique $0 \leq m < N$ with
 $\ell_{p}\in R^{m}(J)$, and then let $d_{p} = R^{-m}(\ell_{p})$. Note that
since $|X_{2j+1}| = p/q$, $d_{2j} = d_{2j-1}$ and that
all $d_{2j+1}$ as well as both endpoints of $J$ are \textit{not} in $G$.
Finally, for $j = 0, 1, \dots, 2k-1$ and $x\in J\cap G$, let 
$M_j(x) = \{0 \leq i < N\colon \zeta(x)_i = j\} = \{i \colon R^i(x) \in X_j\}$. 

\begin{lemma}\label{iff2}
\begin{enumerate}[(a)] Assume $x, x'\in J\cap G$.
\item $\zeta(x) = \zeta(x')$ if and only if $M_{2j+1}(x) = M_{2j+1}(x')$
for all $j = 0, \dots, k-1$. 
\item For each $j$,  $M_{2j+1}(x) = M_{2j+1}(x')$ if and only if $
x$ and $x'$ are in the same component of $J-\{d_{2j+1}\}$.
\item For each $k$, $\# M_{2j}(x) = \#M_{2j}(x')$ 
if and only if $x$ and $x'$ are
in the same component of $\Sigma\setminus\{d_{2j-1}, d_{2j+1} \}$ where 
$\Sigma$ is the circle $\Sigma = J/\mysim$ with $(k-1/N)\sim k$. 
\end{enumerate}
\end{lemma}
\begin{proof}
First note that both endpoints of $J$ are not in $G$ so they are
out of consideration for $x$ and $x'$  in what follows.

For (a) one implication is obvious. For the other, it suffices to  
we show that the collection of  $M_{2j+1}(x)$  determines 
$\us = \zeta(x)$. By Remark~\ref{rk3} we know that $\us\in\Omega_k$ 
and so its
one step transitions are governed by \eqref{transitions}. If $s_i = 2j+1$ then
$s_{i+1} = 2j+2 $ or $2j+3$ and we know which depending on whether
$i+1\in M_{2j + 3}$ or not. Similarly, if $s_i = 2j$ then $s_{i+1}$
is determined by whether $i+1\in M_{2j + 1}$ or not. Thus
$\us$ is determined completing the proof of (a).

For (b), first note that 
$|X_{2k-1}| = p/q$ and $(p/\gcd(p,k)) (k/N) = p/q$. Thus $X_{2k-1}$
is exactly filled with  $p/\gcd(p,k)$
iterates of $J$ with disjoint interiors.
 Thus $M_{2k-1}(x) = M_{2k-1}(x')$ for all $x\in J\cap G$.
Thus we only consider $0 \leq j < k-1$.
If $i$ is such that $R^i(J)\subset X_{2j+1}$, then $i\in M_{2j+1}(x)$
for all $x\in J$ and if $R^i(\Intt(J))\cap X_{2j+1} = \emptyset$ 
then $i\not\in M_{2j+1}(x)$ for all $x\in J\cap G$.
If $\ell_{2j+1}\in R^i(\Intt(J))$, then for 
$x>d_{2j+1}$ in $J$, then $i\in M_{2j+1}(x)$ 
and for $x<d_{2j+1}$, $i\not\in M_{2j+1}(x)$.
The last situation to consider is $r_{2j+1}\in R^{i}(\Intt(J))$.
Since $|X_{2k+1}| = p/q$, we have $R^{i}(d_{2j+1}) = r_{2j+1}$ and
so then for $x>d_{2j+1}$ in $J$, then $i\not\in M_{2j+1}(x)$ 
and for $x<d_{2j+1}$, $i\in M_{2j+1}(x)$, completing the proof
of (b).

If $d_{2j-1} = d_{2j+1}$ every $x\in\Sigma\setminus\{d_{2j-1}\}$ 
as the same number of indices in $M_{2j}$, so assume that
$d_{2j-1} < d_{2j+1}$ with the other inequality being similar.
If $i$ is such that $R^i(J)\subset X_{2j}$ then $i\in M_{2j}(x)$ for
all $x\in J$. 
If $i$ is such that $R^i(J)\cap X_{2j} = \emptyset$ then 
$i\not\in M_{2j}(x)$ for
all $x\in J$.
If $i$ is such that $\ell_{2j}\in R^i(J)$ then $i\in M_{2j}(x)$ 
if and only if $x > d_{2j-1}$ in $J$.
If $i$ is such that $\ell_{2j+1} = r_{2j}\in R^i(J)$ then $i\in M_{2j}(x)$ 
if and only if $x < d_{2j+1}$ in $J$, finishing the proof.
\end{proof}

\begin{corollary}\label{mainratthm}
If the nonempty connected components of $\Sigma \setminus \cup_{j=0}^{k-1}
\{d_{2j+1}\}$ are $K_1, \dots, K_m$, then $B_k(p/q, \nu)$ consists of
exactly $m$ distinct periodic orbits $P_1, \dots, P_m$ with 
$\zeta(x)\in P_j$ if and only if $x\in o(K_j, R)$. Further,
$\lambda_k(p/q,\nu) = \sum  N |K_j| \delta_j$ with 
$\delta_j$ the unique invariant probability measure supported in $P_j$
and $N = qk/\gcd(p,k)$.
\end{corollary}

\begin{proof}
As noted above, $J$ is a fundamental domain for the action of
$R$ on $S_k$ and so it it suffices to study $\zeta(x)$ for $x\in J$.

Combining Lemma~\ref{iff2}(a) and (b) we have that for $x\in \Sigma$,
$\zeta(x) = \zeta(x')$ if and only if $x$ and $x'$ are in the same component
$K_j$. Further, using Lemma~\ref{iff2}(c), $\zeta(x)$ and $\zeta(x')$ 
can be on  the same $\sigma$-orbit if and only if they are 
in the same component $K_j$, proving the first sentence of
the theorem.  The second sentence follows from the definition
of $\lambda_k$,  the fact that $S_k = \cup_{i=1}^N R^i(J)$, and
$R$ preserves Lebesgue measure.
\end{proof}

\subsection{The pure lattice and the structure of $\HM_{k p/q}$}
We now describe the pure affine lattice in more detail with
an eye towards counting the number of $p/q$-periodic kfsm sets.
For this a new method of specifying the address system in $S_k$
 will be useful. We fix a $k$ and an $\omega = p/q$ and sometimes
suppress dependence on them

 Recall that a
pair $(p/q, \vnu)$ specifies an address system
$\{X_j(p/q, \vnu)\}$ with each $X_j(p/q, \vnu) = [\ell_j, r_j]$. 
For each $i = 1, \dots, k-1$ let $\delta_i$ be 
the signed displacement of the address system from its totally
symmetric position given by $(\omega, \vnu_s(\omega))$. Thus
\begin{equation}\label{delta}
\delta_i( \vnu) = (\nu_1 + \dots \nu_i) - i(1-\omega).
\end{equation}
Since in the HM-construction $X_{2k}$ is fixed for all
$\vnu$, the vector $\vdelta(\omega)$ is $(k-1)$-dimensional
and so $\vdelta:\cD_{k,p/q} \raw \vdelta(\cD_{k,p/q})$ is
an affine map from the simplex $\sum \nu_i = k (1-p/q)$ to
a subset of $\R^{k-1}$. Note that $\vdelta(\vnu_s) = \vec{0}$.

\begin{lemma}\label{latticelem} Given $k$ and $ p/q$ there exists
a $\veta$ with $\|\veta\|_\infty \leq  \gcd(p,k)/(2q)$ so that
$\vnu\in\cD_{k,p/q}$ is a pure parameter for $p/q$ if and only if 
$\vdelta(p/q, \vnu) \in \veta +  (\gcd(p,k)/q) \Z^{k-1}$ in 
$\vdelta(\cD_{k,p/q})$. 
\end{lemma}
\begin{proof}
The structure theorem Theorem~\ref{mainratthm} implies that $B_k(p/q, \vnu)$ is a single periodic
orbit if and only if no $d_{2j-1}$ is in the interior of $J$.
This happens if and only if all $\ell_{2j-1}$ are contained
in $o(0, R_{p/q})$. Now $o(0, R_{p/q})$ divides $S_k$ evenly into
subintervals of length  $\gcd(p,k)/q$. For each $j = 1, \dots, k-1$
let $m_j$ be such that $R^{m_j}_{p/q}(0)$ is the point on
$o(0, R_{p/q})$ that is closest to $\ell_{2j-1}$ and define 
$\eta_j =  \ell_{2j-1} - R^{m_j}_{p/q}(0)$. Thus
$\|\veta\|_\infty \leq \gcd(p,k)/(2q)$ and $\vnu$ is pure
if and only if $\phi(\vnu) \in \veta +  (\gcd(p,k)/q) \Z^{k-1}$.
\end{proof}  
\begin{definition}\label{latticedef}
The set $L = \veta +  (\gcd(p,k)/q) \Z^{k-1}\cap \vdelta(\cD_{k,p/q})$
is called the $p/q$-pure affine lattice as is  its pre-image 
$\vdelta^{-1}(L)\subset\R^{n-1}$
\end{definition}

\subsection{Sub-resonance and the size of clusters}\label{subres}
\begin{definition} When $\omega=p/q$, the pair $(p/q, \nu)$ is 
called \textit{sub-resonant} if for some $qk/\gcd(p,k)>n>1$ and $j\not=j'$, 
$R_\omega^n(\ell_j) = \ell_{j'}$. 
\end{definition}
It follows from Theorem~\ref{mainratthm} that the number of 
sub-resonances in a $(p/q, \vnu) \in \cD_{k, p/q}$ controls
the number of distinct periodic orbits in a cluster $B_k(p/q, \vnu)$
with no sub-resonance corresponding to $k$ distinct periodic orbits
and all the $\ell_j$ on a single $R_{p/q}$ orbit   
corresponding  to $B_k(p/q, \vnu)$ being a single periodic orbit
so $(p/q, \vnu)$ is a pure parameter

The set sub-resonance parameters is a finite collection of codimension 
one affine subspaces in $\cD_{k, p/q}$ and thus the 
no sub-resonance case is open, dense and full measure in $\cD_{k, p/q}$.
Thus in $\HM_{k, p/q}$ the typical parameter corresponds to 
a cluster of $k$ periodic orbits. It also follows that the
assignment $\vnu\mapsto B_k(p/q, \vnu)$ restricted to 
$\HM_{k, p/q}$ is constant and thus
continuous on connected components of the no sub-resonance parameters
and is discontinuous at the sub-resonance parameters.

\subsection{Estimating the number of $p/q$ kfsm sets}\label{ratest}
For a given $g\in \cG$ the number of points from the pure 
$p/q$-lattice $\Pure_{k, p/q}$ contained in $\HM_{k,p/q}(g)$
tells us how many distinct periodic orbits there are in 
$\hcB_k(g)$. So by Lemma~\ref{sym} it yields  how many distinct periodic
$p/q$-kfsm sets $g$ has. We get an estimate for this number using
the continuity properties of $B_k$ from Theorem~\ref{iotathm}
 and the relationship
of kfsm sets in $S_k$ to those in $S^1$. The next result proves
Theorem~\ref{main}(b). 
\begin{theorem}\label{ratgrowthyhm}
If  $\alpha\in\Intt(\rho(g))$, $\alpha\not\in\Q$, $k>0$ and
$p_n/q_n$ is a sequence of rationals in lowest terms with
$p_n/q_n \raw \alpha$, then there exists a $C>0$ so that 
for sufficiently large $n$ the number of distinct 
periodic  $p_n/q_n$ kfsm sets in $\Lambda_1(g)$ is greater than or equal to
$C q_n^{k-1}$.
\end{theorem}

\begin{proof}
By Theorem~\ref{iotathm}(b) there is an $\epsilon_1 > 0$ so that 
$N_{\epsilon_1}(\vnu_s(\alpha))\subset \HM_k(g)$ where recall that
$\vnu_s(\alpha)$ is the Sturmian $\vnu$ for $\alpha$ on the diagonal
of $\cD_k$.
Since $\vdelta$ is a homeomorphism there is an $\epsilon$-ball $H$ 
in the max norm with $\epsilon>0$ about 
$\vec{0}$ in $\vdelta(\cD_k)$ with $\vdelta^{-1}(H) \subset \HM_k(g)$. Thus
if $|p_n/q_n - \alpha| < \epsilon$ there is a $\epsilon$-ball
in the max norm, i.e., a $(k-1)$-dimensional hypercube $H_1$, about
$(p_n/q_n, \vec{0})$ in $\vdelta(\cD_{k,p_n/q_n})$ with
$\vdelta^{-1}(H_1) \subset \HM_{k,p_n/q_n}(g)$.

We next estimate  the number of pure resonance $\vnu$ in
$H_1$. By Lemma~\ref{latticelem}, the pure $\vnu$ form an affine lattice 
with linear separation $\gcd(p_n, k)/q_n$.
Thus for  $p_n/q_n$ close enough
to $\alpha$, the number of lattice points in $H_1$
is larger than 
\begin{equation*}
\left( \frac{\epsilon q_n}{ \gcd(p_n, k)}\right)^{k-1}
\geq \left(\frac{\epsilon q_n}{ k}\right)^{k-1}.
\end{equation*}
since $\gcd(p_n, k) \leq k$. Thus since $\vdelta$ is a homeomorphism
the same estimate  holds for the number of
pure lattice points in $\vdelta^{-1}(H_1)\subset \HM_{k, p_n/q_n}(g)$. By
Theorem~\ref{iotathm} this tells us how many distinct periodic $p_n/q_n$ are
in $\hcS_k(g)$ and thus in $\cS_k(g)$ by Theorem~\ref{symandreal}.

To project this estimate to kfsm sets in $S^1$, recall from Theorem~\ref{alphathm} 
that $\theta_k:\bHM_{k, p_n/q_n}(g)\raw \cC_{k, p_n/q_n}(g)$
is injective on the pure lattice. The projection
$\Pure(k, p_n/q_n) \raw \bPure(k, p_n/q_n)$ is at most $k$ to $1$ 
 and so  the number of distinct 
$p_n/q_n$ periodic orbits in $\cC(g)$ is greater than or equal
to 
$$
\frac{1}{k}\left(\frac{\epsilon}{k}\right)^{k-1} q_n^{k-1}.
$$
\end{proof}
\begin{remark}\label{ratperrk} For a pure $\vnu$ for $p/q$
 when $\tau^j(\vnu) = \vnu$ for some $0 < j < k$
the period of the $C_k(p/q, \vnu)$ counted in the theorem is 
$jq/\gcd(j, p)$ (Lemma~\ref{symlem1}). In the typical case of no such symmetry
the period is $kq/\gcd(k, p)$. So, for example, when $p$ and $k$ are 
relatively prime, the counted periodic orbit has rotation type $(pk, qk)$
and when $k$ divides $p$ the rotation type is $(p,q)$. By making
judicious choices of the sequence $p_n/q_n\raw \alpha$, 
one can control the rotation types of the counted periodic orbits.
\end{remark}

%\subsection{Misc questions}

%\begin{enumerate}
%\item partial order in one and two dim restricted to peridic kfsm
%\item Combinatorics of various sequences as generalizations of sturmians
%\item structure of lattice, which points alowed in cluster and geemeotry
%how element of cluster fit together
%\end{enumerate}

\section{Parameterization via the interpolated family of maps}\label{hall}
We return now to  the heuristic description in the introduction
 of kfsm sets via a family
of interpolated semi-monotone maps and prove a few results and connections
to the HM-parameterization. Since we are mainly developing a heuristic,
some details are left to the reader. In many ways this point
of view is better for studying kfsm sets while the HM construction
is better for measures. Initially the 
parameterization depends on the map $\tg\in\cG$ but using the model map
we will get a uniform parameterization.

\subsection{The family of k-fold interpolated maps for $g\in\cG$}
Fix $g\in\cG$ with preferred lift $\tg$. For 
$y\in [g(x_{min} + n), g(x_{max}+ n)]$ there is a unique
 $x \in [x_{min}+n, x_{max}+n]$
with $\tg(x) = y$. Denote this $x$ as $b_n(y)$ ($b$ for branch). 
Let 
$L_g = \tg(\min(\Lambda_\infty(g) \cap I_0))$ and 
$U_g = \tg(\max(\Lambda_\infty(g) \cap I_{-1}))$ with the $I_i$
as defined in Section~\ref{Idef}.
%\tg(\min\{\Lambda_\infty(g)\cap [0,1)\})$ and   
Note from the definition of the class $\cG$, $0 \leq L_g < U_g \leq 1$
and by equivariance, $L_g + j = \tg(\min(\Lambda_\infty(g) \cap I_{2j}))$
and $U_g + j= \tg(\max(\Lambda_\infty(g) \cap I_{2j-1}))$.

\begin{definition}
For $\vec{c}\in \R^k$ define $\tilde{c}\in\R^k$ via 
$\tilde{c}_j = c_j + j-1$ for $j = 1, \dots, k$.
\end{definition}

%We will need a generalized Cartesian product that is adapted to $S_k$.
%\begin{definition}\label{twistproddef}
%For $A\subset S_k$ and $k>0$  let $A^{(k)}\subset (S_k)^k$ be
%\begin{equation*}
%A^{(k)} = (A, T_k A , T_k^2 A \dots T_k^{k-1} A) = 
%(A,  A + 1 ,  A +2 \dots  A + (k-1))\mod k.
%\end{equation*} 
%If $\phi:A\raw B\subset S_k$ let $\phi^{(k)}:A^{(k)}\raw B^{(k)}$
%be 
%\begin{equation*}
%\phi^{(k)} = (\phi, T_k\circ\phi , T_k^2\circ \phi \dots T_k^{k-1}\circ \phi)
%\end{equation*}
%\end{definition}

Fix $k>0$. For $\vec{c}\in  [L_g, U_g]^{k}$ and for
$j = 1, 2, \dots, k$ define $\tH_{k \vec{c}}(x)$ on
 $[b_{-1}(0), b_{-1}(0) + k ]$ as
\begin{equation*}
	\tH_{k\vec{c}}(x) = \begin{cases}
		\tilde{c}_j & \text{when} \ x\in[b_{j-2}(c_j), 
b_{j-1}(c_j)] \\
		\tg(x) & \text{otherwise}
	\end{cases}
\end{equation*}
and extend to $\tH_{k\vec{c}}:\R\raw \R$ so that 
$\tH_{k\vec{c}}(x+k) = \tH_{k\vec{c}}(x) + k$. See Figure~\ref{figure1}.  
Next define $H_{k\vec{c}}:S_k\raw S_k$ as the descent of
$H_{k,\vec{c}}$ to $S_k$.

\medskip

\noindent\textbf{Example: The model map} For the model map $f_m$, $x_{min} = 0,
x_{max} = 1/2, L = 0, U = 1/2$ and $b_j(y) = (y+ 2j)/3$.

\medskip

Given a compact $Z\subset\Lambda_k(g)$, for $j = 0, \dots, k$  let 
$$
\ell_j'(Z) = \tg_k( \max\{Z\cap I_{2j-1}\}) - (j-1)\ \text{and}  \  
r_j'(Z) = \tg_k( \min\{Z\cap I_{2j}\}) - (j-1).  
$$  
If for some $j$ we have $\ell_j' < L_g$ let $\ell_j = L_g$
otherwise let $\ell_j = \ell_j'$. Similarly, 
 and if for some $j$, we have $r_j' > R_g$ let  $r_j = R_g$
 otherwise let $r_j = r_j'$. Not that these $r$'s and $\ell$'s are
unrelated to those in Section~\ref{holesec}.
\begin{theorem}\label{boxthm} Assume $Z\subset \Lambda_k(g)$ 
is compact and invariant.
The following are equivalent
\begin{enumerate}[(a)]
\item $Z$ is a kfsm set
\item For $j = 1, \dots, k$, $\ell_j(Z) \leq  r_j(Z)$. 
\item $Z\subset P(H_{k\vc})$ for
\begin{equation}\label{box}
\vc\in 
\prod_{j=1}^k [ \ell_j(Z), r_j(Z)] 
\end{equation} 
thus $(\tg_k)_{\vert Z} = (H_{k\vc})_{\vert Z}$.
\end{enumerate}
\end{theorem}
\begin{proof}
If for some $j$, $\ell_j(Z) >  r_j(Z)$ then $g$ restricted to $Z$ doesn't
preserve the cyclic order, and so (a) implies (b). (c) implies (a) since
invariant sets in nondecreasing maps are always kfsm. Finally,
(b) says that $Z\subset P(H_{k\vc})$ for $\vc$ in the given range.
\end{proof}
\begin{definition}\label{boxdef}
For $Z\in \cS_k(g)$, let 
$$\Boxx_g(Z) = \prod_{j=1}^k [ \ell_j(Z), r_j(Z)] 
$$
and so $\Boxx_g(Z)\subset [L_g, U_g]^{k}$.
\end{definition}

\begin{remark}\label{boxremark}
\begin{enumerate}[(a)]
\item Nothing in the theorem requires $Z$ to be recurrent. 
If it is, so $Z\in\cB_k(g)$,  from Theorem~\ref{Bkstruc}, 
$\iota_k(Z) = B_k(\omega, \vnu)$
where $\omega = \rho_k(Z)$ and 
$\vnu\in\cD_{k\omega}$.
\item When $Z$ is a periodic orbit or cluster, $\Boxx_g(Z)$ is $k$-dimensional.
When $Z$ is a periodic orbit cluster its box is 
equal to the intersections of the boxes of its constituent single periodic
orbits. 
\item When $Z$ is a semi-Denjoy minimal set contained in $P(H_{k\vc})$
recall that a tight flat spot of $Z$ is one for which both endpoints of a flat spot 
of $H_{k\vc}$ are in $Z$. The dimension of $\Boxx_g(Z)$ is the same 
as the number of \textit{loose} flat spots in $Z$. Since by Lemma~\ref{flat}(b), 
$Z$
cannot have $k$ tight flat spots, the dimension of $\Boxx_g(Z)$ is between
$0$ and $k-1$.
\item Note that in contrast, in the HM parameterization, each single periodic
orbit or semi-Denjoy minimal corresponds to just one point.
\item  When $\alpha\not\in\Q$, if $Z = \iota_k^{-1}(B_k(\alpha, \vnu))$
 from the HM construction
then the  number of loose
flat spots of $Z$ is the same as the number of resonances of $(\alpha, \vnu)$
i.e., $j\not= j'$ with $R_\alpha^N(\ell_j) = \ell_{j'}$ for some $n>1$
which is then the same as the dimension of $\Boxx_g(Z)$.

\end{enumerate}
\end{remark}
\subsection{Nonrecurrence and kfsm sets
 that hit the negative slope region}\label{neghom}
Throughout this paper we have assumed that the kfsm sets were
recurrent and avoided the negative slope region. In this
section we use the interpolated maps to motivate these assumptions.

Assume now that $Z$ is a kfsm set for some $g\in\cG$ and $Z$ contains points in the
negative slope region of $g$. It still follows that $Z$ is an invariant
set of some $H_{k\vc}$. Let $Z'$ be the maximal recurrent set in
$P(H_{k\vc})$. A \textit{gap} of $Z'$ is a
component of the complement
of $Z'$ that contains a flat spot of $H_{k\vc}$. In formulas, a gap
is an interval $
(\max\{Z'\cap I_{2j-1}\}, \min\{Z'\cap I_{2j}\})  
$
for some $j$.
 Since $g_k$ acting on  $Z$ is semi-monotone, $Z$ can contain
at most one point $p_j$ in the negative slope region within each gap.
%and 
%\begin{equation*}
%g_k(p_j) \in [\tg_k( \max\{Z'\cap X_{2j-1}\}),  
%\tg_k( \min\{Z'\cap X_{2j}\})] = [\ell_j, r_j]
%\end{equation*}.

There are two cases. In the first, which may happen for both rational
and irrational rotation number, for all $j'$ there is some 
$n$ so that $f^i(p_{j'})\not\in \{p_j\}$ for all $i\geq n$. This implies
that $H_{k\vc}^i(p_j)\in P(H_{k\vc})$ for all $i\geq n$ and so by Lemma~\ref{flat}(c),
there is an $n'$ so that $H_{k\vc}^i(p_j)\in Z'$ for all $i\geq n'$.
Thus in  this case negative slope orbits  add no additional 
recurrent dynamics. In Figure~\ref{neghomfig} the disks give part of a periodic
kfsm set and the squares show additional homoclinic points to this
kfsm set in the negative and positive slope region.

\begin{figure}[htbp]
\begin{center}
\includegraphics[width=0.4\textwidth]{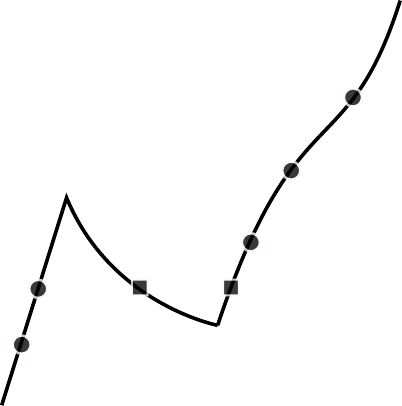}
\end{center}
\caption{A  semi-monotone set with homoclinic points}
\label{neghomfig}
\end{figure}

The second case holds for just rational rotation number and 
is when some $p_j$ is a periodic point; this does 
add new recurrent kfsm sets. Since the periodic
points on the endpoint of a gap always return, there is, in fact,
a periodic point in the negative slope region of each gap of $Z'$. By adjusting
$\vc$ we can assume that all these gap periodic orbits are also 
superstable periodic orbits of $H_{k\vc}$. Since the periodic points in 
$Z'$ are all unstable the periodic points of $Z'$ must alternate 
with these gap periodic points. In particular, the number of
gap periodic orbits equals the number of periodic orbits in $Z'$.
Thus the addition of the negative slope periodic orbits just adds a factor
of two to the basic estimates of Section~\ref{ratest}.

We again use Figure~\ref{neghomfig} but this time to discuss the holes in the
space of recurrent kfsm sets. Let $H_{c_0}$ be an interploted map 
whose flat spot contains the homoclinic points indicated by squares. 
Assume we are in the $k=1$ case and so each interpolated
map $H_c$ contains exactly one recurrent semi-monotone set $Z_c$ 
in its positive slope region.  
As $c_n$ increases to $c_0$  the sets $Z_{c_n}$ converge
in the Hausdorff topology to not just $Z_{c_0}$ but to that set
union the boxed point in the positive slope region.
A similar phenomenon happens as   $c_n$ decreases to $c_0$.
This phenomenon also clearly happens for loose gaps of semi-Denjoy
minimal sets and happens for all $k$. This is the geometric
explanation of the holes
in $\cB_k(g)$ and the discontinuity of $B_k$ discussed in 
Section~\ref{holesec}.

\subsection{The rotation number diagram }
Fix $g\in\cG$.
Since $H_{k\vc}:S_k\raw S_k$ we
define $R_k(\vc) = k \rho(D_k\circ H_{k\vc} \circ D_k^{-1})$.
Thus if $Z\subset P(  H_{k\vc})$ is compact invariant then 
  $\rho_k(Z) = R_k(\vc)$. We treat $R_k$ as a function
$R_k:[L_g, U_g]^{k}\raw \R$.

Let $\R_+^k = \{\vu\in\R_K \colon\ \text{all} \ u_i > 0\}$.
The open projective positive cone in $\R^k$ is 
$Q_k = \{\vu\in\R^k_+\colon  \|\vu\|_2 = 1\}$.
For a given $k, \omega$ define
 $\varphi_{-,\omega},\varphi_{+,\omega} :Q_k\raw \R^+$
as
\begin{align*}
\varphi_{-,\omega}(\vu) &= \min\{t\in\R^+\colon R_k(t\vu+\vL) = \omega\}\\
&\varphi_{+,\omega}(\vu)= \max\{t\in\R^+\colon R_k(t\vu+\vL) = \omega\}
\end{align*}
where $\vL = (L_g, L_g, \dots, L_g)$.
So $\varphi_{-,\omega}$ and $\varphi_{+,\omega}$
 give the top and bottom edges of
the level set $R_k^{-1}(\omega)$ when view from the origin.

\begin{theorem}\label{rhokthm} Assume $g\in\cG$ and construct $R_k$ as above 
\begin{enumerate}[(a)]
\item $R_k$ is continuous function and
is nondecreasing in $t$ along any line
$\vec{c} = t \vu + \vv$ 
 with all $v_i \geq 0$.
\item For all $\omega$ the functions
$\varphi_{-,\omega}$ and $\varphi_{+,\omega}$ are continuous.  
\item For rational $\omega$,  $\varphi_{-,p/q}< \varphi_{+,p/q}$ while for
$\alpha\not\in\Q$ then  $\varphi_{-,\alpha}=\varphi_{+,\alpha}$.
Thus each level set $R_k^{-1}(p/q)$ is homeomorphic to 
a $(k-1)$-dimensional open disk product a nontrivial closed interval while
each $R_k^{-1}(\alpha)$ is homeomorphic to 
a $(k-1)$-dimensional open disk.
\item $\rho(g) = [\rho(H_{L_g}), \rho(H_{U_g})] = \rho(\Lambda_1(g), g)$ 
\end{enumerate}
\end{theorem}
\begin{proof}
Part (a) follows directly from Lemma~\ref{basic}(b) and (c). 
For (b) assume to the contrary that
 $\varphi_-$ is not continuous. Then there is a sequence $\vu_n\raw\vu_0$
with $\varphi_-(\vu_n)\not\raw \varphi_-(\vu_0)$. Passing to a subsequence
if necessary, there is some $t_0$ with $\varphi_-(\vu_n) \vu_n + \vL 
\raw t_0 \vu_0 + \vL$. By the continuity of $R_k$,
$R_k( t_0 \vu_0 + \vL) = \omega$ and by the nonconvergence assumption,
there is some $t' < t_0$ with  $R_k( t' \vu_0 + \vL) = \omega$.
Thus again by the continuity of $R_k$ for $n$ large enough there is 
some $t''_n < \varphi_-(\vu_n)$ with  
$R_k( t''_n \vu_n + \vL) = \omega$, a contradiction.
Therefore, $\varphi_-$ is continuous: the continuity of 
$\varphi_+$ is similar.

For (c), pick any $t_0$ and $\vu_0$ with $R_k(t_0\vu_0 + \vL) = p/q$
and let $\vc = t_0\vu_0 + \vL$. Then by Lemma~\ref{flat}, 
$H_{\vc}$ has a periodic
orbit $Z\subset P(H_{\vc})$. Since $Z$ is a finite set there is a
nontrivial interval $I$ so that $t\in I$ implies 
$R_k(t\vu_0 + \vL) = p/q$ and so $\varphi_{-,p/q}< \varphi_{+,p/q}$.

To complete (c), assume to the contrary that for some $\vu_0$, 
$\varphi_{-,\alpha}(\vu_0) < \varphi_{+,\alpha}(\vu_0)$. Thus by the
continuity of $R_k$ there is an open ball $N\subset R_k^{-1}(\alpha)$.
Pick $\vc\in N$ and let $Z$ be the semi-Denjoy minimal set in $P(H_{\vc})$
guaranteed by Lemma~\ref{flat} which has at least one tight gap, say the gap 
associated with $c_1$ the first coordinate of $\vc$. Let $y$ be the
$x$-coordinate of the right hand endpoint of this gap and so $\tg_k(y) = c_1$.
Since $Z$ is minimal under $H_{\vc}$ there are points $z\in Z$ with
 $z > y$ and arbitrarily close to $y$ which have a $n>0$ with 
$y < H_{\vc}^n(z) < z$. Now let $c_1' = \tg_k(z)$ and 
$\vc' = (c_1', c_2, \dots, c_k)$ and we have that $H^n_{\vc'}(z) < z$
which says that the $n^{th}$ iterate of the first coordinate
flat spot of $H_{\vc'}$ is in that flat  spot. Thus $H_{\vc'}$ has a
periodic orbit and so $R_k(\vc') \not = \alpha$ for some $\vc'$ arbitrarily
close to $\vc$, a contradiction. 
 
For (d), assume $k=1$ and   $\rho(g) = [\rho_1, \rho_2]$. 
Let $H_T$ be the semi-monotone
map constructed from $g$ to have a single flat spot of height $\tg(x_{max})$
and $H_B$ similarly constructed having a single flat spot of height
 $\tg(x_{min})$. Since $H_T \geq \tg$,  we have $\rho(H_T) \geq \rho_1$. Now
by Lemma~\ref{flat}, there is a compact invariant $Z\subset P(H_T)$ and so
$g_{\vert Z} = (H_T)_{\vert Z}$ and so $\rho(H_T) = \rho(Z, g) 
\in \rho(g)$ and so $\rho(H_T) = \rho_1$. Similarly, $\rho(H_B) = \rho_2$.
 Note that by definition of $H_U$,
the compact invariant $Z\subset P(H_T)$  also satisfies  $Z\subset P(H_U)$
and so $\rho(H_T) = \rho(H_U)$. Similarly, $\rho(H_B) = \rho(H_L)$.
Thus  $\rho(g) = [\rho(H_L), \rho(H_U)]$. Finally, consider the entire
family $H_c$ for $c\in[L, U]$. Since $\rho(H_c)$ is continuous in
$c$, for each $\omega\in [\rho(H_L), \rho(H_U)]$ there is a $c$
with $\rho(H_c) = \omega$. Further, for each $c$ there is a compact
invariant $Z_c\subset P(H_c)$ and $Z_c\subset \Lambda_1(g)$ and
thus $\omega\in\rho(\Lambda_1(g))\subset \rho(g)$

\end{proof}

\begin{remark}\label{interprk} $\ $
\begin{enumerate}[(a)]
\item Note that $H_T(x) \leq x+ 1$  and $H_B(x)\geq x$ and 
thus $g\in\cG$ implies $ \rho(g) \subset [0,1]$. Further,
it follows from (d) that the image of each $R_k$ is $\rho(g)$. 

\item Part (b) deals only with the part of the level sets of 
$R_k$ in the open set $(L_g, U_g)^k$. The extension to all of
$[L_g, U_g]^{k}$ is technical and not very illuminating so we leave it
to the interested reader.
\item Let $\tau$ act on $\vc$ as the left cyclic shift.
It easily follows that $R_k(\tau(\vec{c})) = R_k(\vec{c})$.
\item When $k=1$ there is a one-dimensional family 
$H_c$ for $c\in[L_g,U_g]$. The rotation number $R_1(c)$ is nondecreasing
in $c$ and assumes each irrational value at a point and each
rational value on an interval by (a) and (c). For each $c$ there is
a unique recurrent  $Z_c\subset P(H_c)$ and 
$\iota_1(Z_c)$ is the Sturmian with the given rotation number.
This along with the geometry of the family $H_c$ give the proof of 
Lemma~\ref{sturmorder}.
\end{enumerate}
\end{remark}

\subsection{Comparing $g\in\cG$ to the model map}
In this section we use the interpolation parameter $\vc$ 
to parameterize the  $Z\in\cS_k(g)$ for a general $g\in \cG$.
Notice that for
the model map, $\Lambda_k(f_m)$ is all of $\Omega_k$. Thus 
 $\hcS_k(g) \subset \hcS_k(f_m)$ and so we can
pass back to $\cS(g)$ using the inverse of  the itinerary map.
Thus we can use a subset of the interpolation parameters 
of the model map to parameterize $\cS(g)$ using the symbolic
representation of a kfsm set as the link. This subset turns out to 
be a square of the form $[L', U']^k$. In this section we often add
an additional subscript of $f$ or $g$ to indicate which map 
$f_m$ or $g$ is involved

Since $\iota_{k,g}(\Lambda_k(g)) = \hLambda_k(g) \subset \Omega_k =
\hLambda_k(f)$ we may define  
 $\psi':\Lambda_k(g) \raw \Lambda_k(f)$ by 
$\psi' = \iota_{kf} \circ \iota_{kg}^{-1}$. By Theorem~\ref{conjugacy}
 $\psi'$ is 
an orientation preserving homeomorphism
onto its image as well as a conjugacy. It thus induces
a map $\bpsi:\cS_k(g) \raw \cS_k(f)$.

Recall that the 
parameters for the model map are $[L_{f}, U_{f}]^{k} = [0,1/2]^{k}$.
For a map $\phi: [a,b] \raw [a,b]$ extend it to the Cartesian
product as $\phi^{(k)} = (\phi, \phi, \dots, \phi)$.
\begin{theorem}\label{maininterpthm} 
Given $g\in\cG$  and $k>0$ construct the interpolation
parameters $[L_g,U_g]$.  There exists an interval $[L', U']
\subset [0, 1/2]$ and    
an orientation preserving homeomorphism
$\phi:[U_g, L_g] \raw [L', U']$ so that for all $Z\in \cS_k(g)$, 
 $  \phi^{(k)}(\Boxx_g(Z))) = \Boxx_f(\bpsi(Z))$ and for all $\omega\in\rho(g)$,
$\phi^{(k)}\rho_{k,g}^{-1}(\omega) = \rho_{k,f}^{-1}(\omega)$. Where
\end{theorem}

\begin{proof}
Construct $\psi'$ as above. Its properties imply that
\begin{equation}\label{psiprime}
\psi'(\ell_j(Z)) = \ell_j(\bpsi(Z)) \ \text{and} \ 
\psi'(r_j(Z)) = r_j(\bpsi(Z))
\end{equation}
for all $Z\subset \cS_k(g)$ and $j = 1, \dots, k$.
Let $L' = \psi'(L_g)$ and $U' = \psi'(U_g)$. Then $\psi'$ restricts
to  $\psi:\Lambda_k(g)\cap [L_g, U_g]^{k} 
\raw \Lambda_k(f)\cap [L', U']^{k}$. Since $\psi T_k = T_k \psi$
and $\Lambda_k(g)\cap [L_g, U_g]^{k}$ is compact we can extend
$\psi$ equivariantly to a homeomorphism
$\Psi :[L_g, U_g]^{k} \raw [L', U']^{k}$ which using \eqref{psiprime}
satisfies $\Psi\circ \Boxx_g = \Boxx_f\circ\bpsi$. Finally, 
since $\Psi\circ\tau = \tau\circ\Psi$ (recall $\tau$ is the left 
cyclic shift) there is a $\phi:[U_g, L_g] \raw [L', U']$ with
$\Psi = \phi^{(k)}$.
\end{proof} 
Thi result implies that the $\rho_k$-diagram for 
$g$ looks like $k$-dimensional cube cut from inside the
$\rho_k$-diagram of the model map and perhaps rescaled.

\subsection{The case $k=2$: numerics}

Figure~\ref{rect5} shows the $k=2$ rotation number diagram for the 
model map $f_m$. Each connected union of rectangles is the level
set of some rational. The rationals with denominator less than $6$ 
are shown. Only the center rectangle is labeled for each rational.
In the figure each rectangle corresponds to a different $2$-fold
semi-monotone periodic orbit. The intersections of these rectangles
correspond to $H_{\vc}$ which have a cluster of two periodic orbits

\begin{figure}[htbp]
\begin{center}
\includegraphics[width=0.5\textwidth]{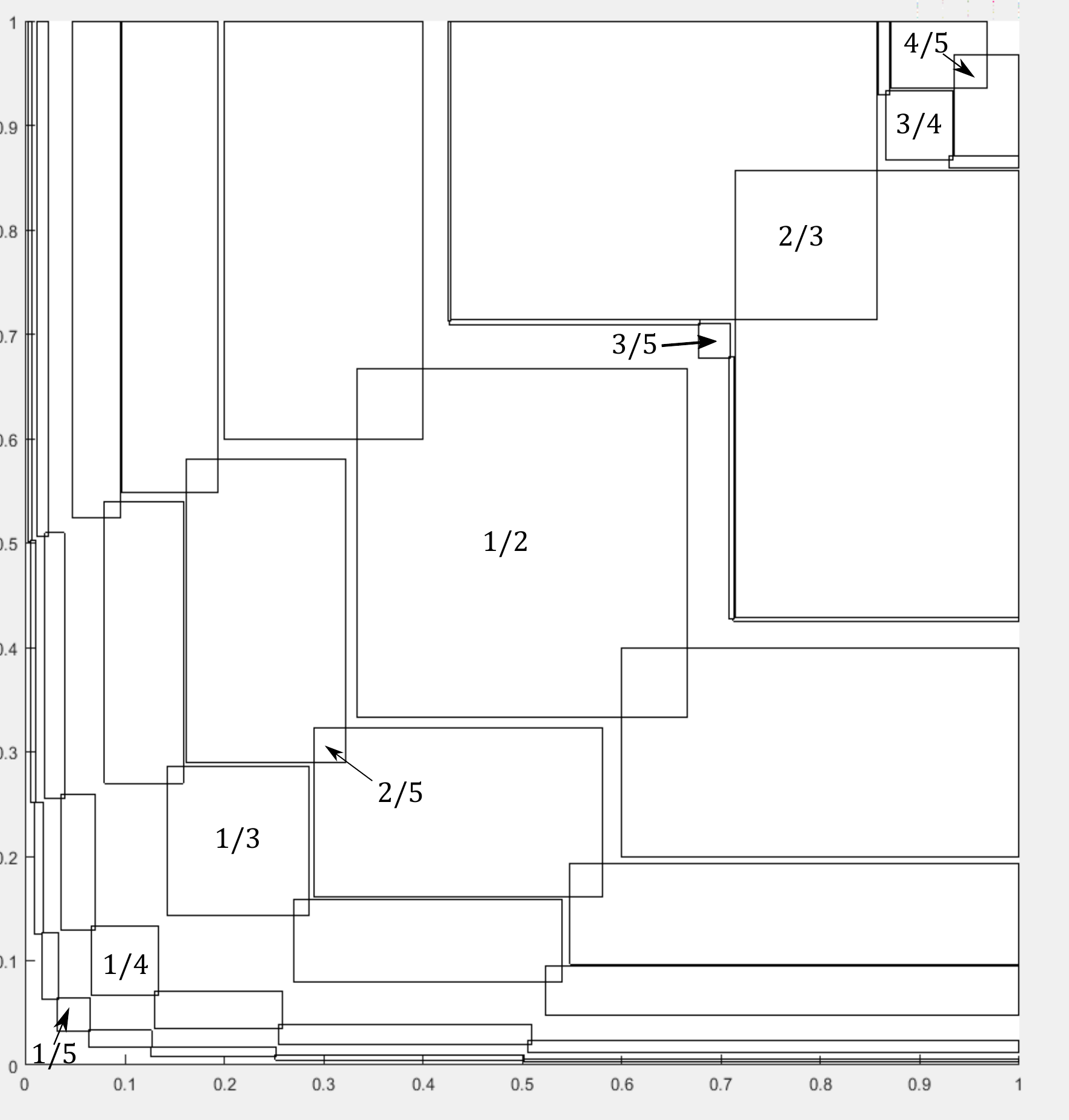}
\end{center}
\caption{The rotation number diagram for the model map with $k=2$.
Figure has been reparameterized for clarity.}
\label{rect5}
\end{figure}

The computation of this diagram
 used a discrete version of the HM construction. The construction
depends on integers $p, q, \mu$ with $0 < p/q < 1$, $p$ and $q$ relatively
prime, and $0 \leq \mu \leq 2(q-p)$. The discrete circle is 
the finite cyclic group 
$\Z/2q\Z = \Z_{2q}$ and it is acted on by $R_p: n \mapsto n+p$. The
address intervals are $X_0' = [1, \mu], X_1' = [\mu + 1, \mu + p],
X_2' = [\mu + p + 1, 2q - p]$, and $I_3 = [2q - p + 1, 2(q-p)]$. 
Let $B'(p, q, \mu)$ be the itinerary
of the point $1$ under $R_p$.

Using Theorem~\ref{Bkstruc}(b),  
when $p$ is odd, $R_p$ has a single period $2q$ orbit in $\Z_{2q}$.
Expanding the points in $\Z_{2q}$ to intervals in the circle
as in the proof of Theorem~\ref{Bkstruc}(c), we see that by varying 
$\mu$ the construction
generates  all the symbolic $p/q$-periodic
$2$-fold semi-monotone sets in $\Omega_2$.  

Now when $p$ is even, $R_p$ has a pair of period $q$ orbits.
When $\mu$ is odd, these generate different  periodic
orbits $B'(p, q, \mu)$. However, $\mu$ even corresponds to 
a pure parameter and so varying $\mu$ through the even
$\mu$ generates all the symbolic $p/q$-periodic
$2$-fold semi-monotone sets in $\Omega_2$.  
 
The next step is to use $B'(p, q, \mu)$ to compute its symbolic box
as in Corollary~\ref{symholecor} below. Finally, we take the inverse
 of the itinerary map
for the model map to get a box in the $\vc$ parameter.
 Because the map $f_m$ has uniform slope of
three in its positive slope region the formula for this inverse is
$\us\in\Sigma_2^+$,
\begin{equation}\label{iotainv}
	\iota_1^{-1}(\us) = \sum_{j=0}^\infty \frac{s_j}{3^{j+1}}.
\end{equation}

\section{Symbolic kfsm sets and the map $z\mapsto z^n$}\label{last}

Using the model map 
Theorem~\ref{boxthm} characterising ``physical'' kfsm sets can be directly
transformed into a characterization of symbolic kfsm sets.
For compact $\hZ\subset\Omega_k$ for $j = 1, \dots, k$ define 
$$
\hell_j(Z) = \sigma_k( \max\{\hZ\cap [2j -1]\})\ \text{and}  \  
\hr_j(Z) = \sigma_k( \min\{\hZ\cap [2j]\}) 
$$

Since $\iota_k$ is order preserving and onto for 
the model map we have
\begin{corollary}\label{symholecor} 
Assume $\hZ\subset \Omega_k$ is compact and 
shift invariant.
The following are equivalent
\begin{enumerate}
\item $\hZ$ is kfsm 
\item For $j = 1, \dots, k$, $\hell_j(\hZ) \leq  \hr_j(\hZ)$ 
with indices reduced mod $2k$
\end{enumerate}
\end{corollary} 
If $Z$ is recurrent 
we know that each $Z\in\cS_k(f)$ has $\iota_k(\hZ) = B_k(\omega, \vnu)$
for some allowable $(\omega, \vnu)$ which yields an indirect connection
between the interpolated semi-monotone maps and HM parameterization.

There is a well known connection between 
 the dynamics of $d_n:z\mapsto z^n$ and the full shift on $n$ symbols. 
 This yields a connection of the symbolic kfsm sets as described by this
corollary to invariant sets of the circle on which
the action of $d_n$ is semi-monotone, sometimes call \textit{circular
orbits}. 

\begin{figure}[htbp]
\begin{center}
\includegraphics[width=0.4\textwidth]{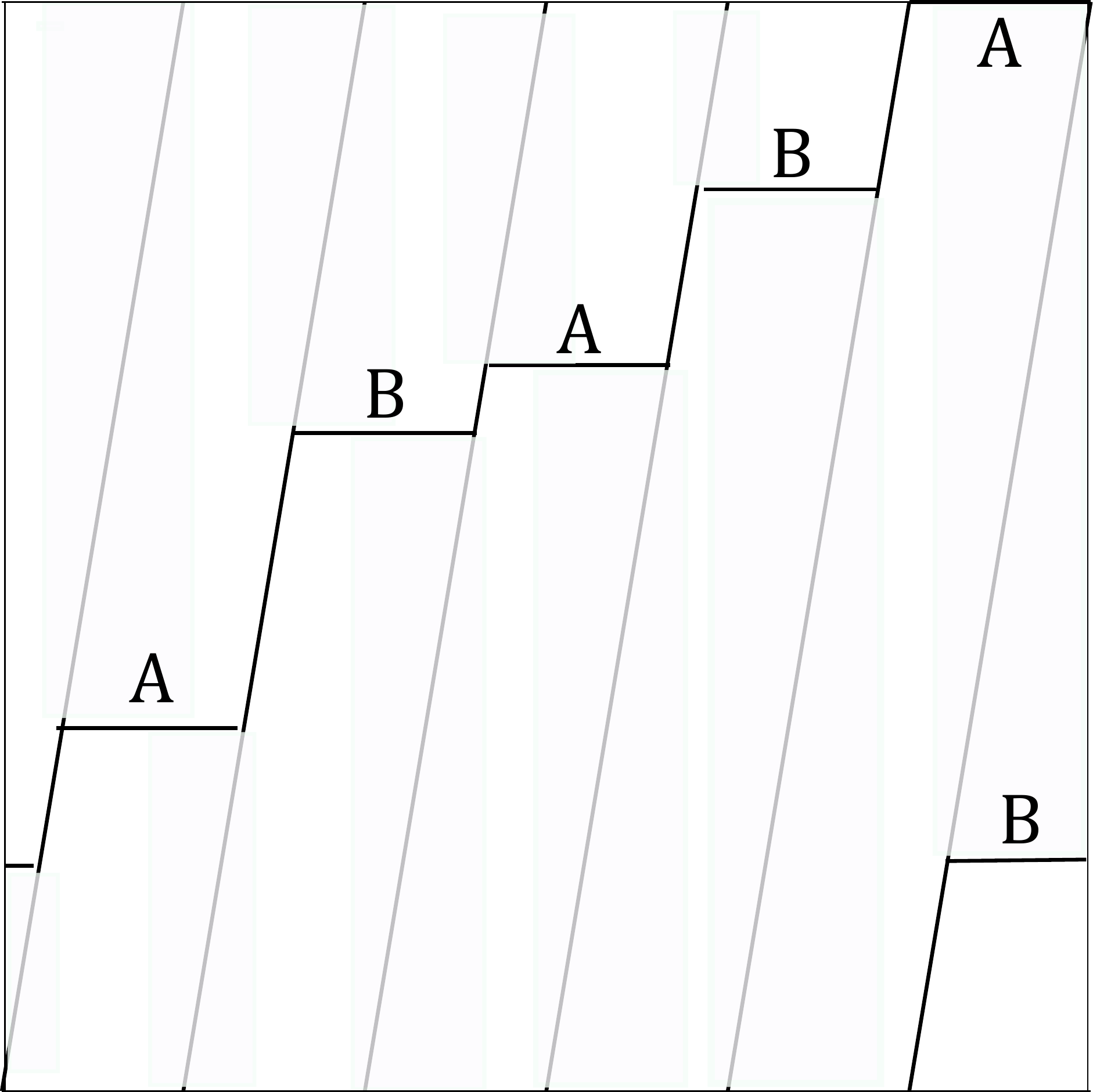}
\end{center}
\caption{The semi-monotone map corresponding to a symbolic $3$-fold
semi-monotone set interpolated into $z\mapsto z^6$ }
\label{complex}
\end{figure}

In Figure~\ref{complex} 
we show the conditions forced by Corollary~\ref{symholecor}
as flat spots in the graph of $d_n$ for $k=3$ and $n=6$. 
There are two classes of flat spots. Those in class A are forced by the
condition that $\hZ\subset \Omega_k$ and thus satisfies
\eqref{transitions}. These are the intervals of
width $1/6$, $[1/18, 2/9], [7/18, 5/9]$ and $[13/18, 8/9]$. These
conditions are satisfied by  all symbolic
kfsm sets in the corollary. The other three flat spots in class B are
determined by the conditions in part (b) of the corollary and
vary with the symbolic kfsm set. Note that the result of addding
all the flat spots is a degree one semi-monotone circle map as expected.
See figures in \cite{goldtress} and \cite{uabthesis}.

The figure also illustrates a clear difference between the kfsm sets
for bimodal circle maps and circular orbits for $d_n$-one. 
Specifically, the kfsm sets correspond
to a specific subclass of circular orbits for $d_{2k}$. 
On the other hand, there is 
clearly a tight relationship between the theories which needs
to be investigated. Perhaps  the degree reduction process described
in \cite{uab, uabthesis} would be a good place to start.

\bibliographystyle{amsplain}
\bibliography{circlerefs}

\end{document}

%% file: familydefs.tex
\def\cP{{\mathcal{P}}}

\def\cD{{\mathcal{D}}}
\def\cC{{\mathcal{C}}}
\def\cH{{\mathcal{H}}}

\def\cB{{\mathcal{B}}}
\def\cG{{\mathcal{G}}}

\def\cN{{\mathcal{N}}}
\def\cS{{\mathcal{B}}} %%note weirdness for notation change
\def\cO{{\mathcal{O}}}
\def\tcG{\tilde{\cG}}
\def\tcH{\tilde{\cH}}
\def\tZ{\tilde{Z}}

\def\R{{\mathbb R}}
\def\Z{{\mathbb Z}}
\def\N{{\mathbb N}}
\def\Q{{\mathbb Q}}

\def\A{{\mathbb A}}

\def\raw{\rightarrow}

\def\us{{\underline s}}

\def\ut{{\underline t}}

\def\uw{{\underline w}}

\def\ukappa{{\underline \kappa}}

\def\tx{{\tilde{x}}}

\def\tY{{\tilde{Y}}}

\def\tZ{{\tilde{Z}}}
\def\tf{\tilde{f}}
\def\tg{\tilde{g}}

\def\tH{\tilde{H}}

\def\th{{\tilde{h}}}

\DeclareMathOperator{\HD}{HD}

\DeclareMathOperator{\HM}{HM}
\DeclareMathOperator{\Boxx}{Box}

\def\bgamma{\overline{\gamma}}
\def\bHMk{\overline{\HM_k(g)}}
\def\bHM{\overline{\HM}}

\def\bpsi{\overline{\psi}}

\def\ilim{{\varprojlim}}

\def\hZ{{\hat{Z}}}
\def\hT{{\hat{T}}}

\def\hp{{\hat{p}}}
\def\hr{{\hat{r}}}

\def\hsigma{{\hat{\sigma}}}
\def\hLambda{{\hat{\Lambda}}}
\def\hmu{{\hat{\mu}}}
\def\hpi{{\hat{\pi}}}
\def\hcS{{\hat{\cS}}} % now \cB from above.
\def\hcB{{\hat{\cB}}} % now \cB from above.

\def\hcO{{\hat{\cO}}}
\def\hcC{{\hat{\cC}}}
\def\hcN{{\hat{\cN}}}
\def\hrho{{\hat{\rho}}}
\def\heta{{\hat{\eta}}}
\def\hgamma{{\hat{\gamma}}}
\def\hell{{\hat{\ell}}}

\DeclareMathOperator{\Cl}{Cl}
\DeclareMathOperator{\Intt}{Int}
\DeclareMathOperator{\Fr}{Fr}
\DeclareMathOperator{\spt}{spt}
\DeclareMathOperator{\Pure}{Pure}
\def\bPure{\overline{\Pure}}

\def\lrk{\langle \ukappa_0, \ukappa_1 \rangle}

\def\vc{\vec{c}}
\def\vu{\vec{u}}
\def\vv{\vec{v}}
\def\vL{\vec{L}}

\def\vnu{\vec{\nu}}
\def\vnup{\vec{\nu}^{\;\prime}}
\def\vdelta{\vec{\delta}}
\def\veta{\vec{\eta}}

\def\otz{{\Cl(o(z',\tg_k))}}
\def\oty{{\Cl(o(y',\tg_k))}}
\def\kludge{{\subsetneq}}

\def\xmin{{x_{min}}}
\def\xmax{{x_{max}}}
\def\zmin{{z_{min}}}
\def\zmax{{z_{max}}}

\def\One{\mathbbm{1}}

\def\mysim{\mkern-8mu\sim}
\def\mymod{\mkern-8mu\mod}